\documentclass[11pt]{amsart}
\usepackage{fullpage}
\usepackage{amssymb,amsmath,amsfonts,amsthm,mathrsfs,verbatim,moreverb}
\usepackage{color}

\numberwithin{equation}{section}

\renewcommand{\AA}{\mathbb A}

\newcommand{\CC}{\mathbb C}

\newcommand{\FF}{\mathbb F}

\newcommand{\PP}{\mathbb P}
\newcommand{\QQ}{\mathbb Q}
\newcommand{\RR}{\mathbb R}

\newcommand{\ZZ}{\mathbb Z}

\newcommand{\OO}{\mathcal O}
\newcommand{\calI}{\mathcal I}
\newcommand{\calF}{\mathcal F}

\newcommand{\calE}{\mathcal E} \newcommand{\scrE}{\mathscr E}
\newcommand{\calG}{\mathcal G}
\newcommand{\calH}{\mathcal H}

\newcommand{\calS}{\mathcal S}

\newcommand{\calX}{\mathcal X}

\newcommand{\calD}{\mathcal D}

\newcommand{\p}{\mathfrak p}

\newcommand{\scrF}{\mathscr F}
\newcommand{\scrM}{\mathscr M}

\newcommand{\ang}[1]{ \langle #1 \rangle  }

\def\un{{\operatorname{un}}}
\def\Spec{\operatorname{Spec}} 
\def\Gal{\operatorname{Gal}}
 
\def \GL {\operatorname{GL}}    
\def \PGL {\operatorname{PGL}}
\def \SL {\operatorname{SL}}

\def\Aut{\operatorname{Aut}} 

\def\Frob{\operatorname{Frob}}

\def\tr{\operatorname{tr}}

\newcommand{\defi}[1]{\textsf{#1}} % for defined terms

\newcommand\bbar[1]{\overline{#1}}
\newcommand{\Qbar}{{\overline{\mathbb Q}}} 
\newcommand{\Kbar}{\bbar{K}}

\newcommand\legendre[2]{{#1\overwithdelims () #2}}

\newtheorem{thm}{Theorem}[section]
\newtheorem{lemma}[thm]{Lemma}

\newtheorem{prop}[thm]{Proposition}

\newtheorem{conj}[thm]{Conjecture}

\theoremstyle{definition}
\newtheorem{definition}[thm]{Definition}

\theoremstyle{remark}
\newtheorem{remark}[thm]{Remark}

\newenvironment{romanenum}{\hfill \begin{enumerate} }{\end{enumerate}}
\newenvironment{alphenum}{\hfill \begin{enumerate} }{\end{enumerate}}

\definecolor{webcolor}{rgb}{0.8,0,0.2}
\definecolor{webbrown}{rgb}{.6,0,0}
\usepackage[
        colorlinks,
        linkcolor=webbrown,  filecolor=webcolor,  citecolor=webbrown, 
        backref,
        pdfauthor={David Zywina}, 
       pdftitle={On the possible images of the mod ell representations associated to elliptic curves over Q}
]{hyperref}
\usepackage[alphabetic,backrefs,lite]{amsrefs} % for bibliography

\begin{document}
\title[]{On the possible images of the mod $\ell$ representations associated to elliptic curves over $\QQ$}

\author{David J. Zywina}
\address{Department of Mathematics, Cornell University, Ithaca, NY 14853, USA}
\email{zywina@math.cornell.edu}
\urladdr{http://www.math.cornell.edu/~zywina}

%%11G05 Elliptic curves over global Þelds
%%11F80 Galois representations 

\begin{abstract}
Consider a non-CM elliptic curve $E$ defined over $\QQ$.  For each prime $\ell$, there is a representation $\rho_{E,\ell}\colon \Gal(\Qbar/\QQ) \to \GL_2(\FF_\ell)$ that describes the Galois action on the $\ell$-torsion points of $E$.   A famous theorem of Serre says that $\rho_{E,\ell}$ is surjective for all large enough $\ell$.       We will describe all known, and conjecturally all, pairs $(E,\ell)$ such that $\rho_{E,\ell}$ is not surjective.   Together with another paper, this produces an algorithm that given an elliptic curve $E/\QQ$, outputs the set of such \emph{exceptional primes} $\ell$ and describes all the groups $\rho_{E,\ell}(\Gal(\Qbar/\QQ))$ up to conjugacy.    Much of the paper is dedicated to computing various modular curves of genus $0$ with their morphisms to the $j$-line.
\end{abstract}

\maketitle

\section{Possible images} \label{S:classification}

Consider an elliptic curve $E$ defined over $\QQ$.   For each prime $\ell$, let $E[\ell]$ be the $\ell$-torsion subgroup of $E(\Qbar)$, where $\Qbar$ is a fixed algebraic closure of $\QQ$.   The group $E[\ell]$ is a free $\FF_\ell$-module of rank $2$ and there is a natural action of the absolute Galois group $\Gal_\QQ:= \Gal(\Qbar/\QQ)$ on $E[\ell]$ which respects the group structure.   After choosing a basis for $E[\ell]$, this action can be expressed in terms of a Galois representation
\[
\rho_{E,\ell} \colon \Gal_\QQ \to \GL_2(\FF_\ell);
\]
its image $\rho_{E,\ell}(\Gal_\QQ)$ is uniquely determined up to conjugacy in $\GL_2(\FF_\ell)$.  A renowned theorem of Serre \cite{MR0387283} says that $\rho_{E,\ell}$ is surjective for all but finitely many $\ell$ when $E$ is non-CM.

In this paper, we shall describe all known (and conjecturally all) proper subgroups of $\GL_2(\FF_\ell)$ that occur as the image of such a representation $\rho_{E,\ell}$.    Applying our classification with earlier work, we will obtain an algorithm to determine the set $\calS$ of primes $\ell$ for which $\rho_{E,\ell}$ is not surjective and also compute $\rho_{E,\ell}(\Gal_\QQ)$ for each $\ell\in \calS$. \\

Before stating our classification in \S\S\ref{SS:applicable 2}--\ref{SS:applicable 17}, let us make some comments.   We will consider each prime $\ell$ separately.   For simplicity, assume that the $j$-invariant $j_E\in \QQ$ of $E/\QQ$ is neither $0$ nor $1728$.  Our first step in determining $\rho_{E,\ell}(\Gal_\QQ)$ is to compute the group 
\[
G:=\pm \rho_{E,\ell}(\Gal_\QQ),
\]
 i.e., the group generated by $-I$ and $\rho_{E,\ell}(\Gal_\QQ)$.    The benefit of studying $G$, up to conjugacy in $\GL_2(\FF_\ell)$, is that it does not change if $E$ is replaced by a quadratic twist.  Moreover, if $E'/\QQ$ is a quadratic twist of $E/\QQ$, then after choosing appropriate bases, we will have $\rho_{E',\ell} = \chi \cdot \rho_{E,\ell}$ for some quadratic character $\chi\colon \Gal_\QQ \to \{\pm 1\}$.    Since $j_E\notin \{0,1728\}$, all twists of $E$ are quadratic twists and hence $G$, up to conjugacy, depends only on the value $j_E$.     The character $\det\circ \rho_{E,\ell}\colon \Gal_\QQ \to \FF_\ell^\times$ describes the Galois action on the $\ell$-th roots of unity, so $\det(G)=\FF_\ell^\times$.

For a subgroup $G$ of $\GL_2(\FF_\ell)$ with $\det(G)=\FF_\ell^\times$ and $-I\in G$, we can associate a modular curve $X_{G}$; it is a smooth, projective and geometrically irreducible curve defined over $\QQ$.   It comes with a natural morphism 
\[
\pi_G\colon X_{G} \to \Spec \QQ[j] \cup \{\infty\}=:\PP_\QQ^1
\]
such that for an elliptic curve $E/\QQ$ with $j_E\notin \{0,1728\}$, the group $\rho_{E,\ell}(\Gal_\QQ)$ is conjugate in $\GL_2(\FF_\ell)$ to subgroup of $G$ if and only if the $j_E=\pi_G(P)$ for some rational point  $P\in X_{G}(\QQ)$.    

Much of this paper is dedicated to describing those modular curves $X_G$ of genus $0$ with $X_G(\QQ)\neq \emptyset$.   Such modular curves are isomorphic to the projective line and their function field is of form $\QQ(h)$ for some modular function $h$ of level $\ell$.   Giving the morphism $\pi_G$ is then equivalent to expressing the modular $j$-invariant in the form $J(h)$ for a unique rational function $J(t)\in \QQ(t)$.

Once we have determined $G$, we know that $\rho_{E,\ell}(\Gal_\QQ)$ will either be the full group $G$ or equal to an index $2$ subgroup $H$ of $G$ for which $-I\notin H$.  For each such $H$, it is then a matter of determining whether the quadratic character $\Gal_\QQ \xrightarrow{\rho_{E,\ell}} G \to G/H\cong\{\pm 1\}$ is trivial or not. 
\\

We will first focus on the general case of non-CM elliptic curves over $\QQ$.   In \S\ref{SS:CM}, we will give a complete description of the groups $\rho_{E,\ell}(\Gal_\QQ)$ when $E/\QQ$ has complex multiplication.
\\

\noindent{\textit{Notation.} }
We now define some specific subgroups of $\GL_2(\FF_\ell)$ for an odd prime $\ell$.  Let  $C_s(\ell)$ be the subgroup of diagonal matrices.   Let $\epsilon=-1$ if $\ell\equiv 3\pmod{4}$ and otherwise let $\epsilon \geq 2$ be the smallest integer which is not a quadratic residue modulo $\ell$.   Let $C_{ns}(\ell)$ be the subgroup consisting of matrices of the form $\left(\begin{smallmatrix}a & b\epsilon \\b & a \end{smallmatrix}\right)$ with $(a,b) \in \FF_\ell^2-\{(0,0)\}$.  Let $N_s(\ell)$ and $N_{ns}(\ell)$ be the normalizers of $C_{s}(\ell)$ and $C_{ns}(\ell)$, respectively, in $\GL_2(\FF_\ell)$.   We have $[N_s(\ell):C_s(\ell)]=2$ and the non-identity coset of $C_s(\ell)$ in $N_s(\ell)$ is represented by $\left(\begin{smallmatrix}0 & 1 \\1 & 0 \end{smallmatrix}\right)$.  We have $[N_{ns}(\ell):C_{ns}(\ell)]=2$ and the non-identity coset of $C_{ns}(\ell)$ in $N_{ns}(\ell)$ is represented by $\left(\begin{smallmatrix}1 &0  \\0 & -1 \end{smallmatrix}\right)$.  Let $B(\ell)$ be the subgroup of upper triangular matrices in $\GL_2(\FF_\ell)$.

\subsection{\underline{$\ell=2$}}  \label{SS:applicable 2}

Up to conjugacy, there are three proper subgroups of $\GL_2(\FF_2)$:
\[
G_1=\{ I\},   \quad\quad G_2=\{ I, \left(\begin{smallmatrix}1 & 1 \\0 & 1 \end{smallmatrix}\right)\},  \quad\quad  G_3=\{I,  \left(\begin{smallmatrix}1 & 1 \\1 & 0 \end{smallmatrix}\right),  \left(\begin{smallmatrix}0 & 1 \\1 & 1 \end{smallmatrix}\right)\}.
\]
For $i=1,2$ and $3$, the index $[\GL_2(\FF_2): G_i]$ is $6$, $3$ and $2$, respectively.   Define the rational functions
\[
J_1(t)=256 \frac{(t^2+t+1)^3}{t^2(t+1)^2}, 
\quad\quad J_2(t) = 256\frac{(t+1)^3}{t},\quad\quad J_3(t)=t^2+1728.
\]

\begin{thm} \label{T:main2}
Let $E$ be a non-CM elliptic curve over $\QQ$.  Then $\rho_{E,2}(\Gal_\QQ)$ is conjugate in $\GL_2(\FF_2)$ to a subgroup of $G_i$ if and only if $j_E$ is of the form $J_i(t)$ for some $t\in \QQ$.
\end{thm}

\subsection{\underline{$\ell=3$}}  \label{SS:applicable 3}
Define the following subgroups of $\GL_2(\FF_3)$:
\begin{itemize}
\item 
Let $G_1$ be the group $C_s(3)$.  
\item
Let $G_2$ be the group $N_s(3)$. 
\item 
Let $G_3$ be the group $B(3)$.    
\item
Let $G_4$ be the group $N_{ns}(3)$.  
\item
Let $H_{1,1}$ be the subgroup consisting of the matrices of the form $\left(\begin{smallmatrix}1 & 0 \\0 & * \end{smallmatrix}\right)$.   
\item
Let $H_{3,1}$ be the subgroup consisting of the matrices of the form $\left(\begin{smallmatrix}1 & * \\0 & * \end{smallmatrix}\right)$.   
\item
Let $H_{3,2}$ be the subgroup consisting of the matrices of the form $\left(\begin{smallmatrix} * & * \\0 & 1 \end{smallmatrix}\right)$.   
\end{itemize}
The index in $\GL_2(\FF_3)$ of the above subgroups are $12$, $6$, $4$, $3$, $24$, $8$ and $8$, respectively.  Each of the groups $G_i$ contain $-I$.   The groups $H_{i,j}$ do not contain $-I$ and we have $G_i = \pm H_{i,j}$.

Define the rational functions:
\begin{align*}
J_1(t) &= 27\frac{(t+1)^3(t+3)^3(t^2+3)^3}{t^3(t^2+3t+3)^3},\,\, 
J_2(t) = 27\frac{(t+1)^3(t-3)^3}{t^3},\,\, 
J_3(t) =  27\frac{(t+1)(t+9)^3}{t^3}, \,\, 
J_4(t)=t^3.
\end{align*}
For $t\in \QQ -\{0\}$, let $\calE_{1,t}$ be the elliptic curve over $\QQ$ defined by Weierstrass equation
\begin{align*}
y^2&= x^3 -3(t+1)(t+3)(t^2+3) x -2(t^2-3)(t^4+6t^3+18t^2+18t+9). 
\end{align*}
For $t\in \QQ -\{0,-1\}$, let $\calE_{3,t}$ be the elliptic curve over $\QQ$ defined by Weierstrass equation
\begin{align*}
y^2&= x^3 -3(t+1)^3(t+9)x -2(t+1)^4(t^2-18t-27).
\end{align*}
The $j$-invariant of $\calE_{i,t}$ is $J_i(t)$.  

\begin{thm} \label{T:main3}
Let $E$ be a non-CM elliptic curve over $\QQ$.
\begin{romanenum}
\item \label{T:main3 a}
If $\rho_{E,3}$ is not surjective, then $\rho_{E,3}(\Gal_\QQ)$ is conjugate in $\GL_2(\FF_3)$ to one of the groups $G_i$ or $H_{i,j}$.
\item \label{T:main3 b}
The group $\rho_{E,3}(\Gal_\QQ)$ is conjugate to a subgroup of $G_i$ if and only if $j_E$ is of the form $J_i(t)$ for some $t\in \QQ$.
\item \label{T:main3 c}
Suppose that $\pm \rho_{E,3}(\Gal_\QQ)$ is conjugate to $G_1$.  Fix an element $t\in \QQ$ such that $J_1(t)=j_E$.  The group $\rho_{E,3}(\Gal_\QQ)$ is conjugate to $H_{1,1}$ if and only if $E$ is isomorphic to $\calE_{1,t}$ or the quadratic twist of $\calE_{1,t}$ by $-3$.
\item \label{T:main3 d}
Suppose that $\pm \rho_{E,3}(\Gal_\QQ)$ is conjugate to $G_3$.  Fix an element $t\in \QQ$ such that $J_3(t)=j_E$. 

\noindent The group $\rho_{E,3}(\Gal_\QQ)$ is conjugate to $H_{3,1}$ if and only if $E$ is isomorphic to $\calE_{3,t}$.

\noindent The group $\rho_{E,3}(\Gal_\QQ)$ is conjugate to $H_{3,2}$ if and only if $E$ is isomorphic to the quadratic twist of $\calE_{3,t}$ by $-3$.

\end{romanenum}
\end{thm}

\begin{remark}
\begin{romanenum}
\item
Let us briefly explain how Theorem~\ref{T:main3} can be used to compute $\rho_{E,3}(\Gal_\QQ)$; similar remarks will hold for the remaining primes (the case $\ell=2$ is particularly simple since $-I=I$).   If $j_E$ is not of the form $J_i(t)$ for any $i\in \{1,2,3,4\}$ and $t\in \QQ$, then $\rho_{E,3}(\Gal_\QQ)=\GL_2(\FF_3)$.  To check if $j_E$ is of the form $J(t)$, clear denominators in $J(t)-j_E$ to obtain a polynomial in $t$  which one can then determine whether it has rational roots or not.

 So assume that $\rho_{E,3}$ is not surjective, and let $i$ be the smallest value in $\{1,2,3,4\}$ for which $j_E=J_i(t)$ for some $t\in \QQ$.   By Theorem~\ref{T:main3}(\ref{T:main3 a}) and (\ref{T:main3 b}), we deduce that $\pm \rho_{E,3}(\Gal_\QQ)$ is conjugate to $G_i$; note that the groups $G_i$ are ordered by decreasing index in $\GL_2(\FF_3)$.  After possibly conjugating $\rho_{E,3}$, we may assume that $\pm \rho_{E,3}(\Gal_\QQ)=G_i$.  If $\rho_{E,3}(\Gal_\QQ)$ does not equal $G_i$, then it is equal to one of the subgroups $H_{i,j}$ and parts (\ref{T:main3 c}) and (\ref{T:main3 d}) give necessary and sufficient conditions to check this.
\item
Our rational functions $J_i(t)$ are certainly not unique.  In particular, any function of the form $J_i((at+b)/(ct+d))$ will work with fixed $a,b,c,d\in \QQ$ satisfying $ad-bc\neq 0$ (though in general, one needs to also consider the value of $J_i(t)$ at $\infty$).    Given $J_i(t)$, our equations for $\calE_{i,t}$ were produced by an algorithm that we will later describe; there are other possibly simpler choices.
\end{romanenum}
\end{remark}

\subsection{\underline{$\ell=5$}}  \label{SS:applicable 5}

Define the following subgroups of $\GL_2(\FF_5)$:
\begin{itemize}
\item
Let $G_1$ be the subgroup  consisting of the matrices of the form $\pm \left(\begin{smallmatrix}1 & 0 \\0 & * \end{smallmatrix}\right)$.   
\item
Let $G_2$ be the group $C_s(5)$. 
\item
Let $G_3$ be the unique subgroup of $N_{ns}(5)$ of index $3$; it is generated by $\left(\begin{smallmatrix}2 & 0 \\0 & 2 \end{smallmatrix}\right)$, $\left(\begin{smallmatrix}1 & 0 \\0 & -1 \end{smallmatrix}\right)$ and $\left(\begin{smallmatrix}0 & 6 \\3 & 0 \end{smallmatrix}\right)$.
\item
Let $G_4$ be the group $N_{s}(5)$.   
\item
Let $G_5$ be the subgroup consisting of the matrices of the form $\pm \left(\begin{smallmatrix}* & * \\0 & 1 \end{smallmatrix}\right)$.   
\item
Let $G_6$ be the subgroup consisting of the matrices of the form $\pm \left(\begin{smallmatrix}1 & * \\0 & * \end{smallmatrix}\right)$.  
\item
Let $G_7$ be the group $N_{ns}(5)$.  
\item
Let $G_8$ be the group $B(5)$.  
\item
Let $G_9$ be the unique maximal subgroup of $\GL_2(\FF_5)$ which contains $N_s(5)$; it is generated by $\left(\begin{smallmatrix}2 & 0 \\0 & 1 \end{smallmatrix}\right)$, $\left(\begin{smallmatrix}1 & 0 \\0 & 2 \end{smallmatrix}\right)$, $\left(\begin{smallmatrix}0 & -1 \\1 & 0 \end{smallmatrix}\right)$ and $\left(\begin{smallmatrix}1 & 1 \\1 & -1 \end{smallmatrix}\right)$.
\item
Let $H_{1,1}$ be the subgroup consisting of the matrices of the form $\left(\begin{smallmatrix}1 & 0 \\0 & * \end{smallmatrix}\right)$.   
\item
Let $H_{1,2}$ be the subgroup consisting of the matrices of the form $\left(\begin{smallmatrix}a^2 & 0 \\0 & a \end{smallmatrix}\right)$.    
\item
Let $H_{5,1}$ be the subgroup  consisting of the matrices of the form $\left(\begin{smallmatrix}* & * \\0 & 1 \end{smallmatrix}\right)$.   
\item
Let $H_{5,2}$ be the subgroup consisting of the matrices of the form $\left(\begin{smallmatrix}a & * \\0 & a^2 \end{smallmatrix}\right)$.   
\item
Let $H_{6,1}$ be the subgroup consisting of the matrices of the form $\left(\begin{smallmatrix}1 & * \\0 & * \end{smallmatrix}\right)$.   
\item
Let $H_{6,2}$ be the subgroup consisting of the matrices of the form $\left(\begin{smallmatrix}a^2 & * \\0 & a \end{smallmatrix}\right)$.  
\end{itemize}
The index in $\GL_2(\FF_5)$ of the above subgroups are $60$, $30$, $30$, $15$, $12$, $12$, $10$, $6$, $5$, $120$, $120$, $24$, $24$, $24$ and $24$, respectively.  Each of the groups $G_i$ contain $-I$.   The groups $H_{i,j}$ do not contain $-I$ and we have $G_i = \pm H_{i,j}$.  

Define the rational functions:
\begin{align*}
J_1(t) &= \frac{(t^{20}+228t^{15}+494t^{10}-228t^5+1)^3}{t^5(t^{10}-11t^5-1)^5}\\
J_2(t) &= \frac{(t^2 + 5t + 5)^3(t^4 + 5t^2 + 25)^3(t^4 + 5t^3 + 20t^2 + 25t + 25)^3}{t^5(t^4 + 5t^3 + 15t^2 + 25t + 25)^5}\\
J_3(t) &= \frac{5^4 t^3 (t^2+5t+10)^3 (2t^2+5t+5)^3 (4t^4+30t^3+95t^2+150t+100)^3}{(t^2+5t+5)^5(t^4+5t^3+15t^2+25t+25)^5}\\
J_4(t) &= \frac{(t + 5)^3 (t^2 - 5)^3(t^2 + 5t + 10)^3}{(t^2 + 5t + 5)^5}\\
J_5(t) &= \frac{(t^4+228t^3+494t^2-228t+1)^3}{t(t^2-11t-1)^5}\\
J_6(t) &= \frac{(t^4-12t^3+14t^2+12t+1)^3}{t^5(t^2-11t-1)}\\
J_7(t) &= \frac{5^3 (t+1)(2t+1)^3(2t^2-3t+3)^3}{(t^2+t-1)^5}\\
J_8(t) &= \frac{5^2(t^2+10t+5)^3}{t^5}  \\
J_9(t) &=t^3(t^2 + 5t + 40)
\end{align*}

\noindent For $t\in \QQ-\{0\}$, let $\calE_{1,t}$ be the elliptic curve over $\QQ$ defined by the Weierstrass equation
\begin{align*}
y^2=x^3 &-27(t^{20} + 228t^{15} + 494t^{10} - 228t^5 + 1)x \\
&+54(t^{30} - 522t^{25} - 10005t^{20} - 10005t^{10} + 522t^5 + 1).
\end{align*} 
\noindent
For $t\in \QQ-\{0\}$, let $\calE_{5,t}$ be the elliptic curve over $\QQ$ defined by the Weierstrass equation
\begin{align*}
y^2=x^3-27(t^4 + 228t^3 + 494t^2 - 228t + 1)x+54(t^6 - 522t^5 - 10005t^4 - 10005t^2 + 522t + 1).
\end{align*}
For $t\in \QQ-\{0\}$, let $\calE_{6,t}$ be the elliptic curve over $\QQ$ defined by the Weierstrass equation
\begin{align*}
y^2 = x^3-27(t^4 - 12t^3 + 14t^2 + 12t + 1)x + 54(t^6-18t^5+75t^4+75t^2+18t+1)
\end{align*}
The $j$-invariant of $\calE_{i,t}$ is $J_i(t)$.  

\begin{thm}  \label{T:main5}
Let $E$ be a non-CM elliptic curve over $\QQ$.
\begin{romanenum}
\item \label{T:main5 a}
If $\rho_{E,5}$ is not surjective, then $\rho_{E,5}(\Gal_\QQ)$ is conjugate in $\GL_2(\FF_5)$ to one of the groups $G_i$ or $H_{i,j}$.
\item  \label{T:main5 b}
The group $\rho_{E,5}(\Gal_\QQ)$ is conjugate to a subgroup of $G_i$ if and only if $j_E$ is of the form $J_i(t)$ for some $t\in \QQ$.
\item \label{T:main5 c}
Suppose that $\pm \rho_{E,5}(\Gal_\QQ)$ is conjugate to $G_i$ with $i\in \{1,5,6\}$.  Fix an element $t\in \QQ$ such that $J_i(t)=j_E$.  

\noindent The group $\rho_{E,5}(\Gal_\QQ)$ is conjugate to $H_{i,1}$ if and only if $E$ is isomorphic to $\calE_{i,t}$.

\noindent The group $\rho_{E,5}(\Gal_\QQ)$ is conjugate to $H_{i,2}$ if and only if $E$ is isomorphic to the quadratic twist of $\calE_{i,t}$ by $5$.
\end{romanenum}
\end{thm}

\subsection{\underline{$\ell=7$}}  \label{SS:applicable 7}

Define the follow subgroups of $\GL_2(\FF_7)$:
\begin{itemize}
\item
Let $G_1$ be the subgroup of $N_s(7)$ consisting of elements of $C_s(7)$ with square determinant and elements of $N_s(7)-C_s(7)$ with non-square determinant; it is generated by $\left(\begin{smallmatrix}2 & 0 \\0 & 4 \end{smallmatrix}\right)$, $\left(\begin{smallmatrix}0 & 2 \\1 & 0 \end{smallmatrix}\right)$ and $\left(\begin{smallmatrix}-1 & 0 \\0 & -1 \end{smallmatrix}\right)$.
\item
Let $G_2$ be the group $N_s(7)$.
\item
Let $G_3$ be the subgroup  consisting of matrices of the form $\pm \left(\begin{smallmatrix}1 & * \\0 & * \end{smallmatrix}\right)$.
\item 
Let $G_4$ be the subgroup consisting of matrices of the form $\pm \left(\begin{smallmatrix}* & * \\0 & 1 \end{smallmatrix}\right)$.
\item 
Let $G_5$ be the subgroup consisting of matrices of the form $\left(\begin{smallmatrix}a & * \\0 & \pm a \end{smallmatrix}\right)$. 
\item 
Let $G_6$ be the group $N_{ns}(7)$.  
\item
Let $G_7$ be the group $B(7)$.  
\item
Let $H_{1,1}$ be the subgroup generated by $\left(\begin{smallmatrix}2 & 0 \\0 & 4 \end{smallmatrix}\right)$ and $\left(\begin{smallmatrix}0 & 2 \\1 & 0 \end{smallmatrix}\right)$.   
\item
Let $H_{3,1}$ be the subgroup consisting of the matrices of the form $\left(\begin{smallmatrix}1 & * \\0 & * \end{smallmatrix}\right)$.   
\item
Let $H_{3,2}$ be the subgroup consisting of the matrices of the form $\left(\begin{smallmatrix}\pm 1 & * \\0 & a^2 \end{smallmatrix}\right)$.   
\item
Let $H_{4,1}$ be the subgroup consisting of the matrices of the form $\left(\begin{smallmatrix}* & * \\0 & 1 \end{smallmatrix}\right)$.   
\item
Let $H_{4,2}$ be the subgroup consisting of the matrices of the form $\left(\begin{smallmatrix} a^2 & * \\0 & \pm 1 \end{smallmatrix}\right)$.   
\item
Let $H_{5,1}$ be the subgroup consisting of the matrices of the form $\left(\begin{smallmatrix}\pm a^2 & * \\0 & a^2 \end{smallmatrix}\right)$.    
\item
Let $H_{5,2}$ be the subgroup consisting of the matrices of the form $\left(\begin{smallmatrix}  a^2 & * \\0 & \pm a^2 \end{smallmatrix}\right)$.   
\item
Let $H_{7,1}$ be the subgroup consisting of the matrices of the form $\left(\begin{smallmatrix} * & * \\0 &  a^2 \end{smallmatrix}\right)$.   
\item
Let $H_{7,2}$ be the subgroup consisting of the matrices of the form $\left(\begin{smallmatrix}  a^2 & * \\0 & * \end{smallmatrix}\right)$.   
\end{itemize}
The index in $\GL_2(\FF_7)$ of the above subgroups are $56$, $28$, $24$, $24$, $24$, $21$, $8$, $112$, $48$, $48$, $48$, $48$, $48$, $48$, $16$ and $16$, respectively.  Each of the groups $G_i$ contain $-I$.   The groups $H_{i,j}$ do not contain $-I$ and we have $G_i = \pm H_{i,j}$.

Define the rational functions
\begin{align*}
J_1(t) &= 3^3\cdot 5\cdot 7^5/2^7	\\
J_2(t)&= \frac{t(t + 1)^3(t^2 - 5t + 1)^3(t^2 - 5t + 8)^3(t^4 - 5t^3 + 8t^2 - 7t + 7)^3 }{(t^3 - 4t^2 + 3t + 1)^7}\\
J_3(t) & = \frac{(t^2 - t + 1)^3(t^6 - 11t^5 + 30t^4 - 15t^3 - 10t^2 + 5t + 1)^3}{(t - 1)^7 t^7(t^3 - 8t^2 + 5t + 1)}\\
J_4(t) & = \frac{(t^2 - t + 1)^3(t^6 + 229t^5 + 270t^4 - 1695t^3 + 1430t^2 - 235t + 1)^3}{(t - 1)t(t^3 - 8t^2 + 5t + 1)^7}\\
J_5(t) & = -\frac{(t^2-3t-3)^3(t^2-t+1)^3(3t^2-9t+5)^3(5t^2-t-1)^3}{(t^3-2t^2-t+1) (t^3-t^2-2t+1)^7} \\
J_6(t) &= \frac{64 t^3(t^2+7)^3(t^2-7t+14)^3(5t^2-14t-7)^3}{(t^3-7t^2+7t+7)^7}\\
J_7(t)&=\frac{(t^2+245t+2401)^3(t^2+13t+49)}{t^7}     
\end{align*}

\noindent Let $\calE_1$ be the elliptic curve over $\QQ$ defined by the Weierstrass equation
\[
y^2= x^3 -5^3  7^3 x - 5^4  7^2 106
\]
\noindent For $t\in \QQ-\{0,1\}$, let $\calE_{3,t}$ be the elliptic curve over $\QQ$ defined by the Weierstrass equation
\begin{align*}
y^2=x^3 &- 27(t^2 - t + 1)(t^6 - 11t^5 + 30t^4 - 15t^3 - 10t^2 + 5t + 1) x\\
+&54(t^{12} - 18t^{11} + 117t^{10} - 354t^9 + 570t^8 - 486t^7 \\
&\quad\quad\quad\quad+ 273t^6 - 222t^5  + 174t^4 - 46t^3 - 15t^2 + 6t + 1).
\end{align*}
For $t\in \QQ-\{0,1\}$, let $\calE_{4,t}$ be the elliptic curve over $\QQ$ defined by the Weierstrass equation
\begin{align*}
y^2=x^3 &-27(t^2-t+1)(t^6+229t^5+270t^4-1695t^3+1430t^2-235t+1) x \\ 
+&54(t^{12}-522 t^{11}-8955 t^{10}+37950 t^9-70998 t^8+131562 t^7\\
&\quad \quad -253239 t^6+316290 t^5-218058 t^4+80090 t^3-14631 t^2+510 t+1).
\end{align*}
For $t\in \QQ$, let $\calE_{5,t}$ be the elliptic curve over $\QQ$ defined by the Weierstrass equation
\begin{align*}
y^2 = &\,\,x^3 -   27\cdot 7 (t^2 - 3t - 3)(t^2 - t + 1)(3t^2 - 9t + 5)(5t^2 - t - 1) x \\
&-54\cdot 7^2  (t^4 - 6t^3 + 17t^2 - 24t + 9)(3t^4 - 4t^3 - 5t^2 - 2t - 1)(9t^4 - 12t^3 - t^2 + 8t - 3).
\end{align*}
For $t\in \QQ-\{0\}$, let $\calE_{7,t}$ be the elliptic curve over $\QQ$ defined by the Weierstrass equation
\begin{align*}
  y^2=x^3 &- 27(t^2 + 13t + 49)^3(t^2 + 245t + 2401)x\\ 
  & +54(t^2 + 13t + 49)^4(t^4 - 490t^3 - 21609t^2 - 235298t - 823543).
\end{align*}
The $j$-invariant of $\calE_{i,t}$ is $J_i(t)$.

\begin{thm} \label{T:main7}
Let $E$ be a non-CM elliptic curve over $\QQ$.
\begin{romanenum}  
\item \label{T:main7 i}
If $\rho_{E,7}$ is not surjective, then $\rho_{E,7}(\Gal_\QQ)$ is conjugate in $\GL_2(\FF_7)$ to one of the groups $G_i$ or $H_{i,j}$.
\item \label{T:main7 ii}
The group $\rho_{E,7}(\Gal_\QQ)$ is conjugate to a subgroup of $G_i$ if and only if $j_E$ is of the form $J_i(t)$ for some $t\in \QQ$.
\item \label{T:main7 iii}
The group $\rho_{E,7}(\Gal_\QQ)$ is conjugate to $H_{1,1}$ if and only if $E/\QQ$ is isomorphic to $\calE_1$ or to the quadratic twist of $\calE_1$ by $-7$.
\item \label{T:main7 iv}
Suppose that $\pm \rho_{E,7}(\Gal_\QQ)$ is conjugate to $G_i$ with $i\in \{3,4,5,7\}$.  Fix an element $t\in \QQ$ such that $J_i(t)=j_E$.  

\noindent The group $\rho_{E,7}(\Gal_\QQ)$ is conjugate to $H_{i,1}$ if and only if $E$ is isomorphic to $\calE_{i,t}$.

\noindent The group $\rho_{E,7}(\Gal_\QQ)$ is conjugate to $H_{i,2}$ if and only if $E$ is isomorphic to the quadratic twist of $\calE_{i,t}$ by $-7$. 
\end{romanenum}
\end{thm}  

\subsection{\underline{$\ell=11$}}  \label{SS:applicable 11}

\begin{itemize}
\item
Let $G_1$ be the subgroup generated by  $\pm \left(\begin{smallmatrix}1 & 1 \\0 & 1 \end{smallmatrix}\right)$ and $\left(\begin{smallmatrix}4 & 0 \\0 & 6 \end{smallmatrix}\right)$.

\item
Let $G_2$ be the subgroup generated by  $\pm \left(\begin{smallmatrix}1 & 1 \\0 & 1 \end{smallmatrix}\right)$ and $\left(\begin{smallmatrix}5 & 0 \\0 & 7 \end{smallmatrix}\right)$.

\item
Let $G_3$ be the group $N_{ns}(11)$.
\item
Let $H_{1,1}$ be the subgroup generated by  $\left(\begin{smallmatrix}1 & 1 \\0 & 1 \end{smallmatrix}\right)$ and $\left(\begin{smallmatrix}4 & 0 \\0 & 6 \end{smallmatrix}\right)$.
\item
Let $H_{1,2}$ be the subgroup generated by  $\left(\begin{smallmatrix}1 & 1 \\0 & 1 \end{smallmatrix}\right)$ and $\left(\begin{smallmatrix}7 & 0 \\0 & 5 \end{smallmatrix}\right)$.
\item
Let $H_{2,1}$ be the subgroup generated by  $\left(\begin{smallmatrix}1 & 1 \\0 & 1 \end{smallmatrix}\right)$ and $\left(\begin{smallmatrix}5 & 0 \\0 & 7 \end{smallmatrix}\right)$.
\item
Let $H_{2,2}$ be the subgroup generated by  $\left(\begin{smallmatrix}1 & 1 \\0 & 1 \end{smallmatrix}\right)$ and $\left(\begin{smallmatrix}6 & 0 \\0 & 4 \end{smallmatrix}\right)$.
\end{itemize}
The index in $\GL_2(\FF_{11})$ of the above subgroups are $60$, $60$, $55$, $110$, $120$, $120$, $120$ and $120$, respectively.  Each of the groups $G_i$ contain $-I$.   The groups $H_{i,j}$ do not contain $-I$ and we have $G_i = \pm H_{i,j}$.  

Let $\calE$ be the elliptic curve over $\QQ$ defined by the Weierstrass equation $y^2+y = x^3-x^2-7x+10$ and let $\OO$ be the point at infinity.   The Mordell-Weil group $\calE(\QQ)$ is an infinite cyclic group generated by the point $(4,5)$.     Define 
\[
J(x,y):=\frac{(f_1 f_2 f_3 f_4)^3}{f_5^2 f_6^{11}},
\]
where
\begin{align*}
f_1&=x^2+3x-6,
&f_2&=11(x^2-5)y+(2x^4+23x^3-72x^2-28x+127),\\
f_3&=6y+11x-19,
&f_4&=22(x-2)y+(5x^3+17x^2-112x+120), \\
f_5&=11y+(2x^2+17x-34), 
&f_6&=(x-4)y-(5x-9).
\end{align*}
We shall view $J$ as a morphism $\calE \to \AA_\QQ^1\cup\{\infty\}$.\\

\noindent
Let $\calE_{1}/\QQ$ be the elliptic curve defined by the Weierstrass equation  $y^2 = x^3-27\cdot 11^4 x + 54\cdot 11^5\cdot 43$

\noindent 
Let $\calE_{2}/\QQ$ be the elliptic curve defined by the Weierstrass equation $y^2=x^3-27\cdot 11^3\cdot 131 x +54\cdot 11^4\cdot 4973$.

\begin{thm} \label{T:11 main}
Let $E$ be a non-CM elliptic curve defined over $\QQ$.
\begin{romanenum}  
\item \label{T:11 main a}
If $\rho_{E,11}$ is not surjective, then $\rho_{E,11}(\Gal_\QQ)$ is conjugate in $\GL_2(\FF_{11})$ to one of the groups $G_i$ or $H_{i,j}$.
\item  \label{T:11 main b}
The group $\pm\rho_{E,11}(\Gal_\QQ)$ is conjugate to $G_1$ in $\GL_2(\FF_{11})$ if and only if $j_E=-11^2$.
\item \label{T:11 main c}
The group $\pm\rho_{E,11}(\Gal_\QQ)$ is conjugate to $G_2$ in $\GL_2(\FF_{11})$ if and only if $j_E=-11\cdot 131^3$.
\item  \label{T:11 main d}
The group $\rho_{E,11}(\Gal_\QQ)$ is conjugate to $G_3$ in $\GL_2(\FF_{11})$ if and only if $j_E=J(P)$ for some point $P\in \calE(\QQ)-
\{\OO\}$. 
\item \label{T:11 main e}
For $i\in \{1,2\}$, the group $\rho_{E,11}(\Gal_\QQ)$ is conjugate in $\GL_2(\FF_{11})$ to $H_{i,1}$ if and only if $E$ is isomorphic to $\calE_{i}$.
\item \label{T:11 main f}
For $i\in \{1,2\}$, the group $\rho_{E,11}(\Gal_\QQ)$ is conjugate in $\GL_2(\FF_{11})$ to $H_{i,2}$ if and only if $E$ is isomorphic to the quadratic twist of $\calE_{i}$ by $-11$.
\end{romanenum}
\end{thm}

\begin{remark}  \label{R:ns computation}
The modular curve $X_{ns}^+(11)=X_{G_3}$ is the only one in our classification that has genus $1$ with infinitely many rational points.  Halberstadt \cite{MR1677158} showed that $X_{ns}^+(11)$ is isomorphic to $\calE$ and that the morphism to the $j$-line corresponds to $J(x,y)$.    

In \S\ref{S:ns section}, we give explicit polynomials $A,B,C\in \QQ[X]$ of degree $55$ such that for a non-CM elliptic curve $E/\QQ$, we have $j_E=J(P)$ for some $P\in \calE(\QQ)-\{\OO\}$ if and only if the polynomial  $A(x)j_E^2+B(x)j_E +C(x) \in \QQ[x]$ has a rational root.  This gives a straightforward way to check the criterion of Theorem~\ref{T:11 main}(\ref{T:11 main d}).  
\end{remark}

\subsection{\underline{$\ell=13$}}  \label{SS:applicable 13}

Define the following subgroups of $\GL_2(\FF_{13})$:
\begin{itemize}
\item Let $G_1$ be the subgroup consisting of matrices of the form $\left(\begin{smallmatrix}* & * \\0 & b^3 \end{smallmatrix}\right)$.
\item Let $G_2$ be the subgroup consisting of matrices of the form $\left(\begin{smallmatrix}a^3 & * \\0 & * \end{smallmatrix}\right)$.
\item Let $G_3$ be the subgroup consisting of matrices $\left(\begin{smallmatrix}a & * \\0 & b \end{smallmatrix}\right)$ for which $(a/b)^4=1$.
\item Let $G_4$ be the subgroup consisting of matrices of the form $\left(\begin{smallmatrix}* & * \\0 & b^2 \end{smallmatrix}\right)$.
\item Let $G_5$ be the subgroup consisting of matrices of the form $\left(\begin{smallmatrix}a^2 & * \\0 & * \end{smallmatrix}\right)$.
\item Let $G_6$ be the group $B(13)$. 
\item Let $G_7$ be the subgroup generated by the matrices 
$\left(\begin{smallmatrix}2 & 0 \\0 & 2 \end{smallmatrix}\right)$, $\left(\begin{smallmatrix}2 & 0 \\0 & 3 \end{smallmatrix}\right)$, $\left(\begin{smallmatrix}0 & -1 \\1 & 0 \end{smallmatrix}\right)$ and $\left(\begin{smallmatrix}1 & 1 \\-1 & 1 \end{smallmatrix}\right)$;  it contains the scalar matrices and its image in $\PGL_2(\FF_{13})$ is isomorphic to $\mathfrak{S}_4$.
\item
Let $H_{4,1}$ be the subgroup consisting of matrices of the form $\left(\begin{smallmatrix}* & * \\0 & a^4 \end{smallmatrix}\right)$.   
\item
Let $H_{4,2}$ be the subgroup generated by matrices of the form $\big(\begin{smallmatrix} b^2 & * \\0 & a^4 \end{smallmatrix}\big)$ and    $\left(\begin{smallmatrix} 2 & 0 \\0 &  4 \end{smallmatrix}\right)$.
\item
Let $H_{5,1}$ be the subgroup consisting of matrices of the form $\left(\begin{smallmatrix}a^4 & * \\0 & * \end{smallmatrix}\right)$.   
\item
Let $H_{5,2}$ be the subgroup generated by matrices of the form $\big(\begin{smallmatrix}a^4 & * \\0 & b^2 \end{smallmatrix}\big)$ and    $\left(\begin{smallmatrix} 4 & 0 \\0 &  2 \end{smallmatrix}\right)$.
\end{itemize}
The index in $\GL_2(\FF_{13})$ of the above subgroups are $42$, $42$, $42$, $28$, $28$, $14$, $91$, $56$, $56$, $56$ and $56$, respectively.   Each of the groups $G_i$ contain $-I$.   The groups $H_{i,j}$ do not contain $-I$ and we have $G_i = \pm H_{i,j}$.  \\

Define the polynomials 
\begin{align*}
P_1(t)&= t^{12} + 231t^{11} + 269t^{10} - 3160t^9 + 6022t^8 - 9616t^7 + 21880t^6 \\ &\quad - 34102t^5 + 28297t^4 - 12455t^3 + 2876t^2 - 243t + 1\\
P_2(t)&= t^{12} - 9t^{11} + 29t^{10} - 40t^9 + 22t^8 - 16t^7 + 40t^6 - 22t^5 - 23t^4 + 25t^3 - 4t^2 - 3t + 1\\
P_3(t)&= (t^4-t^3+2t^2-9t+3)(3t^4-3t^3-7t^2+12t-4)(4t^4-4t^3-5t^2+3t-1)\\
P_4(t)&= t^8 + 235t^7 + 1207t^6 + 955t^5 + 3840t^4 - 955t^3 + 1207t^2 - 235t+ 1\\
P_5(t)&= t^8 - 5t^7 + 7t^6 - 5t^5 + 5t^3 + 7t^2 + 5t + 1\\
P_6(t)&= t^4+7t^3+20t^2+19t+1\\  
Q_4(t)&=\,t^{12} - 512 t^{11} - 13079 t^{10} - 32300 t^9 - 104792 t^8 - 111870 t^7 \\
&\quad\quad\quad - 419368 t^6 + 111870 t^5 - 104792 t^4 + 32300 t^3 - 13079 t^2 + 512 t +1, \\
Q_5(t)&=\, t^{12} - 8 t^{11} + 25 t^{10} - 44 t^9 + 40 t^8 + 18 t^7 - 40 t^6 - 18 t^5 + 40 t^4 + 44 t^3 + 25 t^2 + 8 t + 1
\end{align*}
and  the rational functions
\begin{align*}
J_1(t) & =  \frac{(t^2 - t + 1)^3 P_1(t)^3}{(t - 1)t(t^3 - 4t^2 + t + 1)^{13}}
\quad \quad \quad  \quad \quad \quad \quad    
J_2(t)  = \frac{(t^2 - t + 1)^3P_2(t)^3}{(t - 1)^{13} t^{13} (t^3 - 4t^2 + t + 1)}\\
J_3(t) &= -\frac{13^4(t^2-t+1)^3 P_3(t)^3}{(t^3-4t^2+t+1)^{13} (5t^3-7t^2-8t+5)}
\quad 
J_4(t) = \frac{(t^4 - t^3 + 5t^2 + t + 1)P_4(t)^3}{t (t^2 - 3t - 1)^{13}}  \\
J_5(t) &=  \frac{(t^4 - t^3 + 5t^2 + t + 1) P_5(t)^3}{t^{13} (t^2 - 3t - 1)}
\quad \quad \quad \quad \quad \quad \,\,   J_6(t) =  \frac{(t^2+5t+13)P_6(t)^3}{t}.
\end{align*} 
\noindent 
For $t\in \QQ-\{0\}$, let $\calE_{4,t}$ be the elliptic curve over $\QQ$ defined by the Weierstrass equation
\begin{align*}
y^2= x^3-27(t^4 - t^3 + 5t^2 + t + 1)^3P_4(t) x + 54(t^2 + 1) (t^4 - t^3 + 5 t^2 + t + 1)^4 Q_4(t).
\end{align*}
For $t\in \QQ-\{0\}$, let $\calE_{5,t}$ be the elliptic curve over $\QQ$ defined by the Weierstrass equation
\begin{align*}
y^2= x^3-27(t^4 - t^3 + 5t^2 + t + 1)^3P_5(t) x + 54(t^2 + 1) (t^4 - t^3 + 5 t^2 + t + 1)^4 Q_5(t).
\end{align*}

\begin{thm} \label{T:main 13}
Let $E$ be a non-CM elliptic curve over $\QQ$.   
\begin{romanenum}
\item   \label{T:main 13 a}
If $\rho_{E,13}(\Gal_\QQ)$ is conjugate to a subgroup of $B(13)$,  then $\rho_{E,13}(\Gal_\QQ)$ is conjugate to one of the groups $G_i$ with $1\leq i \leq 6$ or to a group $H_{i,j}$. 

\item \label{T:main 13 b}
For $1\leq i \leq 6$, $\rho_{E,13}(\Gal_\QQ)$ is conjugate in $\GL_2(\FF_{13})$ to a subgroup of $G_i$ if and only if $j_E$ is of the form $J_i(t)$ for some $t\in \QQ$.   

\item  \label{T:main 13 c}
For an $i\in \{4,5\}$, suppose that $J_i(t)=j_E$ for some $t\in \QQ$.

\noindent The group $\rho_{E,13}(\Gal_\QQ)$ is conjugate to $H_{i,1}$ if and only if $E$ is isomorphic to $\calE_{i,t}$.

\noindent The group $\rho_{E,13}(\Gal_\QQ)$ is conjugate to $H_{i,2}$ if and only if $E$ is isomorphic to the quadratic twist of $\calE_{i,t}$ by $13$.

\item \label{T:main 13 d}
If $j_E$ is 
\[
\frac{2^4\cdot 5\cdot 13^4\cdot 17^3}{3^{13}}, \quad  -\frac{2^{12}\cdot 5^3\cdot 11\cdot 13^4}{3^{13}} \quad\text{ or }\quad \frac{2^{18}\cdot3^3\cdot 13^4\cdot 127^3\cdot 139^3\cdot 157^3\cdot 283^3\cdot 929}{5^{13}\cdot 61^{13}},
\]
then $\rho_{E,13}(\Gal_\QQ)$ is conjugate to $G_7$.

\end{romanenum}
\end{thm}

Up to conjugacy, there are four {maximal} subgroups $G$ of $\GL_2(\FF_{13})$ that satisfy $\det(G)=\FF_{13}^\times$; they are $G_6=B(13)$, $N_{s}(13)$, $N_{ns}(13)$ and $G_7$.   The cases concerning subgroups of $B(13)$ are completely handled in Theorem~\ref{T:main 13}.\\

Baran \cite{Baran-13} has shown that the modular curves $X_{s}^+(13)$ and $X_{ns}^+(13)$ attached to $N_s(13)$ and $N_{ns}(13)$, respectively, are both isomorphic to the genus $3$ curve $C$ defined in $\PP^2_\QQ$ by the equation
\[
(-y-z)x^3 +(2y^2+zy)x^2+(-y^3+zy^2-2z^2y+z^3)x+(2z^2y^2-3z^3y)=0.
\]
In \cite{Baran-13}, the morphism from the model of the modular curves to the $j$-line is given.  The seven rational points $(0 , 0, 1)$, $(0 , 1 , 0)$, $(0,3,2)$, $(1, 0, -1)$, $(1, 0, 0)$, $(1,0,1)$, $(1, 1, 0)$ of $C$ all correspond to cusps and CM points on $X_s(13)$ and $X_{ns}(13)$.    Conjecturally, there are no non-CM elliptic curves $E$ over $\QQ$ with $\rho_{E,13}(\Gal_\QQ)$ conjugate to a subgroup of $N_s(13)$ or $N_{ns}(13)$; equivalently, $C$ has no other rational points.   
\\

Denote by $X_{\mathfrak{S}_4}(13)$ the modular curve corresponding to $G_7$.   Banwait and Cremona \cite{banwait-cremona} have shown that $X_{\mathfrak{S}_4}(13)$ is isomorphic to the genus $3$ curve $C'$ defined in $\PP^2_\QQ$ by the equation
\[
4x^3y - 3x^2y^2 + 3xy^3 - x^3z + 16x^2yz - 11xy^2z + 5y^3z + 3x^2z^2 + 9xyz^2 + y^2z^2 + xz^3 + 2yz^3 = 0
\]
and have found the morphism from the modular curve to the $j$-line.  The four rational points $(0,1,0)$, $(0,0,1)$, $(1,0,0)$ and $(1,3,-2)$ of $C'$ correspond to a CM point and three non-CM points; the non-CM points give rise to the three $j$-invariants in Theorem~\ref{T:main 13}(\ref{T:main 13 d}).   

Suppose $E/\QQ$ is an elliptic curve with one of the $j$-invariants from Theorem~\ref{T:main 13}(\ref{T:main 13 d}).    From \cite{banwait-cremona}, we find that the image of $\rho_{E,13}(\Gal_\QQ)$ in $\PGL_2(\FF_{13})$ is isomorphic to $\mathfrak{S}_4$.  Therefore, $\rho_{E,13}(\Gal_\QQ)$ is conjugate to $G_7$ since $G_7$ has no proper subgroups $H$ whose image in $\PGL_2(\FF_{13})$ is isomorphic to $\mathfrak{S}_4$ and satisfies $\det(H)=\FF_{13}^\times$.   In particular, this proves Theorem~\ref{T:main 13}(\ref{T:main 13 d}).

 Conjecturally, if $E$ is a non-CM elliptic curve over $\QQ$, then $\rho_{E,13}(\Gal_\QQ)$ is conjugate to a subgroup of $G_7$ if and only if $j_E$ is one of three values from Theorem~\ref{T:main 13}(\ref{T:main 13 d}); equivalently, $C'$ has no other rational points.

\begin{remark}
The case $\ell=13$ is the first for which we do not have a complete description.  As explained above, it remains to determine all the rational points of the genus $3$ curves $C$ and $C'$.  
\end{remark}

\subsection{\underline{$\ell\geq 17$}}  \label{SS:applicable 17}

We first describe all the known cases of non-CM elliptic curves $E/\QQ$ for which $\rho_{E,\ell}$ is not surjective for some prime $\ell\geq 17$.   Define the following groups:
\begin{itemize}
\item 
Let $G_1$ be the subgroup of $\GL_2(\FF_{17})$ generated by $\left(\begin{smallmatrix}2 & 0 \\0 & 11 \end{smallmatrix}\right)$, $\left(\begin{smallmatrix}4 & 0 \\0 & -4 \end{smallmatrix}\right)$ and $\left(\begin{smallmatrix}1 & 1 \\0 &1 \end{smallmatrix}\right)$.
\item
Let $G_2$ be the subgroup of $\GL_2(\FF_{17})$ generated by $\left(\begin{smallmatrix}11 & 0 \\0 & 2 \end{smallmatrix}\right)$, $\left(\begin{smallmatrix}-4 & 0 \\0 & 4 \end{smallmatrix}\right)$ and $\left(\begin{smallmatrix}1 & 1 \\0 &1 \end{smallmatrix}\right)$.
\item
Let $G_3$ be the subgroup of $\GL_2(\FF_{37})$ consisting of the matrices of the form $\left(\begin{smallmatrix}a^3 & * \\0 & * \end{smallmatrix}\right)$.
\item
Let $G_4$ be the subgroup of $\GL_2(\FF_{37})$ consisting of the matrices of the form $\left(\begin{smallmatrix}* & * \\0 & a^3 \end{smallmatrix}\right)$.
\end{itemize}

\begin{thm} \label{T:17-37}
\begin{romanenum}
\item \label{T:17-37 i}
If $E/\QQ$ has $j$-invariant $-17\cdot 373^3/2^{17}$ or $-17^2 \cdot 101^3/2$, then $\rho_{E,17}(\Gal_\QQ)$ is conjugate in $\GL_2(\FF_{17})$ to $G_1$ or $G_2$, respectively.
\item \label{T:17-37 ii}
If $E/\QQ$ has $j$-invariant $-7\cdot 11^3$ or $-7\cdot 137^3\cdot 2083^3$, then $\rho_{E,37}(\Gal_\QQ)$ is conjugate in $\GL_2(\FF_{37})$ to $G_3$ or $G_4$, respectively.   
\end{romanenum} 
\end{thm}

\begin{thm}[Mazur, Serre, Bilu-Parent-Rebolledo] \label{T:Mazur-Serre-BPR}
Fix a prime $\ell\geq 17$ and let $E$ be a non-CM elliptic curve defined over $\QQ$.     If $(\ell,j_E)$ does not belong to the set
\begin{equation} \label{E:set exceptional 17-37}
\{ (17,-17\cdot 373^3/2^{17}), \, (17, -17^2 \cdot 101^3/2), \, (37,-7\cdot 11^3),\, (37,-7\cdot 137^3\cdot 2083^3) \},
\end{equation}
then either $\rho_{E,\ell}$ is surjective or $\rho_{E,\ell}(\Gal_\QQ)$ is conjugate to a subgroup of $N_{ns}(\ell)$.
\end{thm}
\begin{proof}
The group $\GL_2(\FF_\ell)$ has either three or four maximal subgroups with determinant $\FF_\ell^\times$.   They are $B(\ell)$, $N_s(\ell)$, $N_{ns}(\ell)$ and when $\ell\equiv \pm 3 \pmod{8}$, we also have a maximal subgroup $H_{\mathfrak{S}_4}(\ell)$ whose image in $\PGL_2(\FF_\ell)$ is isomorphic to the symmetric group $\mathfrak{S}_4$.

Take any non-CM elliptic curve $E$ over $\QQ$.  Serre has shown that $\rho_{E,\ell}(\Gal_\QQ)$ cannot be conjugate to a subgroup of $H_{\mathfrak{S}_4}(\ell)$, cf.~\cite{MR644559}*{\S8.4}.    Bilu, Parent and Rebolledo have proved that $\rho_{E,\ell}(\Gal_\QQ)$ cannot be conjugate to a subgroup of $N_s(\ell)$, cf.~\cite{1104.4641} (they make effective the bounds in earlier works of Bilu and Parent using improved isogeny bounds of Gaudron and R\'emond).     The $B(\ell)$ case follows from a famous theorem of Mazur, cf.~\cite{MR482230}.  The modular curves $X_0(17)$ and $X_0(37)$ each have two rational points which are not cusps or CM points and they are accounted for by the curves of Theorem~\ref{T:17-37}.    
\end{proof}

We conjecture that Theorem~\ref{T:Mazur-Serre-BPR} describes all the reasons that $\rho_{E,\ell}$ can fail to be surjective for a non-CM $E/\QQ$ and a prime $\ell \geq 17$; this is a problem raised by Serre, cf.~\cite{MR644559}*{p.399}, who asked if $\rho_{E,\ell}$ is surjective whenever $\ell >37$.

\begin{conj}  \label{C:main}
If $E$ is a non-CM elliptic curve over $\QQ$ and $\ell\geq 17$ is a prime such that the pair $(\ell,j_E)$ does not belong to the set (\ref{E:set exceptional 17-37}), then $\rho_{E,\ell}(\Gal_\QQ)=\GL_2(\FF_\ell)$.
\end{conj}

Even if Conjecture~\ref{C:main} is false for some $E/\QQ$ and $\ell\geq 17$, the following proposition gives at most two possibilities for the image of $\rho_{E,\ell}$ (they can be distinguished computationally by looking at the division polynomial of $E$ at $\ell$).

\begin{prop} \label{P:big inertia last}
Suppose that $\rho_{E,\ell}$ is not surjective for a non-CM elliptic curve $E/\QQ$ and a prime $\ell\geq 17$ for which $(\ell,j_E)$ does not lie in the set (\ref{E:set exceptional 17-37}).  
\begin{romanenum}
\item
If $\ell \equiv 1 \pmod{3}$, then $\rho_{E,\ell}(\Gal_\QQ)$ is conjugate in $\GL_2(\FF_\ell)$ to $N_{ns}(\ell)$.
\item
If $\ell \equiv 2 \pmod{3}$, then $\rho_{E,\ell}(\Gal_\QQ)$ is conjugate in $\GL_2(\FF_\ell)$ to $N_{ns}(\ell)$ or to the group
\[
G:=\big\{a^3: a \in C_{ns}(\ell)\big\} \cup \big\{ \left(\begin{smallmatrix}1 &0  \\0 & -1 \end{smallmatrix}\right) \cdot a^3: a \in C_{ns}(\ell) \big\}.
\]
\end{romanenum}  
\end{prop}

\subsection{Algorithm} 
Let $E/\QQ$ be a non-CM elliptic curve (when $E/\QQ$ has complex multiplication, the groups $\rho_{E,\ell}(\Gal_\QQ)$ are all described in \S\ref{SS:CM} below).    In \cite{Zywina-images}, we give an algorithm to compute the set $S'$ of primes $\ell \geq 13$ for which $\rho_{E,\ell}$ is not surjective.  

Combined with the theorems from \S\S\ref{SS:applicable 2}--\ref{SS:applicable 11}, we are now able to compute the (finite) set $S$ of primes $\ell$ for which $\rho_{E,\ell}$ is not surjective.  Moreover, using the results from \S\S\ref{SS:applicable 2}--\ref{SS:applicable 11}, we can give the group $\rho_{E,\ell}(\Gal_\QQ)$, up to conjugacy in $\GL_2(\FF_\ell)$, for each $\ell \in S$.\\

Sutherland has a probabilistic algorithm to determine the groups $\rho_{E,\ell}(\Gal_\QQ)$ by consider Frobenius at many primes $p$, \cite{Sutherland2015}.  His algorithm can in principle be made deterministic using effective versions of the Chebotarev density theorem.     Sutherland's algorithm has the advantage that it can be used for elliptic curves over a number field $K\neq \QQ$ (for our approach, we would have more modular curves to consider and those modular curves not isomorphic to $\PP^1_\QQ$ would need to be reconsidered).

The next task that needs to be completed is to consider the images of $\rho_{E,\ell^n}$ for small primes $\ell$ and $n\geq 2$.  Rouse and Zureick-Brown have already done this for $\ell=2$, cf.~\cite{R-DZ}; the case $\ell=2$ is rather accessible since all the groups that occur are solvable.

\newpage
\subsection{Complex multiplication} \label{SS:CM}

Up to isomorphism over $\Qbar$, there are thirteen elliptic curves with complex multiplication that are defined over $\QQ$.  In Table 1 below, we give an elliptic curve $E_{D,f}/\QQ$ with each of these thirteen $j$-invariants (this comes from \cite{MR1312368}*{Appendix A \S3} though with some different models).   The curve $E_{D,f}$ has conductor $N$ and has complex multiplication by an order $R$ of conductor $f$ in the  imaginary quadratic field with discriminant $-D$.

{
\renewcommand{\arraystretch}{1.1}

\begin{table}[htdp] 
\begin{center}\begin{tabular}{|c|c|c|l|c|}\hline  
$j$-invariant & $D$ &  $f$ &  Elliptic curve $E_{D,f}$ & $N$ \\ \hline\hline
$0$ &$3$ & $1$ &  $y^2=x^3+16$ & $3^3$ \\  
$2^4 3^3 5^3$ & & $2$ & $y^2=x^3-15x+22$ & $2^2 3^2$  \\  % 2
 $-2^{15} 3 \cdot 5^3$ & &  $3$ &  $y^2=x^3-480x+4048$ & $3^3$  \\ \hline  %6
$2^6 3^3=1728$ & $4$ & $1$ &  $y^2=x^3+x$ & $2^6$  \\   % ---
 $2^3 3^3 11^3$ &  & $2$ &  $y^2=x^3-11x+14$ & $2^5$ \\ \hline % 2
$-3^3 5^3$ & $7$ & $1$ &  $y^2=x^3-1715x+33614$  & $7^2$  \\   %2
$3^3 5^3 17^3$ & &  $2$ & $y^2=x^3-29155x+1915998$ & $7^2$  \\ \hline %2
 $2^6 5^3$ & $8$ & $1$ & $y^2=x^3-4320x+96768$ & $2^8$  \\  \hline %2
$-2^{15}$ & $11$ & $1$ & $y^2=x^3-9504x+365904$ & $11^2$  \\ \hline %6
$-2^{15} 3^3$ & $19$ & $1$ & $y^2=x^3-608x+5776$ & $19^2$  \\ \hline %6
$-2^{18} 3^3 5^3$ & $43$ & $1$ & $y^2=x^3-13760x+621264$ & $43^2$  \\ \hline %6
$-2^{15} 3^3 5^3 11^3$ & $67$ & $1$ & $y^2=x^3-117920x+15585808$ & $67^2$   \\ \hline  %6
$-2^{18} 3^3 5^3 23^3 29^3$ & $163$ & $1$ & $y^2=x^3-34790720x+78984748304$ & $163^2$  \\ \hline
\end{tabular} \caption{CM elliptic curves over $\QQ$}
\end{center}
\label{defaulttable}
\end{table}
}

We first describe the group $\rho_{E,\ell}(\Gal_\QQ)$ up to conjugacy when $E$ is a CM elliptic curve with non-zero $j$-invariant and $\ell$ odd.

\begin{prop}  \label{P:CM main}
Let $E$ be a CM elliptic curve defined over $\QQ$ with $j_E\neq 0$.  The ring of endomorphisms of $E_{\Qbar}$ is an order of conductor $f$ in the ring of integers of an imaginary quadratic field of discriminant $-D$. Take any odd prime $\ell$.
\begin{romanenum}
\item \label{P:CM main a}
If $\legendre{-D}{\ell}=1$, then $\rho_{E,\ell}(\Gal_\QQ)$ is conjugate in $\GL_2(\FF_\ell)$ to $N_s(\ell)$.
\item \label{P:CM main b}
If $\legendre{-D}{\ell}=-1$, then $\rho_{E,\ell}(\Gal_\QQ)$ is conjugate in $\GL_2(\FF_\ell)$ to $N_{ns}(\ell)$.

\item \label{P:CM main c}
Suppose that $\ell$ divides $D$ and hence $D=\ell$.  Define the groups 
\[
G=\{ \left(\begin{smallmatrix} a & b \\0 & \pm a \end{smallmatrix}\right) : a\in \FF_\ell^\times, b\in \FF_\ell\},
\]
\[
H_1 = \{ \left(\begin{smallmatrix} a & b \\0 & \pm a \end{smallmatrix}\right) : a\in (\FF_\ell^\times)^2, b\in \FF_\ell\}, \quad \text{ and }\quad H_2 = \{ \left(\begin{smallmatrix} \pm a & b \\0 &  a \end{smallmatrix}\right) : a\in (\FF_\ell^\times)^2, b\in \FF_\ell\} 
\]

\noindent 
If $E$ is isomorphic to $E_{D,f}$, then $\rho_{E,\ell}(\Gal_\QQ)$ is conjugate in $\GL_2(\FF_\ell)$ to $H_1$.

\noindent
If $E$ is isomorphic to the quadratic twist of $E_{D,f}$ by $-\ell$, then $\rho_{E,\ell}(\Gal_\QQ)$ is conjugate in $\GL_2(\FF_\ell)$ to $H_2$.

\noindent
If $E$ is not isomorphic to $E_{D,f}$ or its quadratic twist by $-\ell$, then $\rho_{E,\ell}(\Gal_\QQ)$ is conjugate in $\GL_2(\FF_\ell)$ to $G$.

\end{romanenum}
\end{prop}

The following deals with the excluded prime $\ell=2$.

\begin{prop} \label{P: prime 2}
Let $E/\QQ$ be a CM elliptic curve.  Define the subgroup $G_2=\{ I, \left(\begin{smallmatrix}1 & 1 \\0 & 1 \end{smallmatrix}\right)\}$ of $\GL_2(\FF_2)$.
\begin{romanenum}
\item \label{P: prime 2 i}
If  $j_E \in \{2^4 3^3 5^3,\, 2^3 3^3 11^3,\, -3^3 5^3,\, 3^3 5^3 17^3,\, 2^6 5^3\},$
then $\rho_{E,2}(\Gal_\QQ)$ is conjugate to $G_2$.
\item
If $j_E \in \{ -2^{15} 3 \cdot 5^3,\, -2^{15},\, -2^{15} 3^3,\, -2^{18} 3^3 5^3,\, -2^{15} 3^3 5^3 11^3,\, -2^{18} 3^3 5^3 23^3 29^3\},$
then $\rho_{E,2}(\Gal_\QQ)=\GL_2(\FF_2)$.

\item
Suppose that $j_E=1728$.  The curve can be given by a Weierstrass equation $y^2=x^3-dx$ for some $d\in \QQ^\times$.

\noindent
If $d$ is a square, then $\rho_{E,2}(\Gal_\QQ)=\{I\}$.  

\noindent 
If $d$ is not a square, then the group $\rho_{E,2}(\Gal_\QQ)$ is conjugate to $G_2$.  

\item 
Suppose that $j_E=0$.  The curve $E$ can be given by a Weierstrass equation $y^2=x^3+d$ for some $d\in \QQ^\times$. 

\noindent 
If $d$ is a cube, then $\rho_{E,2}(\Gal_\QQ)$ is conjugate in $\GL_2(\FF_2)$ to the group $G_2$.

\noindent 
If $d$ is not a cube, then $\rho_{E,2}(\Gal_\QQ)=\GL_2(\FF_2)$. 
\end{romanenum}
\end{prop}

It remains to consider the situation where $\ell$ is an odd prime and $E/\QQ$ is an elliptic curve with $j_E=0$.   That such curves have cubic twists make the classification more involved.

\begin{prop}  \label{P:j=0 situation}
Let $E$ be an elliptic curve over $\QQ$ with $j_E=0$.  Take any odd prime $\ell$.
\begin{romanenum}
\item  \label{P:j=0 situation i}
If $\ell \equiv 1 \pmod{9}$, then $\rho_{E,\ell}(\Gal_\QQ)$ is conjugate to $N_{s}(\ell)$ in $\GL_2(\FF_\ell)$.
\item \label{P:j=0 situation ii}
If $\ell \equiv 8 \pmod{9}$, then $\rho_{E,\ell}(\Gal_\QQ)$ is conjugate to $N_{ns}(\ell)$ in $\GL_2(\FF_\ell)$.
\item \label{P:j=0 situation iii}
Suppose that $\ell$ is congruent to $4$ or $7$ modulo $9$.  Let $E'/\QQ$ be the elliptic curve over $\QQ$ defined by $y^2=x^3+16 \ell^e$, where $e\in \{1,2\}$ satisfies $ \frac{\ell-1}{3} \equiv e \pmod{3}$.   

\noindent
If $E$ is not isomorphic to a quadratic twist of $E'$, then $\rho_{E,\ell}(\Gal_\QQ)$ is conjugate to $N_{s}(\ell)$ in $\GL_2(\FF_\ell)$.

\noindent 
If $E$ is isomorphic to a quadratic twist of $E'$, then $\rho_{E,\ell}(\Gal_\QQ)$ is conjugate in $\GL_2(\FF_\ell)$ to the subgroup $G$ of $N_s(\ell)$ consisting of the matrices of the form  $\left(\begin{smallmatrix} a & 0 \\0 & b\end{smallmatrix}\right)$ or $\left(\begin{smallmatrix} 0 & a \\b & 0 \end{smallmatrix}\right)$ with $a/b \in (\FF_\ell^\times)^3$.

\item 
\label{P:j=0 situation iv}
Suppose that $\ell$ is congruent to $2$ or $5$ modulo $9$.  Let $E'/\QQ$ be the elliptic curve over $\QQ$ defined by $y^2=x^3+16 \ell^e$, where $e\in \{1,2\}$ satisfies $ \frac{\ell+1}{3} \equiv -e \pmod{3}$.   

\noindent
If $E$ is not isomorphic to a quadratic twist of $E'$, then $\rho_{E,\ell}(\Gal_\QQ)$ is conjugate to $N_{ns}(\ell)$ in $\GL_2(\FF_\ell)$.

\noindent 
If $E$ is isomorphic to a quadratic twist of $E'$, then $\rho_{E,\ell}(\Gal_\QQ)$ is conjugate in $\GL_2(\FF_\ell)$ to the subgroup $G$ of $N_{ns}(\ell)$ generated by the unique index $3$ subgroup of $C_{ns}(\ell)$ and by $\left(\begin{smallmatrix} 1 & 0 \\0 & -1\end{smallmatrix}\right)$.

\item  \label{P:j=0 situation v}
Suppose that $\ell=3$.  The curve $E$ can be given by a Weierstrass equation $y^2=x^3+d$ for some $d\in \QQ^\times$.   Fix notation as in \S\ref{SS:applicable 3}.

\noindent
If $d$ or $-3d$ is a square and $-4d$ is a cube, then $\rho_{E,3}(\Gal_\QQ)$ is conjugate to $H_{1,1}$.

\noindent
If $d$ and $-3d$ are not squares and $-4d$ is a cube, then $\rho_{E,3}(\Gal_\QQ)$ is conjugate to $G_1$.

\noindent
If $d$ is a square and $-4d$ is not a cube, then $\rho_{E,3}(\Gal_\QQ)$ is conjugate to $H_{3,1}$.

\noindent
If $-3d$ is a square and $-4d$ is not a cube, then $\rho_{E,3}(\Gal_\QQ)$ is conjugate to $H_{3,2}$.

\noindent
If $d$ and $-3d$ are not squares and $-4d$ is not a cube, then $\rho_{E,3}(\Gal_\QQ)$ is conjugate to $G_3$.

\end{romanenum}
\end{prop}

\subsection{Overview} \label{SS:overview}

We now give a very brief overview of the paper.  In \S\ref{SS:applicable}, we describe \defi{applicable subgroups} $G$ of $\GL_2(\FF_\ell)$; these groups have many of the properties that the groups $\pm \rho_{E,\ell}(\Gal_\QQ)$ do.

In \S\ref{S:modular}, we recall what we need concerning the modular curve $X_G/\QQ$; we will identify its function field with a subfield of the field of modular function for the congruence subgroup $\Gamma(\ell)$.

In \S\ref{S:main classification}, we prove the parts of our main theorems that determine $\pm \rho_{E,\ell}(\Gal_\QQ)$.   We describe the rational points of $X_G$ when $\ell$ is small.   When $X_G$ has genus $0$ and $X_G(\QQ)\neq \emptyset$, then the function field of $X_G$ is of the form $\QQ(h)$ for some modular function $h$.    Much of this section is dedicated to describing such $h$ and determining the rational function $J(t) \in \QQ(t)$ such that $J(h)$ is the modular $j$-invariant.

Assuming that $G:=\pm \rho_{E,\ell}(\Gal_\QQ)$ is known, with $E/\QQ$ non-CM, we describe in \S\ref{S:twist 1} how to determine the (finite number of) quadratic twists of $E'$ of $E$ for which $\rho_{E',\ell}(\Gal_\QQ)$ is not conjugate to $G$.   In \S\ref{S:twists 2}, we prove the parts of our main theorems that determine $\rho_{E,\ell}(\Gal_\QQ)$ given $\pm \rho_{E,\ell}(\Gal_\QQ)$.

In \S\ref{SS:CM proofs},  we prove the propositions from \S\ref{SS:CM} concerning CM elliptic curves defined over $\QQ$.   The $j$-invariant $0$ case requires special attention since one has to worry about cubic twists.  Finally, in \S\ref{S:big inertia last}, we prove Proposition~\ref{P:big inertia last}.

The equations in \S\ref{SS:applicable 2}--\ref{SS:applicable 17} and \texttt{Magma} code verifying some claims in \S\ref{S:main classification} and \S\ref{S:twists 2} can be found at:
\begin{center}\url{http://www.math.cornell.edu/~zywina/papers/PossibleImages/}\end{center}

\subsection*{Acknowledgments}
Thanks to Andrew Sutherland,  David Zureick-Brown and Ren\'e Schoof.   The computations in this paper were performed using the \texttt{Magma} computer algebra system \cite{Magma}.

\section{Applicable subgroups} \label{SS:applicable}

Fix an integer $N\geq 2$.   For an elliptic curve $E/\QQ$, let $E[N]$ be the $N$-torsion subgroup of $E(\Qbar)$.   After choosing a basis for $E[N]$ as a $\ZZ/N\ZZ$-module, the natural $\Gal_\QQ$-action on $E[N]$ can be expressed in terms of a Galois representation
\[
\rho_{E,N}\colon \Gal_\QQ \to \GL_2(\ZZ/N\ZZ).
\]
When $N$ is a prime, these agree with the representations of \S\ref{S:classification}.   We now describe some restrictions on the possible images of $\rho_{E,N}$.

\begin{definition}
We say that a subgroup $G$ of $\GL_2(\ZZ/N\ZZ)$ is \defi{applicable}  if it satisfies the following conditions:
\begin{itemize}
\item $G\neq \GL_2(\ZZ/N\ZZ)$,
\item $-I \in G$ and $\det(G)=(\ZZ/N\ZZ)^\times$,
\item $G$ contains an element with trace $0$ and determinant $-1$ that fixes a point in $(\ZZ/N\ZZ)^2$ of order $N$.   
\end{itemize}
\end{definition}

This definition is justified by the following.

\begin{prop} \label{P:basic applicable}
Let $E$ be an elliptic curve over $\QQ$ for which $\rho_{E,N}$ is not surjective.  Then $\pm \rho_{E,N}(\Gal_\QQ)$ is an applicable subgroup of $\GL_2(\ZZ/N\ZZ)$.
\end{prop}
\begin{proof}
The group $G:=\pm \rho_{E,N}(\Gal_\QQ)$ clearly contains $-I$.    The character $\det\circ \rho_{E,N}\colon \Gal_\QQ \to (\ZZ/N\ZZ)^\times$ is the surjective homomorphism describing the Galois action on the group of $N$-th roots of unity in $\Qbar$, i.e., for a $N$-th root of unity $\zeta\in \Qbar$, we have $\sigma(\zeta)=\zeta^{\det(\rho_{E,N}(\sigma))}$ for all $\sigma\in \Gal_\QQ$.   Therefore, $\det\circ \rho_{E,N}$ is surjective and hence $\det(G)=(\ZZ/N\ZZ)^\times$.   

Let $c \in \Gal(\Qbar/\QQ)$ be an automorphism corresponding to complex conjugation under some embedding $\Qbar \hookrightarrow \CC$.   Set $g:=\rho_{E,N}(c)$.  As a topological group, the connected component of $E(\RR)$ containing the identity is isomorphic to  $\RR/\ZZ$.  Therefore, $E(\RR)$ contains a point $P_1$ of order $N$.   We may assume that $\rho_{E,N}$ is chosen with respect to a basis whose first term is $P_1$, and hence $g$ is upper triangular whose first diagonal term is $1$.    We have $\det(g)=-1$ since $c$ acts by inversion on $N$-th roots of unity.  Therefore, $g$ is upper triangular with diagonal entries $1$ and $-1$, and hence $\tr(g)=0$.  

Now suppose that $G=\GL_2(\ZZ/N\ZZ)$.   Define $S=\rho_{E,N}(\Gal_\QQ) \cap \SL_2(\ZZ/N\ZZ)$.   Since $G=\GL_2(\ZZ/N\ZZ)$, $\rho_{E,N}(\Gal_\QQ)\neq \GL_2(\ZZ/N\ZZ)$ and $\det(\rho_{E,N}(\Gal_\QQ))=(\ZZ/N\ZZ)^\times$, we deduce that $S\neq \SL_2(\ZZ/N\ZZ)$ and $\pm S=\SL_2(\ZZ/N\ZZ)$.  However, this is impossible by Lemma~\ref{L:proper SL2} below, so we must have $G\neq \GL_2(\ZZ/N\ZZ)$.
\end{proof}

\begin{lemma} \label{L:proper SL2}
There is no proper subgroup $S$ of $\SL_2(\ZZ/N\ZZ)$ such that $\pm S=\SL_2(\ZZ/N\ZZ)$.
\end{lemma}
\begin{proof}
Suppose that $S$ is a subgroup of $\SL_2(\ZZ/N\ZZ)$ for which $\pm S=\SL_2(\ZZ/N\ZZ)$.  By \cite{MR2721742}*{Lemma A.6}, we deduce that there is a prime power $\ell^e$ dividing $N$ such that the image $\tilde{S}$ of $S$ in $\SL_2(\ZZ/\ell^e\ZZ)$ is a proper subgroup satisfying $\pm \tilde{S} = \SL_2(\ZZ/\ell^e\ZZ)$.  So without loss of generality, we may assume that $N=\ell^e$.

The group $S$ has index $2$ in $\SL_2(\ZZ/\ell^e\ZZ)$.   Therefore, $S$ is normal in $\SL_2(\ZZ/\ell^e\ZZ)$ and the quotient is cyclic of order $2$.   However, the abelianization of $\SL_2(\ZZ/\ell^e\ZZ)$ is a cyclic group of order $\gcd(\ell^e,12)$, cf.~\cite{MR2721742}*{Lemma A.1}.    Therefore, we must have $\ell=2$.   Since the abelianization of $\SL_2(\ZZ/2^e\ZZ)$ is cyclic of order $2$ or $4$, we find that $S$ is the unique subgroup of $\SL_2(\ZZ/2^e\ZZ)$ of index $2$.  The group $S$ is now easy to describe; it is the group of elements in $\SL_2(\ZZ/2^e\ZZ)$ whose image in $\SL_2(\ZZ/2\ZZ)$ lies in the unique cyclic group of order $3$.   However, this implies that $\pm S \neq \SL_2(\ZZ/2^e\ZZ)$ since $-I \equiv I \pmod{2}$.    This contradiction ensures that no such $S$ exists.
\end{proof}

\begin{remark}
When $N$ is a prime $\ell$, which is the setting of this paper, the last condition in the definition of applicable subgroup can be simplified to say simply that $G$ contains an element with trace $0$ and determinant $-1$.  
\end{remark}

\section{Modular curves}   \label{S:modular}

Fix an integer $N\geq 1$; in our later application, we will take $N$ to be a prime $\ell$.     In \S\ref{SS:modular functions intro}, we recall the Galois theory of the field of modular functions of level $N$.   In \S\ref{SS:modular curves intro}, we define modular curves in terms of their functions fields.   We take an unsophisticated approach to modular curves and develop what we need from Shimura's book \cite{MR1291394}; it will be useful for reference in future work.  Alternatively, one could develop modular curves as in \cite{MR0337993}*{IV-3}.

\subsection{Modular functions of level $N$} \label{SS:modular functions intro}

The group $\SL_2(\ZZ)$ acts on the complex upper half plane $\mathfrak{h}$ via linear fractional transformations, i.e., $\gamma_*(\tau) = (a\tau+b)/(c\tau +d )$ for $\gamma= \left(\begin{smallmatrix}a & b \\ c & d \end{smallmatrix}\right) \in \SL_2(\ZZ)$ and $\tau\in \mathfrak{h}$.    Let $\Gamma(N)$ be the congruence subgroup consisting of matrices in $\SL_2(\ZZ)$ that are congruent to $ I$ modulo $N$.   The quotient $\Gamma(N) \backslash\mathfrak{H}$ is a Riemann surface and can be completed to a compact and smooth Riemann surface $X_{N}$.  Let $\tau$ be a variable of the complex upper half plane.

 Every meromorphic function $f$ on $X_{N}$ has a $q$-expansion $\sum_{n\in \ZZ} c_n q^{n/N}$; here the $c_n$ are complex numbers which are $0$ for all but finitely many negative $n$ and $q^{1/N}:= e^{2\pi i \tau/N}$.  We define $\calF_N$ to be the field of meromorphic functions on $X_{N}$ whose $q$-expansion has coefficients in $\QQ(\zeta_N)$, where $\zeta_N$ is the $N$-th root of unity $e^{2\pi i/N}$.  For example, $\calF_1=\QQ(j)$ where $j=j(\tau)$ is the modular $j$-invariant with the familiar expansion
\[
j=q^{-1}+ 744 + 196884q + 21493760q^2 + 864299970q^3 +  \ldots. 
\]   

For each $d\in (\ZZ/N\ZZ)^\times$, let $\sigma_d$ be the automorphism of the field $\QQ(\zeta_N)$ for which $\sigma_d(\zeta_N)=\zeta_N^d$.  We extend $\sigma_d$ to an automorphism of $\calF_N$ by taking a function with $q$-expansion $\sum_n c_n q^{n/N}$ to $\sum_n \sigma_d(c_n) q^{n/N}$.   We let $\SL_2(\ZZ)$ act on $\calF_N$ by taking a modular function $f\in \calF_N$ and a matrix $\gamma\in \SL_2(\ZZ)$ to $f\circ \gamma^t$, i.e., the function $(f\circ \gamma^t)(\tau)=f(\gamma^t_*(\tau))$ where $\gamma^t$ is the transpose of $\gamma$.

\begin{prop} \label{P:modular galois}
The extension $\calF_N$ of $\QQ(j)$ is Galois.  There is a unique isomorphism
\[
\theta_N \colon \GL_2(\ZZ/N\ZZ)/\{\pm I\} \xrightarrow{\sim}  \Gal(\calF_N/\QQ(j))
\]
such that the following holds for all $f\in \calF_N$:  
\begin{alphenum}
\item \label{P:modular galois a}
For $A\in \SL_2(\ZZ/N\ZZ)$,  we have $\theta_N(A) f = f\circ \gamma^t$, where $\gamma$ is any matrix in $\SL_2(\ZZ)$ that is congruent to $A$ modulo $N$. 
\item \label{P:modular galois b}
For $A=\left(\begin{smallmatrix}1 & 0 \\ 0 & d \end{smallmatrix}\right) \in \GL_2(\ZZ/N\ZZ)$, we have $\theta_N(A) f = \sigma_d(f)$.
\end{alphenum}
The field $\QQ(\zeta_N)$ is the algebraic closure of $\QQ$ in $\calF_N$ and corresponds to the subgroup $\SL_2(\ZZ/N\ZZ)/\{\pm I\}$.
\end{prop}

We will sketch Propostion~\ref{P:modular galois} in \S\ref{SS:modular proof 1}.  Throughout the paper, we will let $\GL_2(\ZZ/N\ZZ)$ act on $\calF_N$ via the isomorphism $\theta_N$ (with $-I$ acting trivially).  

\begin{remark}
There are other choices for an isomorphism $\GL_2(\ZZ/N\ZZ)/\{\pm I\}$; for example, one could instead replace the transpose by an inverse in (\ref{P:modular galois a}).  Our choice is explained by our application to modular curves.   As a warning, there are several places in the literature where incompatible choices are made with respect to modular curves.
\end{remark}

\subsection{Modular curves} \label{SS:modular curves intro}

Let $G$ be a subgroup of $\GL_2(\ZZ/N\ZZ)$ containing $-I$ that satisfies $\det(G)=(\ZZ/N\ZZ)^\times$.    Denote by $\calF_N^G$ the subfield of $\calF_N$ fixed by the action of $G$ from Proposition~\ref{P:modular galois}. Using Proposition~\ref{P:modular galois} and the assumption $\det(G)=(\ZZ/N\ZZ)^\times$, we find that $\QQ$ is algebraically closed in $\calF_N^G$.  

Let $X_G$ be the smooth projective curve with function field $\calF_N^G$; it is defined over $\QQ$ and is geometrically irreducible.  The inclusion of fields $ \calF_N^G \supseteq \QQ(j)$ gives rise to a non-constant morphism 
\[
\pi_G \colon X_G \to \Spec \QQ[j] \cup \{\infty\} =\PP^1_\QQ.
\]   
The morphism $\pi_G$ is non-constant and we have 
\[
\deg(\pi_G)=[\GL_2(\ZZ/N\ZZ)/\{\pm I\}: G/\{\pm I\}]=[\GL_2(\ZZ/N\ZZ): G].
\]  
We will also denote the function field $\calF_N^G$ of $X_G$ by $\QQ(X_G)$.   A point in $X_G$ is a \defi{cusp} or a \defi{CM point} if $\pi_G$ maps it to $\infty$ or to the $j$-invariant of a CM elliptic curve, respectively.\\

The following property of the curve $X_G$ is key to our application; we will give a proof in \S\ref{SS:modular proof 2}.

\begin{prop} \label{P:main moduli}
Let $G$ be a subgroup of $\GL_2(\ZZ/N\ZZ)$ that contains $-I$ and satisfies $\det(G)=(\ZZ/N\ZZ)^\times$.    Let $E$ be an elliptic curve defined over $\QQ$ with $j_E\not\in\{0,1728\}$.   Then $\rho_{E,N}(\Gal_\QQ)$ is conjugate in $\GL_2(\ZZ/N\ZZ)$ to a subgroup of $G$ if and only if $j_E$ belongs to $\pi_G(X_G(\QQ))$.
\end{prop}

The following lemma will be key to finding modular curves of genus $0$ with rational points.

\begin{lemma} \label{L:key}
Fix a modular function $h\in \calF_N - \QQ(j)$ such that $J(h)=j$ for a rational function $J(t) \in \QQ(t)$.    Let $G$ be the subgroup of $\GL_2(\ZZ/N\ZZ)$ that fixes $h$ under the action on $\calF_N$ from Proposition~\ref{P:modular galois}. 
\begin{romanenum}
\item  \label{L:key c}
The subgroup $G$ of $\GL_2(\ZZ/N\ZZ)$ is applicable.
\item  \label{L:key a}
The modular curve $X_G$ has function field $\QQ(h)$.  In particular, it is isomorphic to $\PP^1_\QQ$. 
\item  \label{L:key b}
Let $E/\QQ$ be an elliptic curve with $j_E\notin\{0,1728\}$.   The group $\rho_{E,N}(\Gal_\QQ)$ is conjugate in $\GL_2(\ZZ/N\ZZ)$ to a subgroup of $G$ if and only if $j_E=J(t)$ for some $t\in \QQ\cup \{\infty\}$.
\end{romanenum}
\end{lemma}
\begin{proof}
By the Galois correspondence coming from the isomorphism $\theta_N$ of Proposition~\ref{P:modular galois}, the field $\QQ(h)$ equals $\calF_N^{G}$ and is an extension of $\QQ(j)$ of degree 
\[
[\GL_2(\ZZ/N\ZZ)/\{\pm I\}:G/\{\pm I\}]=[\GL_2(\ZZ/N\ZZ):G].
\]  
The field $\QQ$ is algebraically closed in $\calF_N^G = \QQ(h)$, so $\det(G)=(\ZZ/N\ZZ)^\times$ by Proposition~\ref{P:modular galois}.  Therefore, $\QQ(h)$ is the function field of $X_G$ and the field extension $\QQ(h)/\QQ(j)$ given by $j=J(h)$ corresponds to the morphism $\pi_G\colon X_G\to \PP^1_\QQ$.     The modular curve $X_G$ is thus isomorphic to $\PP^1_\QQ$ and we have $\pi_G(X_G(\QQ))=J(\QQ\cup\{\infty\})$.   This proves (\ref{L:key a}).  Part (\ref{L:key b}) follows from Proposition~\ref{P:main moduli}.

Finally, we prove that $G$ is applicable.  We have $G\neq \GL_2(\ZZ/N\ZZ)$ since the extension $\QQ(h)/\QQ(j)$ is non-trivial by our assumption on $h$.   Using part (\ref{L:key b}) and Proposition~\ref{P:basic applicable}, we find that $G$ contains an applicable subgroup.    Since $G\neq \GL_2(\ZZ/N\ZZ)$ and $G$ contains an applicable subgroup, we deduce that $G$ is applicable.
\end{proof}

If $X_G$ has genus $0$ and has rational points, then there are in fact curves $E/\QQ$ with $\pm\rho_{E,N}(\Gal_\QQ)$ conjugate to $G$.

\begin{lemma} \label{L:basic HIT}
Suppose that $X_G$ is isomorphic to $\PP^1_\QQ$; equivalently, the function field of $X_G$ is of the form $\QQ(h)$.   We have $j=J(h)$ for a unique $J(t) \in \QQ(t)$  because of the inclusion $\QQ(h)\supseteq \QQ(j)$.    Then for ``most'' $u \in \QQ$ (more precisely, outside a set of density $0$ with respect to height), the groups $\pm \rho_{E,N}(\Gal_\QQ)$ and $G$ are conjugate in $\GL_2(\ZZ/N\ZZ)$ for any elliptic curve $E/\QQ$ with $j$-invariant $J(u)$.
\end{lemma}
\begin{proof}
Let $\calG$ be the (finite) set of applicable subgroups $H$ of $\GL_2(\ZZ/N\ZZ)$ satisfying $H \subsetneq G$.   For each $H \in \calG$, let $\pi_{H,G}$ be the natural morphism $X_H \to X_G$; it has degree $[G:H] >1$.    To prove the lemma, it suffices to show that the set $\calS:=\cup_{H \in \calH} \pi_{H,G}(X_H(\QQ))$ has density $0$ (with respect to the height) in $X_G(\QQ) \cong \PP^1(\QQ)$.   This is a consequence of Hilbert irreducibility; in the language of \cite{MR1757192}*{\S9}, the set $\calS$ is \emph{thin} and hence has density $0$.  
\end{proof}

\subsection{The modular curve $X_0(N)$}
Let $X_0(N)/\QQ$ be the modular curve $X_{B(N)^t}$, where $B(N)^t$ is the transpose of $B(N)$; it consists of the lower triangular matrices and is conjugate to $B(N)$ in $\GL_2(\ZZ/N\ZZ)$.    Let $\Gamma_0(N)$ be the group of matrices in $\SL_2(\ZZ)$ whose image modulo $N$ is upper triangular.   A function $f\in \calF_N$ belongs to $\QQ(X_0(N))$ if and only if it has rational Fourier coefficients and $f\circ \gamma =f$ for all $\gamma\in \Gamma_0(N)$.
 Define the modular curve $X_s(N) := X_{C_s(N)}$, where $C_s(N)$ is the subgroup of diagonal matrices in $\GL_2(\ZZ/N\ZZ)$. 

\begin{lemma} \label{L:borel to split}
The map $\QQ(X_0(N^2)) \to \QQ(X_s(N))$, $f(\tau)\mapsto f(\tau/N)$ is an isomorphism of fields.   This isomorphism induces an isomorphism between the modular curves $X_s(N)$ and $X_0(N^2)$ which gives a bijection between their cusps.
\end{lemma}
\begin{proof}
Let $\Gamma_s(N)$ be the group of matrices in $\SL_2(\ZZ)$ whose image modulo $N$ is diagonal.   The function field of $X_s(N)$ then consist of the $f\in \calF_N$ with rational Fourier coefficients for which $f\circ \gamma = f$ for all $\gamma$ in $\Gamma_s(N)$.

Define $w=\left(\begin{smallmatrix} 1 & 0 \\0 & N \end{smallmatrix}\right)$; it acts on $\mathfrak{h}$ by linear fractional transformation, i.e., $w_*(\tau) = \tau/N$.  Take any $f\in \calF_N$ whose Fourier coefficients are rational.  We have $f \circ w$ in $\QQ(X_s(N))$ if and only if $f \circ w \circ \gamma = f\circ w$ for all $\gamma \in \Gamma_s(N)$.  Since $w \Gamma_s(N) w^{-1} = \Gamma_0(N^2)$, we deduce that $f\circ w$ belongs to $\QQ(X_s(N))$ if and only if $f$ belongs to $\QQ(X_0(N^2))$.  It is now straightforward to show that the map of fields is well-defined and an isomorphism.   The isomorphism of function fields of course induces an isomorphism of the corresponding curves.   That the cusps are in correspondence is a consequence of the map $\Gamma_0(N^2)\backslash \mathfrak{h} \to \Gamma_s(N)\backslash \mathfrak{h}$, $\tau \to w_*(\tau)=\tau/N$ being an isomorphism of Riemann surfaces.
\end{proof}

\begin{lemma}   \label{L:hauptmodul}
Let $\eta(\tau)$ be the Dedekind eta function.
\begin{romanenum}
\item \label{L:hauptmodul i}
We have $\QQ(X_0(4))=\QQ(h)$, where $h(\tau) = \eta(\tau)^8/\eta(4\tau)^8$.
\item \label{L:hauptmodul ii} 
We have $\QQ(X_0(9))=\QQ(h)$, where $h(\tau) =  \eta({\tau})^3/\eta(9\tau)^3$.
\end{romanenum}
\end{lemma}
\begin{proof}
This is well-known; for example see \cite{modular-towers}.
\end{proof}

\subsection{Proof of Proposition \ref{P:modular galois}}
\label{SS:modular proof 1}

For $\tau\in \mathfrak{H}$, let $\Lambda_\tau$ be the lattice $\ZZ \tau + \ZZ$ in $\CC$.    Set $g_2(\tau)=g_2(\Lambda_\tau)$ and $g_3(\tau)=g_3(\Lambda_\tau)$, and let $\wp(z;\tau)$ be the Weierstrass $\wp$-function relative to $\Lambda_\tau$, cf.~\cite{MR2514094}*{{\S}VI.3} for background on elliptic functions.    For each pair  $a=(a_1,a_2) \in N^{-1}\ZZ^2-\ZZ^2$, define the function
\[
f_{a}(\tau) := \frac{g_2(\tau)g_3(\tau)}{g_2(\tau)^3-27 g_3(\tau)^2} \cdot \wp(a_1\tau  + a_2; \tau)
\]
of the upper half plane.  The function $f_a$ is modular of level $N$.  Moreover, Proposition~6.9(1) of \cite{MR1291394} says that
\begin{equation} \label{E:FN gen}
\calF_N = \QQ\big(j , f_a\, \big| \, a\in N^{-1}\ZZ^2-\ZZ^2 \big).
\end{equation}
For $a,b\in N^{-1}\ZZ^2-\ZZ^2$, we have $f_a=f_b$ if and only if $a$ lies in the same coset of $\QQ^2/\ZZ^2$ as $b$ or $-b$, cf.~equation (6.1.5) of \cite{MR1291394}.  So for any $A \in M_2(\ZZ)$ with determinant relatively prime to $N$, the function $f_{aA}$ depends only on the image $\tilde{A}$ of $A$ in $\GL_2(\ZZ/N\ZZ)/\{\pm 1\}$.  By abuse of notation, we shall denote $f_{aA}$ by $f_{a\tilde{A}}$.  

By Theorem~6.6 of \cite{MR1291394}, there is a unique isomorphism 
\[
\theta_N\colon \GL_2(\ZZ/N\ZZ)/\{\pm I\} \xrightarrow{\sim} \Gal(\calF_N/\QQ(j))
\]
such that $\theta_N(A) f_a = f_{aA^t}$ for all $A\in \GL_2(\ZZ/N\ZZ)/\{\pm I\}$ and $a\in N^{-1} \ZZ^2-\ZZ^2$; we have added the transpose so the map is  a homomorphism (and not an antihomorphism).

Fix any $\gamma\in \SL_2(\ZZ)$ and let $A\in \SL_2(\ZZ/N\ZZ)$ be its image modulo $N$.   For any $a\in N^{-1} \ZZ^2-\ZZ^2$, the function $f_a \circ \gamma^t$ agrees with $f_{a\cdot \gamma^t} = \theta_N(A) f_a$ by equation (6.1.3) of \cite{MR1291394}.  Using (\ref{E:FN gen}), we deduce that $\theta_N(A) f = f \circ \gamma^t$ for all $f\in \calF_N$; this shows that (\ref{P:modular galois a}) holds.

Now take integer $d$ relatively prime to $N$ and let $A$ be the image of $\left(\begin{smallmatrix}1 & 0 \\0 & d \end{smallmatrix}\right)$ in $\GL_2(\ZZ/N\ZZ)$. Take any $a\in N^{-1}\ZZ^2-\ZZ^2$; we have $a=(r/N,s/N)$ with $r,s\in \ZZ$.   Since $f_a = f_b$ when $a\equiv b \pmod{\ZZ^2}$, we may assume that $0 \leq r <N$.  We have $\theta_N(A)f_{a}= f_{aA^t}=f_{(r/N,ds/N)}$.   
  By equation (6.2.1) of \cite{MR1291394}, we have
\begin{align*}
(2\pi)^{-2}\wp(a_1\tau +a_2 ; \tau) =& -1/12+2{\sum}_{n=1}^\infty nq^n/(1-q^n) - \zeta_N^s q^{r/N}/(1- \zeta_N^s q^{r/N})^2\\
& -{\sum}_{n=1}^\infty (\zeta_N^{ns} q^{nr/N} + \zeta_N^{-ns} q^{-nr/N}) \cdot nq^n/(1-q^n);
\end{align*}
applying $\sigma_d$ to this series gives the same thing with $s$ replaced by $ds$.   The Fourier coefficients of the expansion of $g_2(\tau)/g_3(\tau)$ are all $\pi^{-2}$ times a rational number.  Therefore, $\sigma_d(f_a)=\sigma_d(f_{(r/N,s/N)})$  equals $f_{(r/N,ds/N)} = f_{a A^t}$.   Using (\ref{E:FN gen}), we deduce that $\theta_N(A) f = \sigma_d(f)$ for all $f\in \calF_N$; this shows that (\ref{P:modular galois b}) holds.

This explains the existence of an isomorphism $\theta_N$ as in the statement of Proposition~\ref{P:modular galois}.   The uniqueness if immediate since $\GL_2(\ZZ/N\ZZ)$ is generated by $\SL_2(\ZZ/N\ZZ)$ and matrices of the form $\left(\begin{smallmatrix}1 & 0 \\0 & * \end{smallmatrix}\right)$.   Theorem~6.6 of \cite{MR1291394} implies that $\QQ(\zeta_N)$ is the algebraic closure of $\QQ$ in $\calF_N$ and that $\theta_N(A) \zeta_N = \zeta_N^{\det A}$.

\subsection{Proof of Proposition~\ref{P:main moduli}}
\label{SS:modular proof 2}
We first construct the inverse of $\theta_N$ using elliptic curves;  we shall freely use definitions from \S\ref{SS:modular proof 1}.  Let $E$ be an elliptic curve defined over an algebraically closed field $k$ of characteristic 0.   Take any non-zero $N$-torsion point $P\in E(k)$.    If $P=(x_0,y_0)$ with respect to some Weierstrass model  $y^2=4x^3  -c_2 x -c_3$ of $E/k$, define $h_E(P):= c_2 c_3/(c_2^3-27 c_3^2) \cdot x_0$.    If $j_E \notin \{0,1728\}$, then one can show that $h_E(P)$ does not depend on the choice of model.
\\

Let $\calE$ be the elliptic curve over $\calF_1=\QQ(j)$ defined by the Weierstrass equation
\begin{equation} \label{E:generic Weierstrass}
y^2= 4x^3 - \frac{27j}{j-1728} x - \frac{27j}{j-1728};
\end{equation}
it has $j$-invariant $j$.    Fix an algebraic closed field $K$ that contains $\calF_N\supseteq \QQ(j)$ and let $\calE[N]$ be the $N$-torsion subgroup of $\calE(K)$.  

\begin{lemma} \label{L:basis P1P2}
There is a basis $\{P_1, P_2\}$ of the $\ZZ/N\ZZ$-module $\calE[N]$ such that $h_\calE( r P_1 +s P_2) = f_{(r/N,s/N)}$ for all $(r,s)\in \ZZ^2-N \ZZ^2$.
\end{lemma}
\begin{proof}
Let $K_0$ be the extension of $\calF_N$ generated by the functions $g_2(\tau)$, $g_3(\tau)$, $\wp({\tau}/{N}; \tau)$, $\wp'(\tau/N;\tau)$,  $\wp({1}/{N}; \tau)$ and $\wp'(1/N;\tau)$.    We may assume that $K \supseteq K_0$.   Let $E$ be the elliptic curve over $K_0$ defined by $y^2=4x^3 -g_2(\tau) x -g_3(\tau)$; its $j$-invariant is $j=j(\tau)$.   The curves $E$ and $\calE$ are isomorphic over $K$ since they both have $j$-invariant $j$.  Since $j\notin\{0,1728\}$, it suffices to prove the lemma for $E$ instead of $\calE$.

Define the pairs
\[
P_1 := (\wp({\tau}/{N}; \tau), \wp'(\tau/N;\tau) ) \quad \text{ and }\quad P_2 := (\wp({1}/{N}; \tau), \wp'(1/N;\tau) ).
\]
We claim that $P_1$ and $P_2$ form a basis of the $\ZZ/N\ZZ$-module of $N$-torsion in $E(K)$.  To prove the claim it suffices to prove the analogous results  after specializing the coefficients of $E$ and the entries of $P_1$ and $P_2$ by an arbitrary $\tau_0 \in \mathfrak{h}$  (since the claim comes down to verifying certain polynomial equations whose variables are the coefficients of the model of $E$ and the entries of the points).   So fix an arbitrary $\tau_0 \in \mathfrak{h}$.  Specializing the model of $E$ at $\tau_0$ gives an elliptic curve $E_{\tau_0}$ over $\CC$ defined by $y^2=4x^3 -g_2(\tau_0) x -g_3(\tau_0)$.  From  Weierstrass, we know that the map
\[
\CC/\Lambda_{\tau_0} \to E_{\tau_0}(\CC),\quad z\mapsto (\wp(z;\tau_0),\wp'(z;\tau_0)),
\]
with $0$ mapping to the point at infinity, gives an isomorphism of complex Lie groups.   In particular, the points $P_{1,\tau_0}= (\wp({\tau_0}/{N}; \tau_0), \wp'(\tau_0/N;\tau_0) )$ and $P_{2,\tau_0}= (\wp({1}/{N}; \tau_0), \wp'(1/N;\tau_0))$ give a basis for the $N$-torsion in $E_{\tau_0}(\CC)$.  This is enough to prove our claim.   Moreover, we have  $r P_{1,\tau_0} + s P_{2,\tau_0} = (\wp(r/N\cdot  \tau_0 + s/N;\tau_0),\wp'(r/N \cdot \tau_0 + s/N;\tau_0))$ for all $(r,s)\in \ZZ^2-N\ZZ^2$.  Therefore, 
\[
h_{E_{\tau_0}}(rP_1+sP_2) = g_2(\tau_0) g_3(\tau_0)/(g_2(\tau_0)^3-27g_3(\tau_0)^2) \cdot \wp(r/N \cdot \tau_0 + s/N;\tau_0) = f_{(r/N,s/N)}(\tau_0).
\]
for all $(r,s)\in \ZZ^2-N\ZZ$.  Since this holds for all $\tau_0\in \mathfrak{h}$, we deduce that $h_{E}( r P_1 +s P_2 ) = f_{(r/N,s/N)}$.
\end{proof}

 Let  $\rho_N \colon \Gal(K/\QQ(j)) \to \GL_2(\ZZ/N\ZZ)$ be the representation describing the Galois action on $\calE[N]$ with respect to the  basis $\{P_1,P_2\}$ of Lemma~\ref{L:basis P1P2}.  The fixed field of the kernel of $\Gal(K/\QQ(j)) \xrightarrow{\rho_N} \GL_2(\ZZ/N\ZZ)\to \GL_2(\ZZ/N\ZZ)/\{\pm I\}$ is generated by the $x$-coordinates of the non-zero points in $\calE[N]$.  By (\ref{E:FN gen}) and Lemma~\ref{L:basis P1P2}, the extension $\calF_N$ of $\QQ(j)$ is generated by the $x$-coordinates of the non-zero points in $\calE[N]$.   Therefore, the representation $\rho_N$ induces an injective homomorphism
\begin{equation} \label{E:Galois FN}
\bbar\rho_N \colon \Gal(\calF_N/\QQ(j)) \hookrightarrow \GL_2(\ZZ/N\ZZ)/\{\pm I\}.
\end{equation}
In fact, (\ref{E:Galois FN}) is an isomorphism since the groups have the same cardinality by Proposition~\ref{P:modular galois}.  

\begin{lemma} \label{L:inverse of thetaN}
The homomorphism $\bbar\rho_N$ is an isomorphism.  Moreover, the inverse of $\bbar\rho_N$ is the homomorphism  $\theta_N\colon  \GL_2(\ZZ/N\ZZ)/\{\pm I\} \to \Gal(\calF_N/\QQ(j))$.
\end{lemma}
\begin{proof}
Take any $\sigma\in \Gal(K/\QQ(j))$ and set $\tilde\sigma:=\sigma|_{\calF_N}$.    There are integers $a,b,c,d\in \ZZ$ such that $\sigma(P_1)=a P_1 + cP_2$ and $\sigma(P_2)=bP_1+dP_2$, so $\rho_N(\sigma)=A$, where $A\in \GL_2(\ZZ/N\ZZ)$ is the image of $\left(\begin{smallmatrix}a & b \\c & d \end{smallmatrix}\right)$ modulo $N$.   Therefore, $\bbar\rho_N(\tilde\sigma)$ is the class of $A$ in $\GL_2(\ZZ/N\ZZ)/\{\pm I\}$.   We need to show that $\theta_N(A)=\tilde\sigma$.

 Take any pair of integers $(r,s)\in \ZZ^2-N\ZZ^2$.  We have
\[
\sigma(rP_1 +sP_2)=r\sigma(P_1) +s \sigma(P_2)= (ra+sb)P_1 + (rc+sd)P_2.
\]
Comparing $x$-coordinates and using Lemma~\ref{L:basis P1P2}, we find that $\tilde\sigma(f_{(r/N,s/N)})=\sigma(f_{(r/N,s/N)})$ is equal to $f_{((ra+sb)/N,(rc+sd)/N)} = f_{(r/N,s/N) A^t}$ which is $\theta_N(A) f_{(r/N,s/N)}$ from \S\ref{SS:modular proof 1}.   Since the extension $\calF_N/\QQ(j)$ is generated by the functions $f_a$ with $a\in \ZZ^2-N\ZZ^2$, we deduce that $\theta_N(A)=\tilde\sigma$.  
\end{proof}

Define the $\QQ$-variety
\[
U := \AA^1_\QQ-\{0,1728\} = \Spec \QQ[j,j^{-1},(j-1728)^{-1}];
\] 
note that we are now viewing $j$ as simply a transcendental variable.   The equation (\ref{E:generic Weierstrass}) defines a (relative) elliptic curve $\pi\colon \scrE \to U$.   The fiber of $\scrE\to U$ over the generic fiber of $U$ is the elliptic curve $\calE/\QQ(j)$. 

Let $\bbar{\eta}$ be the geometric generic point of $U$ corresponding to the algebraically closed extension $K$ of $\calF_N$.  Let $\scrE[N]$ be the $N$-torsion subscheme of $\scrE$.  We can identify the fiber of $\scrE[N]\to U$ at $\bbar\eta$ with the group $\calE[N]$.  Let $\pi_1(U,\bbar{\eta})$ be the \'etale fundamental group of $U$.  We can view $\scrE[N]$ as a rank $2$ lisse sheaf of $\ZZ/N\ZZ$-modules $U$ and it hence corresponds to a representation  
\[
\varrho_N \colon \pi_1(U,\bbar{\eta}) \to \Aut(\calE[N])\cong \GL_2(\ZZ/N\ZZ)
\] 
where the isomorphism uses the basis $\{P_1,P_2\}$ of Lemma~\ref{L:basis P1P2}.   Taking the quotient by the group generated by $-I$, we obtain a homomorphism 
\[
\bbar\varrho_N\colon \pi_1(U,\bbar\eta) \to \GL_2(\ZZ/N\ZZ)/\{\pm I\}.
\]     
Note that the representation $\Gal(K/\QQ(j))\to \GL_2(\ZZ/N\ZZ)/\{\pm I\}$ coming from $\bbar\varrho_N$ factors through the homomorphism $\bbar\rho_N$.  So by Proposition~\ref{P:modular galois} and Lemma~\ref{L:inverse of thetaN},  the representation $\bbar\varrho_N$ is surjective and satisfies $\bbar\varrho_N(\pi_1(U_{\Qbar}))=\SL_2(\ZZ/N\ZZ)/\{\pm I\}$. 

Now take any subgroup $G$ of $\GL_2(\ZZ/N\ZZ)$ that satisfies $-I\in G$ and $\det(G)=(\ZZ/N\ZZ)^\times$.    Using $\bbar\varrho_N$, the group $G/\{\pm I\}$ corresponds to an \'etale morphism $\pi\colon Y_G\to U$.   The smooth projective closure of $Y_G$ is thus $X_G$ and the morphism $X_G\to \PP^1_\QQ$ arising from $\pi$ is simply $\pi_G$.
  
Take any rational point $u\in U(\QQ)=\QQ-\{0,1728\}$.    Viewed as a morphism $\Spec \QQ \to U$, the point $u$ induces a homomorphism $u_*\colon \Gal_\QQ \to \pi_1(U)$; we are suppressing base points so everything is uniquely defined only up to conjugacy.  Composing $u_*$ with $\bbar\varrho_N$ we obtain a homomorphism $\beta_u\colon \Gal_\QQ \to \GL_2(\ZZ/N\ZZ)/\{\pm I\}$.     Observe that the group $\beta_u(\Gal_\QQ)$ is conjugate to a subgroup of $G/\{\pm I\}$ if and only if $u$ lies in $\pi_1(Y_G(\QQ))=\pi_G(X_G(\QQ))-\{0,1728,\infty\}$.

The fiber of $\scrE\to U$ over $u$ is the elliptic curve $\scrE_u/\QQ$ obtained by setting $j$ to $u$ in (\ref{E:generic Weierstrass}).   Composing $\rho_{\scrE_u,N}\colon \Gal_\QQ \to \GL_2(\ZZ/N\ZZ)$ with the quotient map $\GL_2(\ZZ/N\ZZ)\to \GL_2(\ZZ/N\ZZ)/\{\pm I\}$  gives a homomorphism that agrees with $\beta_u$ up to conjugation.   Since $-I \in G$, we find that $\rho_{\scrE_u,N}(\Gal_\QQ)$ is conjugate in $\GL_2(\ZZ/N\ZZ)$ to a subgroup of $G$ if and only if $u \in \pi_G(X_G(\QQ))$.

Finally, let $E/\QQ$ be any elliptic curve with $j$-invariant $u$.   The curve $\scrE_u/\QQ$ also has $j$-invariant $u$.    As noted in the introduction, since $E$ and $\scrE_u$ are elliptic curves over $\QQ$ with common $j$-invariant $u\notin \{0,1728\}$, the groups $\pm \rho_{E,N}(\Gal_\QQ)$ and $\pm \rho_{\scrE_u,N}(\Gal_\QQ)$ must be conjugate.   This completes the proof of Proposition~\ref{P:main moduli}.

\section{Classification up to a sign}  \label{S:main classification}

In this section, we prove the parts of the theorems of \S\ref{S:classification} that involve the groups $\pm \rho_{E,\ell}(\Gal_\QQ)$ for an elliptic curve $E/\QQ$.  In the notation of \S\ref{SS:applicable}, the group $\pm \rho_{E,\ell}(\Gal_\QQ)$ is either applicable or is the full group $\GL_2(\FF_\ell)$.   We consider the primes $\ell$ separately and keep the notation of the relevant subsection of \S\ref{S:classification}.   

One of the main tasks is to construct modular curves of genus $0$.  We will do this by finding functions $h\in \calF_\ell-\QQ(j)$ such that $j=J(h)$ for some $J\in \QQ(t)$.  Let $H$ be the subgroup of $\GL_2(\FF_\ell)$ consisting of elements that fix $h$ under the action from Proposition~\ref{P:modular galois}.    By Lemma~\ref{L:key}, the group $H$ is an applicable subgroup of $\GL_2(\FF_\ell)$.   Furthermore, $X_H$ has function field $\QQ(h)$ and the morphism $\pi_H\colon X_H\to \PP^1_\QQ$ is described by the inclusion $\QQ(h) \supseteq \QQ(j)$.   So if $E/\QQ$ is a non-CM elliptic curve, then $\rho_{E,\ell}(\Gal_\QQ)$ is conjugate to a subgroup of $H$ if and only if $j_E$ belongs to $\pi_H(X_H(\QQ))=J(\QQ\cup \{\infty\})$.  

  We will need to recognize $H$ as a conjugate of one of our applicable subgroups $G_i$ of $\GL_2(\FF_\ell)$.   The degree of $\pi_H$, which is the same as the degree of $J(t)$, is equal to the index $[\GL_2(\FF_\ell):H]$; this observation will immediately rule out most candidates.   We will also make use of Proposition~\ref{P:main moduli}; observe that the set $\pi_H(X_H(\QQ))$ depends only on the conjugacy class of $H$. 
  
Most of this section involves basic algebraic verifications (which are straightforward to check with a computer, see the link in \S\ref{SS:overview} for many such details); much of the work, which we will not touch on, is finding the various equations in the first place.

\subsection{$\ell=2$}

Fix notation as in \S\ref{SS:applicable 2}.   Up to conjugacy, $G_1$, $G_2$ and $G_3$ are the proper subgroups of $\GL_2(\FF_2)$.
\begin{itemize}
\item
Define the function
\begin{align*}
h_1(\tau) &:= 16\eta(2\tau)^8/\eta(\tau/2)^8 = 16\big( q^{1/2} + 8q + 44q^{3/2} + 192q^2 + 718q^{5/2}+\cdots \big).
\end{align*}
By Lemmas~\ref{L:borel to split} and \ref{L:hauptmodul}(\ref{L:hauptmodul i}), we have $\QQ(X_s(2))=\QQ(h_1)$.   We have $C_s(2)=G_1$, so $\QQ(X_{G_1})=\QQ(h_1)$.   The extension $\QQ(h_1)/\QQ(j)$ has degree $6$, so there is a unique rational function $J(t) \in \QQ(t)$ such that $j=J(h_1)$.  We have $J(t)= f_1(t)/f_2(t)$ for relatively prime $f_1,f_2 \in \QQ[t]$ of degree at most $6$.   Expanding the $q$-expansion of $j f_2(h_1) - f(h_1) = J(h_1) f_2(h_1) - f_1(h_1)=0$ gives many linear equations in the coefficients of $f_1$ and $f_2$.   Using enough terms of the $q$-expansion, we can compute the coefficients of $f_1$ and $f_2$ (they are unique up to scaling $f_1$ and $f_2$ by some constant in $\QQ^\times$).  Doing this, we found that $J_1(h_1)=j$.

\item
Define $h_2:=h_1^2/(h_1 + 1)$.  Since $J_2(t^2/(t + 1))=J_1(t)$, we have $J_2(h_2)=j$.   

\item
Define $h_3:=F(h_1)$ where $F(t)=(-16t^3 - 24t^2 + 24t + 16)/(t^2 + t)$.  Since $J_3(F(t))=J_1(t)$, we have $J_3(h_3)=j$.
\end{itemize}

For each integer $1\leq i \leq 3$, let $H_i$ be the subgroup of $\GL_2(\FF_2)$ that fixes $h_i$.   By Lemma~\ref{L:key},  $H_i$ is an applicable subgroup of $\GL_2(\FF_2)$ with index equal to the degree of $J_i(t)$.  By comparing the degree of $J_i(t)$ with our list of proper subgroups, we deduce that $H_i$ is conjugate to $G_i$ in $\GL_2(\FF_2)$.

Theorem~\ref{T:main2} now follows from Lemma~\ref{L:key}(\ref{L:key b}); we can ignore $t=\infty$ since $J_i(\infty)=\infty$.

\subsection{$\ell=3$}  

Fix notation as in \S\ref{SS:applicable 3}.  Up to conjugacy, the groups $G_i$ with $1\leq i\leq 4$ are the applicable subgroups of $\GL_2(\FF_3)$.   
\begin{itemize}
\item 
Define the function $h_1 := 1/3\cdot \eta({\tau/3})^3/\eta(3\tau)^3$.  By Lemmas~\ref{L:borel to split} and \ref{L:hauptmodul}(\ref{L:hauptmodul ii}), we have $\QQ(X_s(3))=\QQ(h_1)$.   We have $C_s(3)=G_1$, so $\QQ(X_{G_1})=\QQ(h_1)$.   The extension $\QQ(h_1)/\QQ(j)$ has degree $12$, so there is a unique rational function $J(t) \in \QQ(t)$ such that $j=J(h_1)$.  We have $J(t)= f_1(t)/f_2(t)$ for relatively prime $f_1,f_2 \in \QQ[t]$ of degree at most $12$.   Expanding the $q$-expansion of $j f_2(h_1) - f(h_1) = J(h_1) f_2(h_1) - f_1(h_1)=0$ gives many linear equations in the coefficients of $f_1$ and $f_2$.   Using enough terms of the $q$-expansion, we can compute the coefficients of $f_1$ and $f_2$ (they are unique up to scaling $f_1$ and $f_2$ by some constant in $\QQ^\times$).  Doing this, we found that $J_1(h_1)=j$.

\item
Define $h_2=F_1(h_1)$ where $F_1(t)=(t^2 + 3t + 3)/t$.   Since $J_2(F_1(t))=J_1(t)$, we have $J_2(h_2)=j$.
\item
Define $h_3=F_2(h_1)$ where $F_2(t)=t(t^2 + 3t + 3)$.  Since $J_3(F_2(t))=J_1(t)$, we have $J_3(h_3)=j$.
\item
Define $h_4= F_3(h_2)$ where $F_3(t)= 3(t+1)(t-3)/t$.  Since $J_4(F_3(t))=J_2(t)$, we have $J_4(h_4)=j$.
\end{itemize}

Fix an integer $1\leq i \leq 4$, and let $H_i$ be the subgroup of $\GL_2(\FF_3)$ that fixes $h_i$.   By Lemma~\ref{L:key}, we find that $H_i$ is an applicable subgroup and the morphism $\pi_{H_i}\colon X_{H_i} \to  \PP^1_\QQ$ is described by the rational function $J_i(t)$.    The index $[\GL_2(\FF_3): H_i]$ agrees with the degree of $J_i(t)$.  By comparing the degree of $J_i(t)$ with our list of applicable subgroups, we deduce that $H_i$ is conjugate to $G_i$ in $\GL_2(\FF_3)$.

Theorem~\ref{T:main3}(\ref{T:main3 b}) now follows from Lemma~\ref{L:key}(\ref{L:key b}); we can ignore $t=\infty$ since $J_i(\infty)=\infty$.  A computation shows that if $H$ is a proper subgroup of $G_i$ satisfying $\pm H=G_i$, then $i \in \{1,3\}$ and $H$ is one of the groups $H_{i,j}$; this proves Theorem~\ref{T:main3}(\ref{T:main3 a}).

\subsection{$\ell=5$}
Fix notation as in \S\ref{SS:applicable 5}.  Up to conjugacy, the applicable subgroups of $\GL_2(\FF_5)$ are the groups $G_i$ with $1\leq i \leq 9$. Recall that the \defi{Rodgers-Ramanunjan continued fraction} is 
\begin{align*}
r(\tau) &:=q^{1/5}\cdot \frac{1}{1+}\, \frac{q}{1+} \, \frac{q^2}{1+}\, \frac{q^3}{1+}\, \frac{q^4}{1+} \cdots. 
\end{align*}
The function 
\[
h_1(\tau):=1/r(\tau) = q^{-1/5}(1 + q - q^3 + q^5 + q^6 - q^7 - 2q^8 + 2q^{10} + 2q^{11} +\cdots)
\] 
is a modular function of level $5$ and satisfies $J_1(h_1)=j$; we refer to Duke \cite{MR2133308} for an excellent exposition.  An expression for $h_1(\tau)$ in terms of Klein forms can be found in \cite{MR2264315}.  

Set $w:=(1+\sqrt{5})/2 \in \QQ(\zeta_5)$.
\begin{itemize}
\item
Define the function $h_2=h_1 - 1- 1/h_1$.  We have $J_2(t-1-1/t)=J_1(t)$, so $J_2(h_2)=j$.  (As noted in equation (7.2) of \cite{MR2133308}, $h_2$ equals $\eta(\tau/5)/\eta(5\tau)$.)  
\item
  Define $h_3=F_1(h_2)$ where
\[
F_1(t)=\frac{(- 3+w)t - 5}{t + (3-w)}.  
\]
We have $J_3(F_1(t))=J_2(t)$ and hence $J_3(h_3)=j$.
\item
Define $h_4=h_2+5/h_2$.   We have $J_4(t+5/t)=J_2(t)$ and hence $J_4(h_4)=j$.   
\item
Define $h_5=h_1^5$.   We have $J_5(t^5)=J_1(t)$ and hence $J_5(h_5)=j$.
\item
Define $h_6= F_2(h_5)$ where
\[
F_2(t) =\frac{-(w-1)^5t+1}{t+(w-1)^5}
\]
We have $J_6(F_2(t))=J_5(t)$ and hence $J_6(h_6)=j$.  (In the notation of \cite{MR2133308}*{\S8}, we have $b=h_6$.)
\item
Define $h_7=F_3(h_3)$ where
\[
F_3(t)=-\frac{t^3 + 10t^2 + 25t + 25}{2t^3 + 10t^2 + 25t + 25}.   
\]
We have $J_7(F_3(t))=J_3(t)$ and hence $J_7(h_7)=j$.
\item
Define $h_8=h_5-11-h_5^{-1}$.    We have $J_8(t-11-t^{-1})=J_5(t)$ and hence $J_8(h_8)=j$.  (As noted in equation (7.7) of \cite{MR2133308}, $h_8$ equals $(\eta(\tau)/\eta(5\tau))^6$.)
\item
Define $h_9=F_4(h_4)$ where
\[
F_4(t)= \frac{(t+5)(t^2-5)}{t^2 + 5t + 5}.
\]
We have $J_9(F_4(t))=J_4(t)$ and hence $J_9(h_9)=j$.
\end{itemize}

Fix an integer $1\leq i \leq 9$.    Let $H_i$ be the subgroup of $\GL_2(\FF_5)$ that fixes $h_i$.   By Lemma~\ref{L:key}, we find that $H_i$ is an applicable subgroup and the morphism $\pi_{H_i}\colon X_{H_i} \to  \PP^1_\QQ$ is described by the rational function $J_i(t)$.   

\begin{lemma}
The groups $H_i$ and $G_i$ are conjugate in $\GL_2(\FF_5)$ for each $1\leq i \leq 9$.
\end{lemma}
\begin{proof}
 The index $[\GL_2(\FF_5): H_i]$ agrees with the degree of $J_i(t)$.  By comparing the degree of $J_i(t)$ with our list of applicable subgroups, we deduce that $H_i$ is conjugate to $G_i$ in $\GL_2(\FF_5)$ for all $i\in \{1,4,7,8,9\}$.    

The groups $H_5$ and $H_6$ are not conjugate since one can check that the image of $\PP^1(\QQ)=\QQ\cup\{\infty\}$ under $J_5(t)$ and $J_6(t)$ are different.  The groups $H_5$ and $H_6$ have index $12$ in $\GL_2(\FF_5)$ and are not conjugate, so they are conjugate to $G_5$ and $G_6$ (though we need to determine which is which).   The elliptic curve given by the Weierstrass equation $y^2+(1-t)xy -ty = x^3-tx^2$ has $j$-invariant $J_6(t)$ and the point $(0,0)$ has order $5$.  Therefore, $H_6$ is conjugate to $G_6$ and thus $H_5$ is conjugate to $G_5$.

The groups $H_2$ and $H_3$ are not conjugate since one can check that the image of $\PP^1(\QQ)=\QQ\cup\{\infty\}$ under $J_2(t)$ and $J_3(t)$ are different.  The groups $H_2$ and $H_3$ have index $30$ in $\GL_2(\FF_5)$ and are not conjugate, so they are conjugate to $G_2$ and $G_3$ (though we need to determine which is which).  Since $h_7=F_3(h_3)$ and $F_3(t)$ belongs to $\QQ(t)$, we find that $H_3$ is a subgroup of $H_7$.   We already know that $H_7$ is conjugate to $N_{ns}(5)$ and one can check that $G_2=C_s(5)$ is not conjugate to a subgroup of $N_{ns}(5)$.  Therefore, $H_3$ is conjugate to $G_3$ and thus $H_2$ is conjugate to $G_2$.
\end{proof}

Theorem~\ref{T:main5}(\ref{T:main5 b}) now follows from Lemma~\ref{L:key}(\ref{L:key b}); we have $J_i(\infty)=\infty$ for $i\notin\{3,7\}$ and we can ignore the values $J_3(\infty)=0$ and $J_7(\infty)=8000$ since they are the $j$-invariants of CM elliptic curves.   A direct computation shows that if $H$ is a proper subgroup of $G_i$ satisfying $\pm H=G_i$, then $i \in \{1,5,6\}$ and $H$ is one of the groups $H_{i,j}$; this proves Theorem~\ref{T:main5}(\ref{T:main5 a}).

\subsection{$\ell=7$}  \label{SS:main proof 7}

Fix notation as in \S\ref{SS:applicable 7}.  Up to conjugacy, the applicable subgroups of $\GL_2(\FF_7)$ are the groups $G_i$ with $1\leq i\leq 7$ from \S\ref{SS:applicable 7} and the groups:
\begin{itemize}
\item 
Let $G_8$ be the subgroup of $\GL_2(\FF_7)$ consisting of matrices of the form $\pm \left(\begin{smallmatrix}1 & 0 \\0 & * \end{smallmatrix}\right)$.
\item
Let $G_9$ be the subgroup of $\GL_2(\FF_7)$ consisting of matrices of the form $\left(\begin{smallmatrix}a & 0 \\0 & \pm a \end{smallmatrix}\right)$.
\item
Let $G_{10}$ be the subgroup of $\GL_2(\FF_7)$ generated by $\left(\begin{smallmatrix}0 & -2 \\2 & 0 \end{smallmatrix}\right)$ and $\left(\begin{smallmatrix}1 & 0 \\0 & -1 \end{smallmatrix}\right)$.
\item 
Let $G_{11}$ be the subgroup $C_s(7)$ of $\GL_2(\FF_7)$.
\item 
Let $G_{12}$ be the subgroup of $\GL_2(\FF_7)$ generated by $\left(\begin{smallmatrix}1 & -1 \\1 & 1 \end{smallmatrix}\right)$ and $\left(\begin{smallmatrix}1 & 0 \\0 & -1 \end{smallmatrix}\right)$.
\end{itemize}
For $i=8, 9, 10, 11$ and $12$, the index $[\GL_2(\FF_7): G_i]$ is $168$, $168$, $84$, $56$ and $42$, respectively.\\

The \defi{Klein quartic} is the curve $\calX$ in $\PP^2_\QQ$ defined by the equation $x^3y+y^3z+z^3x=0$; it is a non-singular curve of genus $3$.  The relevance to us is that $\calX$ is isomorphic to the modular curve $X(7):=X_{G_8}$; we refer to Elkies \cite{MR1722413} for a lucid exposition.    In \S4 of \cite{MR1722413}, Elkies defines a convenient basis $\defi{x}$, $\defi{y}$ and $\defi{z}$ for the space of cusp forms of $\Gamma(7)$ which satisfy the equation of the Klein quartic and have product expansions
\[
\defi{x},\, \defi{y},\, \defi{z} =\varepsilon q^{a/7} \prod_{n=1}^\infty (1-q^n)^3 (1-q^7)  \prod_{\substack{n>0\\n\equiv \pm n_0 \,\bmod{7}}} (1-q^n)
\]
where $(\varepsilon,a,n_0)$ is $(-1,4,1)$, $(1,2,2)$ or $(1,1,4)$ for $\defi{x}$, $\defi{y}$ or $\defi{z}$, respectively.   The coordinates $(\defi{x}:\defi{y}\colon \defi{z})$ then give the desired isomorphism $X(7)\to\calX$.   \\

Define 
\[
h_4:=-(\defi{y}^2\defi{z})/\defi{x}^3=q^{-1} + 3 + 4q + 3q^2 - 5q^4 - 7q^5 + \ldots; 
\]
it is a modular function of level $7$.  Define $h_7:=F_1(h_4)$ where
\[
F_1(t)=t+\frac{1}{1-t}+\frac{t-1}{t}-8.
\]
From equations (4.20) and (4.24) of \cite{MR1722413}, with a correction in the sign of (4.23) of loc.~cit., we have $J_7(h_7)=j$.   
Since $J_7(F_1(t))=J_4(t)$, we have $J_4(h_4)=j$.   

Define $h_3:=F_2(h_4)$ and $h_5:=F_3(h_4)$, where
\[
F_2(t)= \frac{\beta t - (\beta-1)}{t-\beta}\quad \text{ and }\quad F_3(t)=\frac{ t - \gamma}{\gamma t-(\gamma-1)}
\]
with $\beta=4+3\zeta_7+3\zeta_7^{-1}+\zeta_7^2+\zeta_7^{-2}$ and $\gamma= \zeta_7^5 + \zeta_7^4 + \zeta_7^3 + \zeta_7^2 + 1$.  Since $J_3(F_2(t))=J_4(t)$ and $J_5(F_3(t))=J_4(t)$, we have $J_3(h_3)=j$ and $J_5(h_5)=j$.\\

For $i \in \{3,4,5,7\}$, let $H_i$ be the subgroup of $\GL_2(\FF_7)$ that fixes $h_i$.   We have shown that $J_i(h_i)=j$.  By Lemma~\ref{L:key}, we find that $H_i$ is an applicable subgroup and that the morphism $\pi_{H_i}\colon X_{H_i} \to  \PP^1_\QQ$ is described by the rational function $J_i(t)$.  

\begin{lemma}
The groups $H_i$ and $G_i$ are conjugate in $\GL_2(\FF_7)$ for all $i\in\{3,4,5,7\}$.
\end{lemma}
\begin{proof}
  The index of $H_i$ in $\GL_2(\FF_7)$ agrees with the degree of $J_i(t)$ which is $24$ or $8$ if $i\in\{3,4,5\}$ or $i=7$, respectively.   By our list of applicable subgroups, we deduce that $H_7$ is conjugate to $G_7$ in $\GL_2(\FF_7)$.  The groups $H_3$, $H_4$ and $H_5$ are not conjugate in $\GL_2(\FF_7)$ (since one can show that the images of $\PP^1(\QQ)=\QQ \cup\{\infty\}$ under $J_3$, $J_4$ and $J_5$ are pairwise distinct).    By our list of applicable subgroups, the groups $H_3$, $H_4$ and $H_5$ are conjugate to the three subgroups $G_3$, $G_4$ and $G_5$; we still need to identify $H_3$ with $G_3$, etc.

The modular function $h_4\in \calF_7$ is a Laurent series in $q$ and has rational coefficients.  Using Proposition~\ref{P:modular galois}, this implies that $H_4$ contains the group of matrices of the form $\pm \left(\begin{smallmatrix}1 & 0 \\ * & * \end{smallmatrix}\right)$ in $\GL_2(\FF_7)$.   Therefore, $H_4$ must be conjugate to $G_4$   in $\GL_2(\FF_7)$.  The elliptic curve given by the Weierstrass equation $y^2+(1+t-t^2)xy+(t^2-t^3)y = x^3+(t^2-t^3)x^2$ has $j$-invariant $J_3(t)$ and the point $(0,0)$ has order $7$, so $H_3$ is conjugate to $G_3$.    Therefore, $H_5$ is conjugate to $G_5$.
\end{proof}

Following Elkies (\cite{MR1722413}*{p.68}), we multiply the equation of the Klein curve to obtain $(\defi{x}^3\defi{y}+\defi{y}^3\defi{z}+\defi{z}^3\defi{x})(\defi{x}^3\defi{z}+\defi{z}^3\defi{y}+\defi{y}^3\defi{x})=0.$  Noting that the left hand side is a symmetric polynomial in $\defi{x}$, $\defi{y}$ and $\defi{z}$, one can show that $s_2^4+s_3(s_1^5-5s_1^3 s_2 +s_1 s_2^2+7 s_1^2 s_3) =0$ where $s_1=\defi{x}+\defi{y}+\defi{z}$, $s_2=\defi{x}\defi{y} + \defi{y}\defi{z}+\defi{z}\defi{x}$ and $s_3=\defi{x}\defi{y}\defi{z}$.  We now deviate from Elkies' treatment.   Divide by $s_1^2 s_3^2$ and rearrange to obtain 
\[
\Big(\frac{s_2^2}{s_1 s_3}\Big)^2 + \Big(\frac{s_1^2}{s_2}\Big)^2 \cdot \frac{s_2^2}{s_1 s_3} - 5 \frac{s_1^2}{s_2} \cdot \frac{s_2^2}{s_1 s_3} +   \frac{s_2^2}{s_1 s_3} + 7 =0.
\]
We thus have $v^2 + (h_2^2  -5 h_2  + 1)v +7 = 0$, where 
\[
h_2:= s_1^2/s_2=q^{-1/7} + 2 + 2q^{1/7} + q^{2/7} + 2q^{3/7} + 3q^{4/7} + 4q^{5/7} + 5q^{6/7} + 7q + 8q^{8/7}+\cdots 
\] 
and $v:=s_2^2/(s_1 s_3)$ are modular functions.  We claim that
\begin{equation} \label{E:h7 expression}
h_7+(h_2^3 - 4h_2^2 + 3h_2 + 1)((h_2^2 - 5h_2 + 1)v + 7) = 0.
\end{equation}
This can be verified algebraically:   In the left-hand side of (\ref{E:h7 expression}), replace $h_7$ by $F_1(-y^2z/x^3)$, $h_2$ by $(x+y+z)^2/(xy+yz+zx)$, and $v$ by $(xy+yz+zx)^2/((x+y+z)xyz)$; the numerator of the resulting rational function is divisible by $xy^3+yz^3+zx^3$.

Completing the square in the equation $v^2 + (h_2^2  -5 h_2  + 1)v +7 = 0$, we have 
\begin{equation} \label{E:w2 expression}
w^2=   h_2^4 - 10h_2^3 + 27h_2^2 - 10h_2 - 27,
\end{equation}
where $w:=2v+(h_2^2  -5 h_2  + 1)$.  From (\ref{E:h7 expression}), we find that
\begin{equation} \label{E:h7 expression 2}
h_7 = \tfrac{1}{2} (h_2^3-4h_2^2+3h_2+1)((h_2^4-10h_2^3+27h_2^2-10h_2-13)-(h_2^2-5h_2+1)w).
\end{equation}
We have $j=J_7(h_7)$, so  (\ref{E:w2 expression}) and (\ref{E:h7 expression 2}) imply that $j$ can be written in the form $\alpha(h_2)+ \beta(h_2)w$ for rational functions $\alpha(t)$ and $\beta(t)$.   A direct computation shows that $\alpha(t)=J_2(t)$ and $\beta(t)=0$, and hence $J_2(h_2)=j$.

Let $H_2$ be the subgroup of $\GL_2(\FF_7)$ that fixes $h_2$.   We have  $J_2(h_2)=j$, so Lemma~\ref{L:key} implies that $H_2$ is an applicable subgroup and that the morphism $\pi_{H_2}\colon X_{H_2} \to  \PP^1_\QQ$ is described by the rational function $J_2(t)$.     The index of $H_2$ in $\GL_2(\FF_7)$ is $28$ since it agrees with the degree of $J_2(t)$.    By our list of applicable subgroups, we deduce that $H_2$ is conjugate to $G_2$ in $\GL_2(\FF_7)$.\\

Let $H_{11}$ be the subgroup of $\GL_2(\FF_7)$ that fixes $h_2$ and $w$.    The group $H_{11}$ is an index $2$ subgroup of $H_2$ since the extension $\QQ(h_2,w)/\QQ(h_2)$ has degree $2$.  The group $H_{11}$ contains $G_8$ since $\QQ(h_2,w)$ is contained in $\QQ(\defi{x}/\defi{z},\defi{y}/\defi{z})$ which is the function field of $X(7)$; in particular, $H_{11}$ is applicable.   From our classification of applicable subgroups, we find that $H_{11}$ is conjugate to $G_{11}$.   The modular curve $X_{G_{11}}$ thus has function field $\QQ(h_2,w)$ and is hence isomorphic to the smooth projective curve over $\QQ$ with affine model
\begin{equation} \label{E:XH11}
y^2= x^4 - 10x^3 + 27x^2 - 10x - 27.
\end{equation}
The only rational points for the smooth model of (\ref{E:XH11}) are the two points at infinity (one can show that it is isomorphic to the quadratic twist by $-7$ of the curve $E_{7,1}$ from \S\ref{SS:CM}, and that this curve has only two rational points).  Using that $J_2(\infty)=\infty$, we find that the two rational points of $X_{H_{11}}$, and hence of $X_{G_{11}}$, are cusps.  Therefore, there is no non-CM elliptic curve $E/\QQ$ for which $\rho_{E,7}(\Gal_\QQ)$ is conjugate to a subgroup of $G_{11}$; the same holds for the group $G_8$ since $G_8\subseteq G_{11}$.
\\

Now consider the subfield $K:=\QQ(h_2,w/\sqrt{-7})$ of $\calF_7$.   Let $H_1$ be the subgroup of $\GL_2(\FF_7)$ that fixes $K$.   From the inclusions $K\supseteq \QQ(h_2) \supseteq \QQ(j)$ and (\ref{E:w2 expression}), we find that $K$ is the function field of the geometrically irreducible curve
\begin{equation} \label{E:XH1}
-7y^2= x^4 - 10x^3 + 27x^2 - 10x - 27
\end{equation}
defined over $\QQ$ (with $(x,y)=(h_2, w/\sqrt{-7})$). The curve $X_{H_1}$ is defined over $\QQ$ since $\QQ$ is algebraically closed in $K$.  The only rational points of the  smooth projective model of (\ref{E:XH1}) are $(x,y)=(5/2,\pm 1/4)$ (one can show that it is isomorphic to the curve $E_{7,1}$ from \S\ref{SS:CM}, and that this curve has only two rational points).  These two rational points on $X_{H_1}$ lie over the $j$-invariant $J_2(5/2)=3^3\cdot 5\cdot 7^5/2^7$.   This shows that for an elliptic curve $E/\QQ$, $\rho_{E,7}(\Gal_\QQ)$ is conjugate to a subgroup of $H_1$ in $\GL_2(\FF_7)$ if and only if $j_E=3^3\cdot 5\cdot 7^5/2^7$.   Since $X_{H_1}$ has a rational point that is not a cusp, the group $H_1$ must be applicable and not conjugate to $G_{11}$.  The group $H_{1}$ is an index $2$ subgroup of $H_2$ since $[\QQ(h_2,w/\sqrt{-7}):\QQ(h_2)]=2$.   From our description of applicable groups, we deduce that $H_1$ is conjugate to $G_1$.   

\begin{remark}
The rational points on $X_{H_1}$ were first described by A.~Sutherland in \cite{1006.1782}.    An elliptic curve $E/\QQ$ with $j$-invariant $3^3\cdot 5\cdot 7^5/2^7$ has the distinguished property of not having a $7$-isogeny, yet its reduction at primes of good reduction all have a $7$-isogeny.
\end{remark}

From equation (4.35) of \cite{MR1722413}, the modular curve $X_{ns}^+(7):=X_{G_6}$ has function field of the form $\QQ(x)$ and the morphism down to the $j$-line is given by $J_6(x)$; note that there is a small typo in the numerator of equation (4.35) of \cite{MR1722413} though the given expression for $j-1728$ is correct.

\begin{lemma} \label{L:Schoof 7}
The rational points of the modular curve $X_{G_{12}}$ are all CM.
\end{lemma}
\begin{proof}
The fiber in $X_{ns}^+(7)$ over $j=1728$ is the (non-reduced) subscheme given by 
\[
(2x^4 - 14x^3 + 21x^2 + 28x + 7) (x-3) \big((x^4 - 7x^3 + 14x^2 - 7x + 7)(x^4 - 14x^2 + 56x + 21)\big)^2 = 0;
\]
this can be found by factoring $J_6(x)-1728$.  Define the modular curve $X_{ns}(7):= X_{C_{ns}(7)}$.   One can show that the morphism $X_{ns}(7) \to X_{ns}^+(7)$ is ramified at precisely four points lying over $j=1728$.    Since it is defined over $\QQ$, these four ramification points are the ones given by $2x^4 - 14x^3 + 21x^2 + 28x + 7=0$.   Therefore, $X_{ns}(7)$ is defined by an equation 
\[
y^2= c(2x^4 - 14x^3 + 21x^2 + 28x + 7)
\] 
for some squarefree $c\in \ZZ$.

We claim that $c=-1$.  Consider an elliptic curve $E/\QQ$ with $j$-invariant $-2^{15}$.   The value $x=1$ is the unique rational solution to $J(x) = -2^{15}$.    Setting $x=1$, we have $y^2= 44c$.  Therefore, $K=\QQ(\sqrt{11c})$ is the unique quadratic extension of $\QQ$ for which $\rho_{E,7}(\Gal_K) \subseteq C_{ns}(7)$.    Since $j_E=-2^{15}$, the curve $E$ has CM by $\QQ(\sqrt{-11})$ and hence $\rho_{E,7}(\Gal_{\QQ(\sqrt{-11})})= C_{ns}(7)$ and $\rho_{E,7}(\Gal_{\QQ})=N_{ns}(7)$; see \S\ref{SS:CM proofs}.   Therefore, $K=\QQ(\sqrt{-11})$ and hence $c=-1$ as claimed.    (The above argument comes from Schoof.)

Define the subfield $L=\QQ(x,v)$ of $\calF_7$ where $v:=y/\sqrt{-7}$; we have 
\begin{equation} \label{E:XG12 7}
7 v^2 = 2x^4 - 14x^3 + 21x^2 + 28x + 7.
\end{equation}   
Let $G$ be the subgroup of $\GL_2(\FF_7)$ that fixes $L$; it is an index $2$ subgroup of $G_6=N_{ns}(11)$ since $L/\QQ(x)$ has degree $2$.   The field $\QQ$ is algebraically closed in $L$ since $L/\QQ(x)$ is a geometric extension.   Therefore, $\det(G)=\FF_7^\times$.   There are only two index $2$ subgroups of $G_6$ with full determinant; they are $G_{12}$ and $C_{ns}(7)$.    The group $G$ is thus $G_{12}$ since $C_{ns}(7)$ corresponds to the field $\QQ(x,y)$.   

Therefore, $X_{G_{12}}$ has function field $\QQ(x,v)$ with $x$ and $v$ related by (\ref{E:XG12 7}).   The smooth projective curve defined by (\ref{E:XG12 7}) has genus $1$ and a rational point $(x,v)=(0,1)$; it is thus an elliptic curve.  A computation shows that this elliptic curve is isomorphic to the curve $E_{7,2}$ of \S\ref{SS:CM}.  The curve $E_{7,2}$ has only two rational points, so $(x,v)=(0,\pm 1)$ are the only rational points of the curve defined by (\ref{E:XG12 7}).   The lemma follows since $J_6(0)=0$.
\end{proof}

If $E/\QQ$ is a non-CM elliptic curve, Lemma~\ref{L:Schoof 7} shows that $\rho_{E,7}(\Gal_\QQ)$ is not conjugate to a subgroup of $G_{12}$.  The same holds for $G_9$ and $G_{10}$ since they are both subgroups of $G_{12}$.\\

Suppose that $H$ is a proper subgroup of $G_i$ satisfying $\pm H=G_i$ for a fixed $1\leq i\leq 7$.  If $i\neq 1$, then $i \in \{3,4,5,7\}$ and $H$ is one of the groups $H_{i,j}$.  If $i=1$, the $H$ is either $H_{1,1}$ or another subgroup that is conjugate to $H_{1,1}$ in $\GL_2(\FF_7)$.     This completes the proof of Theorem~\ref{T:main7}(\ref{T:main7 i}) and (\ref{T:main7 ii}); we can ignore $t=\infty$ for $2\leq i \leq 7$ since $J_i(\infty)$ is either $\infty$ or the $j$-invariant of a CM elliptic curve.  

\subsection{$\ell=11$}
Fix notation as in \S\ref{SS:applicable 11}.    Up to conjugacy, the group $\GL_2(\FF_{11})$ has four maximal applicable subgroups:  $B(11)$, $N_{s}(11)$, $N_{ns}(11)$ and a group $H_{\mathfrak{S}_4}$ whose image in $\PGL_2(\FF_{11})$ is isomorphic to $\mathfrak{S}_4$.

 \subsubsection{Exceptional case} \label{SS:exceptional 11}
 The curve $X_{\mathfrak{S}_4}(11):=X_{H_{\mathfrak{S}_4}}$ has no rational points corresponding to a non-CM elliptic curve; it is isomorphic to an elliptic curve which has only one rational point \cite{MR0463118}*{Prop.~4.4.8.1} and this point corresponds to an elliptic curve with CM by $\sqrt{-3}$.

\subsubsection{Split case}

The curve $X_{s}^+(11):=X_{N_s(11)}$ has no rational points corresponding to a non-CM elliptic curve; see \cite{1104.4641} for a more general result.    Therefore, there are no non-CM elliptic curves $E/\QQ$ such that $\rho_{E,\ell}(\Gal_\QQ)$ is conjugate to a subgroup of $N_s(11)$. 

 \subsubsection{Non-split case}  

The modular curve $X_{ns}^+(11):=X_{G_3}=X_{N_{ns}(11)}$ has genus $1$.   Halberstadt \cite{MR1677158} showed that the function field of $X_{ns}^+(11)$ is of the form $K:=\QQ(x,y)$ with $y^2+y = x^3-x^2-7x+10$ such that the inclusion $\QQ(j) \subseteq \QQ(x,y)$ is given by $j=J(x,y)$.   Therefore, if $E/\QQ$ is a non-CM elliptic curve, then $\rho_{E,11}(\Gal_\QQ)$ is conjugate to a subgroup of $N_{ns}(11)$ if and only if $j_E=J(P)$ for some point $P\in \calE(\QQ)$.  We only need consider $P\neq \OO$ since, as noted in \cite{MR1677158}, $J(\OO)$ is the $j$-invariant of a CM elliptic curve.

Let $G_4$ be the subgroup of $G_3$  consisting of $g \in G_3=N_{ns}(11)$ such that  $g\in C_{ns}(11)$ and $\det(g) \in (\FF_{11}^\times)^2$, or $g\notin C_{ns}(11)$ and $\det(g) \notin (\FF_{11}^\times)^2$.  

\begin{lemma} \label{L:XG4}
The modular curve  $X_{G_4}$ has no rational points.
\end{lemma}
\begin{proof}
Define the modular curve $X_{ns}(11):= X_{C_{ns}(11)}$.   Proposition~1 of \cite{DoseFernandezGonzalezSchoof} shows that $X_{ns}(11)$ can be defined by the equations $y^2+y = x^3-x^2-7x+10$ and $u^2= - (4x^3+7x^2-6x+19)$, where $K=\QQ(x,y)$.  

Define the field $L:=K(v)$ with $v=u/\sqrt{-11}$.  We have $L \subseteq \calF_{11}$ since $\sqrt{-11}\in \QQ(\zeta_{11})$.  Let $G$ be the subgroup of $\GL_2(\FF_{11})$ that fixes $L$; it is an index $2$ subgroup of $G_3$ since $L/K$ has degree $2$.   The field $\QQ$ is algebraically closed in $L$ since it is algebraically closed in $K$ and $L/K$ is a geometric extension.  Therefore, $\det(G)=\FF_{11}^\times$.    There are only two index $2$ subgroups of $G_3$ with full determinant; they are $G_4$ and $C_{ns}(11)$.    The group $G$ is thus $G_4$ since $C_{ns}(11)$ corresponds to the field $K(u)$.

Therefore, $X_{G_4}$ has function field $\QQ(x,y,v)$ where $y^2+y = x^3-x^2-7x+10$ and $v^2= 11 (4x^3+7x^2-6x+19)$.   We now homogenize our equations:
\begin{equation} \label{E:homogenous XG4} 
y^2z+yz^2 = x^3-x^2z-7xz^2+10z^3,\quad 11 v^2z= (4x^3+7x^2z-6xz^2+19z^3).
\end{equation}
Combining the two equations (\ref{E:homogenous XG4}) to remove the $x^3$ term, we find that $11 v^2 z = ( 4y^2z +4yz^2 +11 x^2 z + 22 x z^2 -21z^3)$.   Factoring off $z$, we deduce that the following equations give a model of $X_{G_4}$ in $\PP^3_\QQ$:  
\begin{equation} \label{E:homogenous XG4 2} 
y^2z+yz^2 = x^3-x^2z-7xz^2+10z^3,\quad 11 v^2 = ( 4y^2 +4yz +11 x^2  + 22 x z -21z^2).
\end{equation}
Suppose $(x,y,z,v) \in \PP^3(\QQ)$ is a solution to (\ref{E:homogenous XG4 2}).  If $z=0$, then we have $0=x^3$ and $11 v^2= 4 y^2$, which is impossible since $44$ is not a square in $\QQ$.   So assume that $z=1$.  We can then recover the equation $v^2= 11 (4x^3+7x^2-6x+19)$ which has no solutions $(x,v) \in \QQ^2$; it defines an elliptic curve and a computation shows that its only rational point is the point at $\infty$.    Therefore, $X_{G_4}(\QQ)=\emptyset$.
\end{proof}

Let $E/\QQ$ be a non-CM elliptic curve for which $\rho_{E,11}(\Gal_\QQ)$ is conjugate to a subgroup of $G_3$.  Suppose that $\rho_{E,11}(\Gal_\QQ)$ is conjugate to a subgroup of $G_3$.  The group $G_3$ has no index $2$ subgroups $H$ that satisfy $\pm H = G_3$.   Therefore, $\rho_{E,11}(\Gal_\QQ)$ is conjugate to a subgroup of a maximal applicable subgroup of $G_3$. Up to conjugacy, there are two maximal applicable subgroups of $G_3$; one is $G_4$ and the other is a subgroup $G_5$ of index $3$ in $G_3$.   The image $\bbar{G}_5$ of $G_5$ in $\PGL_2(\FF_{11})$ has order $8$ and is hence a $2$-Sylow subgroup of $\PGL_2(\FF_{11})$.   Therefore, $\bbar{G}_5$ lies in a subgroup of $\PGL_2(\FF_{11})$ that is isomorphic to $\mathfrak{S}_4$ and hence $G_5$ is conjugate to a subgroup of $H_{\mathfrak{S}_4}$.   However, we saw in \S\ref{SS:exceptional 11} that $\rho_{E,11}(\Gal_\QQ)$ cannot be conjugate to a subgroup of $H_{\mathfrak{S}_4}$.   This implies that $\rho_{E,11}(\Gal_\QQ)$ is conjugate to a subgroup of $G_4$ which is impossible by Lemma~\ref{L:XG4}.  Therefore, $\rho_{E,11}(\Gal_\QQ)$ must be conjugate to $G_3$.

\subsubsection{Borel case} \label{SS:11 borel}
 
The modular curve $X_{B(11)}$ is known to have exactly three rational points that are not cusps; they lie above the $j$-invariants $-2^{15}$, $-11^2$ and $-11\cdot 131^3$, cf.~\cite{MR0376533}*{p.~79}.  An elliptic curve with $j$-invariant $-2^{15}$ has CM, so we need only consider the other two.\\

Consider the elliptic curve $E/\QQ$ defined by $y^2+xy+y= x^3+x^2-305x+7888$; it has $j$-invariant $-11^2$ and conductor $11^2$.   The division polynomial at $11$ of $E$ factors as the product of the irreducible polynomial $f(x)=x^5 - 129x^4 + 800x^3 + 81847x^2 - 421871x - 4132831$ and an irreducible polynomial $g(x)$ of degree $55$.  Since $11$ divides the degree of $g(x)$, we find that $\rho_{E,11}(\Gal_\QQ)$ contains an element of order $11$.    Therefore, there are unique characters $\chi_1,\chi_2\colon \Gal_\QQ\to \FF_{11}^\times$ such that with respect to an appropriate change of basis we have
\begin{equation} \label{E:X011 form}
\rho_{E,11}(\sigma) = \left(\begin{smallmatrix}\chi_1(\sigma) & * \\0 & \chi_2(\sigma) \end{smallmatrix}\right).
\end{equation}
We have $\chi_1\chi_2=\omega$ where $\omega\colon \Gal_\QQ \to \FF_{11}^\times$ is the character describing the Galois action on the $11$-th roots of unity (we have $\omega(p)\equiv p \pmod{11}$ for primes $p\neq 11$).  The characters $\chi_1$ and $\chi_2$ are unramified at primes $p\nmid 11$, so $\chi_1= \omega^a$ and $\chi_2=\omega^{11-a}$ for a unique integer $0\leq a <10$.  Let $w\in \Qbar$ be a fixed root of $f(x)$.  One can show that 
\[
P=\big(w,-(w^4 - 79w^3 - 3150w^2 + 12193w+1520110)/11^4\big)
\]
is an $11$-torsion point of $E(\Qbar)$.   The field $\QQ(w)$ is a Galois extension of $\QQ$ and that the group generated by $P$ is stable under the action of $\Gal_\QQ$.   We thus have $\sigma(P)=\chi_1(\sigma)\cdot P$ for all $\sigma\in \Gal_\QQ$, and hence $\chi_1(\Gal_\QQ)$ is a group of order $[\QQ(w):\QQ]=5$.    

We have $a_2(E)=-1$, so the roots of the polynomial $\det(xI -\rho_{E,11}(\Frob_2))\equiv x^2-(-1)x+2 \pmod{11}$ are $4=2^2$ and $6\equiv 2^9 \pmod{11}$.  Since $\chi_1(\Frob_2)\equiv 2^a$ and $\chi(\Frob_2)\equiv 2^{11-a}$ are the roots of $\det(xI -\rho_{E,11}(\Frob_2))$ and $2$ is a primitive root modulo $11$, we have $a\in \{2,9\}$ and hence $\{\chi_1,\chi_2\}=\{ \omega^2,\omega^9\}$.  Since $\chi_1(\Gal_\QQ)$ has cardinality 5, we have $\chi_1=\omega^2$ and $\chi_2=\omega^9$.   Since $2$ is a primitive root modulo $11$, the group $\rho_{E,11}(\Gal_\QQ)$ is generated by $\big(\begin{smallmatrix}2^2 & 0 \\0 & 2^9 \end{smallmatrix}\big)=\left(\begin{smallmatrix}4 & 0 \\0 & 6 \end{smallmatrix}\right)$ and $\left(\begin{smallmatrix}1 & 1 \\0 & 1 \end{smallmatrix}\right)$, i.e., it equals $H_{1,1}$.   In particular, $\pm \rho_{E,11}(\Gal_\QQ) = G_1$.\\

Consider the elliptic curve $E/\QQ$ defined by $y^2+xy= x^3+x^2-3632x+82757$; it has $j$-invariant $-11\cdot 131^3$ and conductor $11^2$.   The division polynomial at $11$ of $E$ factors as the product of the irreducible polynomial $f(x)=x^5 - 129x^4 + 4793x^3 + 9973x^2 - 3694800x + 52660939$ and an irreducible polynomial $g(x)$ of degree $55$.  Since $11$ divides the degree of $g(x)$, we find that $\rho_{E,11}(\Gal_\QQ)$ contains an element of order $11$.    Therefore, there are unique characters $\chi_1,\chi_2\colon \Gal_\QQ\to \FF_{11}^\times$ such that with respect to an appropriate change of basis we have (\ref{E:X011 form}).
The characters $\chi_1$ and $\chi_2$ are unramified at primes $p\nmid 11$ and $\chi_1\chi_2=\omega$, so $\chi_1= \omega^a$ and $\chi_2=\omega^{11-a}$ for a unique integer $a\in \{0,1,\ldots, 9\}$.     Let $w\in \Qbar$ be a fixed root of $f(x)$.  One can show that 
\[
P=\big(w, (w^4 - 79w^3 + 843w^2 + 45468w - 722625)/11^3\big)
\]
is an $11$-torsion point of $E(\Qbar)$.   The field $\QQ(w)$ is a Galois extension of $\QQ$ and that the group generated by $P$ is stable under the action of $\Gal_\QQ$.   We thus have $\sigma(P)=\chi_1(\sigma)\cdot P$ for all $\sigma\in \Gal_\QQ$, and hence $\chi_1(\Gal_\QQ)$ is a group of order $[\QQ(w):\QQ]=5$.  
  We have $a_2(E)=1$, so the roots of the polynomial $\det(xI -\rho_{E,11}(\Frob_2))\equiv x^2-1\cdot x+2 \pmod{11}$ are $5\equiv 2^4$ and $7\equiv 2^7 \pmod{11}$.  Since $\chi_1(\Frob_2)\equiv 2^a$ and $\chi(\Frob_2)\equiv 2^{11-a}$ are the roots of $\det(xI -\rho_{E,11}(\Frob_2))$ and $2$ is a primitive root modulo 11, we have $a\in \{4,7\}$ and hence $\{\chi_1,\chi_2\}=\{ \omega^4,\omega^7\}$.  Since $\chi_1(\Gal_\QQ)$ has cardinality 5, we have $\chi_1=\omega^4$ and $\chi_2=\omega^7$.   Since $2$ is a primitive root modulo $11$, the group $\rho_{E,11}(\Gal_\QQ)$ is generated by $\big(\begin{smallmatrix}2^4 & 0 \\0 & 2^7 \end{smallmatrix}\big)=\left(\begin{smallmatrix}5 & 0 \\0 & 7 \end{smallmatrix}\right)$ and $\left(\begin{smallmatrix}1 & 1 \\0 & 1 \end{smallmatrix}\right)$, i.e., it equals $H_{2,1}$.   In particular, $\pm \rho_{E,11}(\Gal_\QQ) = G_2$.
%\newpage

\subsubsection{Polynomials for $X_{ns}^+(11)$} \label{S:ns section}

This subsection is dedicated to sketching Remark~\ref{R:ns computation} and making the polynomials explicit; fix notation as in \S\ref{SS:applicable 11}.  Define the polynomials:
{
\smaller
\begin{align*}
A(x)&=(x^5 - 9x^4 + 17x^3 + 20x^2 - 73x + 43)^{11},\\
B(x) &=-(x^2+3x-6)^3  \big(108000  x^{49} + 23793840  x^{48} - 413223722  x^{47} - 5377010368  x^{46} + 230799738529  x^{45} \\
& - 3137869050351  x^{44} + 23205911712335  x^{43} - 90936268647246  x^{42} + 33563647471596  x^{41} \\
&+ 1631415220074871  x^{40} - 7744726079195413  x^{39} - 3218815397602111  x^{38} + 236712051437217644  x^{37} \\
&- 1686428698022253344  x^{36} + 7984804002023063554  x^{35} - 30444784135263860996  x^{34}\\
& + 96849826504401032248  x^{33} - 232064394883539673213  x^{32} + 210175535413395353857  x^{31} \\
&+ 1609695806324946484826  x^{30} - 11768533689837648360109  x^{29} + 48291196122826259771817  x^{28} \\
&- 143943931899306373170309  x^{27} + 315827025781563232420857  x^{26} - 421596979720485992629121  x^{25} \\
&- 234929885880162547645306  x^{24} + 3668241437553022801950917  x^{23} - 14221091463553801024770599  x^{22}\\
& + 39148264563215734730610917  x^{21} - 87534472061810348609315974  x^{20}\\ 
& + 166474240219619575379485393  x^{19} - 275040771573054834247036345  x^{18} \\
&+ 399144725377223909937142938  x^{17} - 511840960382358144595839458  x^{16} \\ 
&+ 581656165535334214665717816  x^{15} - 586206578096981243980668654  x^{14} \\
&+ 523465655841901079370457175  x^{13} - 413200824632802503354807972  x^{12}\\
& + 287270832775316643952335709  x^{11} - 175049577131269087795781453  x^{10} \\ 
&+ 92916572268973769104815620  x^9 - 42636417323385892254033027  x^8 \\
&+ 16754292456737738144357709  x^7 - 5570911068111617263502302  x^6 + 1542648801995330874184236  x^5\\
& - 347819053424928336793068  x^4 + 61683475328903338239178  x^3 - 8117056250720937228985  x^2\\
& + 708318740340941449799  x - 30857360406231018655\big),
\\
C(x)&=(4x-5)(x^2+3x-6)^6(9x^2-28x+23)^3(x^4-5x^3+74x^2-245x+223)^3\\
&\quad\cdot(4x^4-9x^3-x^2+21x-32)^3(25x^4-114x^3+167x^2-86x+20)^3.
\end{align*}
}
\begin{prop}
For $j \in \QQ$, we have $J(P)=j$ for some point $P \in \calE(\QQ)-\{\OO\}$ if and only if $A(x) j^2 + B(x) j +C(x) \in \QQ[x]$ has a rational root.
\end{prop}
\begin{proof}
Take $(x,y) \in \calE-\{\OO\}$.  Using the equation $y^2+y = x^3-x^2-7x+10$, a direct computation shows that $J(x,y) A(x) = a(x)y+b(x)$ for unique $a,b\in \QQ[x]$.   Multiplying $y^2+y = x^3-x^2-7x+10$ by $a^2$, we deduce that 
$(J A -b)^2 + a(JA-b) - a^2(x^3-x^2-7x+10)=0$.   Therefore, $A^2 J^2 +  (-2b+a)A J + b^2-ba-a^2(x^3-x^2-7x+10) = 0$.   Our polynomials $B$ and $C$ satisfy $B=-2b+a$ and $C=(b^2-ba-a^2(x^3-x^2-7x+10))/A$.   We thus have 
\begin{equation} \label{E:J quadratic}
A(x) J(x,y)^2 + B(x) J(x,y) + C(x) = 0
\end{equation}
for all $(x,y)\in \calE-\{\OO\}$.

First suppose that $j=J(x_0,y_0)$ for some $(x_0,y_0) \in \calE(\QQ)-\{\OO\}$.   Then $0 = A(x_0) J(x_0,y_0)^2 + B(x_0) J(x_0,y_0) + C(x_0) = A(x_0) j^2 + B(x_0) j +C(x_0)$ and hence $A(x) j^2 + B(x) j +C(x)$ has a rational root.  

Now fix $j\in \QQ$ and suppose that there is an $x_0\in\QQ$ such that $A(x_0) j^2 + B(x_0) j +C(x_0)=0$.   Define $\Delta(x):=B(x)^2-4A(x) C(x)$.   A computation shows that $\Delta(x) = D(x)^2 (x^3 - x^2 - 7x + 41/4)$ for a polynomial $D\in \QQ[x]$ that has no rational roots.   The rational number $\Delta(x_0)=D(x_0)^2 (x_0^3 - x_0^2 - 7x_0+ 41/4)$ is a square since $j$ is a root of $A(x_0) X^2 + B(x_0) X + C(x_0) \in \QQ[X]$.   Therefore, $v^2 = x_0^3 - x_0^2 - 7x_0 + 41/4$ for some $v\in \QQ$.  With $y_0 = v -1/2$, we have $y_0^2+ y_0 = x_0^3 - x_0^2 - 7x_0 + 10$ and hence $P:=(x_0,y_0)$ is a point in $\calE(\QQ)-\{\OO\}$.   We could have chose $v$ with a different sign, so $P':=(x_0, -v -1/2 ) = (x_0, -y_0-1)$ also belongs to $\calE(\QQ)-\{\OO\}$.

We claim that $J(P)\neq J(P')$.  Suppose that they are in fact equal.   Using that $J(x,y) A(x) = a(x)y+b(x)$, we find that $a(x_0) y_0 = a(x_0) (-y_0 - 1)$.  Since $a(x)$ has no rational roots, we must have $y_0 = -1/2$ and hence $v=0$.   However, this is impossible since $x^3 - x^2 - 7x + 41/4$ has no rational roots, so the claim follows.  From (\ref{E:J quadratic}), we find that $J(P)$ and $J(P')$ are distinct roots of $A(x_0) X^2 + B(x_0) X + C(x_0)$.  Since $j$ is also a root of this quadratic polynomial, we deduce that $j= J(P)$ or $j=J(P')$.
\end{proof}

\subsection{$\ell=13$}
We shall prove parts (\ref{T:main 13 a}) and (\ref{T:main 13 b}) of Theorem~\ref{T:main 13} (part (\ref{T:main 13 d}) was explained in the introduction); so we will focus on $B(13)$ and its subgroups.  We first rule out subgroups of $C_s(13)$.

\begin{lemma}
There are no non-CM elliptic curves $E/\QQ$ for which $\rho_{E,13}(\Gal_\QQ)$ is conjugate in $\GL_2(\FF_{13})$ to a subgroup of $C_s(13)$.
\end{lemma}
\begin{proof}
Kenku has proved that the only rational points of $X_0(13^2)$ are cusps, cf.~\cite{MR588271,MR616547}.  By Lemma~\ref{L:borel to split}, we deduce that the only rational points of the modular curve $X_{C_s(13)}$ are cusps. 
\end{proof}

One can show that the applicable subgroups of $B(13)=G_6$ that are not subgroups of $C_s(13)$ are $G_1$, $G_2$, $G_3$, $G_4$, $G_5$, and $G_i \cap G_j$ with $i \in \{1,2\}$ and $j\in \{3,4,5\}$.   Note that these subgroups are normal in $B(13)$. \\

We now describe several modular function constructed by Lecacheux \cite{MR978099}*{p.56}.  Define 
\[
f(\tau)=\frac{\wp(\frac{1}{13};\tau) - \wp(\frac{2}{13}; \tau)}{\wp(\frac{1}{13};\tau) - \wp(\frac{3}{13}; \tau)} \quad \text{ and }\quad 
g(\tau)=\frac{\wp(\frac{1}{13};\tau) - \wp(\frac{2}{13}; \tau)}{\wp(\frac{1}{13};\tau) - \wp(\frac{5}{13}; \tau)}
\]
where $\wp(z;\tau)$ is the Weierstrass $\wp$-function at $z$ of the lattice $ \ZZ \tau + \ZZ\subseteq \CC$.  Define the functions 
\[
 h_5:= \frac{(g-1)(g(g-1)+1-f)}{(f-1)(f-g)} \quad \text{ and }\quad h_2:=\frac{f-1}{g-1}.
\]
The functions $h_5$ and $h_2$ belong to $\calF_{13}$ and satisfy $F_2(h_2) = F_5(h_5)$, where 
\[
F_2(t)=t+(t-1)/t-1/(t-1) -4=(t^3 - 4t^2 + t + 1)/(t^2 - t) \quad\text{ and }\quad F_5(t)=t -1/t -3=(t^2 - 3t - 1)/t;
\] 
this follows from \cite{MR978099}*{p.56--57} with $H=h_5$ and $h=h_2$.  

Let  $h_6$ be the function $F_2(h_2)=F_5(h_5)$; it is called $a-3$ in \cite{MR978099} and satisfies $J_6(h_6)=j$, cf.~\cite{MR978099}*{p.62}.  Since $J_2(t)=J_6(F_2(t))$ and $J_5(t)=J_6(F_5(t))$, we have $J_2(h_2)=j$ and $J_5(h_5)=j$.\\

Define $\alpha:= -\zeta_{13}^{11} - \zeta_{13}^{10} - \zeta_{13}^3 - \zeta_{13}^2 + 1$.  Define the rational functions 
\[
F_1(t)=13(t^2-t)/(t^3-4t^2+t+1)\quad \text{ and }\quad \phi_1(t)= ({\alpha t+1-\alpha})/({t -\alpha});
\]  
Define the modular function $h_1:=\phi_1(h_2) \in \calF_{13}$.  One can check that $F_1(\phi_1(t))=F_2(t)$ and hence $F_1(h_1)=F_2(h_2)=h_6$.  Since $J_1(t)=J_6(F_1(t))$,  we have $J_1(h_1)=j$.

Define $\beta:= \zeta_{13}^{11} + \zeta_{13}^{10} + \zeta_{13}^9 + \zeta_{13}^7 + \zeta_{13}^6 + \zeta_{13}^4 + \zeta_{13}^3 + \zeta_{13}^2 + 2$.  Define the rational functions 
\[
F_3(t)=(-5t^3 + 7t^2 + 8t - 5)/(t^3 - 4t^2 + t + 1) \quad \text{ and }\quad \phi_3(t)= (\beta t -1)/(t+ \beta-1).
\]  
Define the modular function $h_3:=\phi_3(h_2) \in \calF_{13}$.  One can check that $F_3(\phi_3(t))=F_2(t)$ and hence $F_3(h_3)=F_2(h_2)=h_6$.  Since $J_3(t)=J_6(F_3(t))$,  we have $J_3(h_3)=j$.

Define $\gamma=(1+\sqrt{13})/2$; it belongs to  $\QQ(\zeta_{13})$ and moreover equals $\gamma=-\zeta_{13}^{11} - \zeta_{13}^8 - \zeta_{13}^7 - \zeta_{13}^6 - \zeta_{13}^5 - \zeta_{13}^2$.   Define the rational functions
\[
F_4(t)=13t/(t^2 - 3t - 1)\quad \text{ and }\quad \phi_4(t)=((2-\gamma)t + 1)/(t -2+\gamma)).
\]
Define the modular function $h_4:=\phi_4(h_5) \in \calF_{13}$.  One can check that $F_4(\phi_4(t))=F_5(t)$ and hence $F_4(h_4)=F_5(h_5)=h_6$.  Since $J_4(t)=J_6(F_4(t))$,  we have $J_4(h_4)=j$.

For $1\leq i \leq 6$, let $H_i$ be the subgroup of $\GL_2(\FF_7)$ that fixes $h_i$.   We have shown that $J_i(h_i)=j$.  By Lemma~\ref{L:key}, we find that $H_i$ is an applicable subgroup and that the morphism $\pi_{H_i}\colon X_{H_i} \to  \PP^1_\QQ$ is described by the rational function $J_i(t)$.

\begin{lemma}
The groups $H_i$ and $G_i$ are conjugate in $\GL_2(\FF_{13})$ for all $1\leq i \leq 6$.
\end{lemma}
\begin{proof}
The index of $H_6$ in $\GL_2(\FF_{13})$ is equal to $14$, i.e., the degree of $J_6$ as a morphism.   Therefore, $H_6$ must be conjugate to $B(13)$.    The index $[H_6:H_i]$ equals the degree of $F_i(t)$, and is thus $3$ if $i\in \{1,2,3\}$ and $2$ if $i\in \{4,5\}$.    

The groups $H_1$, $H_2$ and $H_3$ are not conjugate in $\GL_2(\FF_{13})$ since one can show that the images of $\PP^1(\QQ)=\QQ \cup\{\infty\}$ under $J_1$, $J_2$ and $J_3$ are distinct.    Therefore, $H_1$, $H_2$ and $H_3$ are conjugate to $G_1$, $G_2$ and $G_3$ which are the applicable subgroups of $B(13)$ of index $2$;  however, we still need to determine which group is conjugate to which.      

Let $E/\QQ$ be the elliptic curve defined by  $y^2=x^3-338x+2392$.  The group $\rho_{E,13}(\Gal_\QQ)$ is conjugate to a subgroup of $H_3$ since $j_E=J_3(0)$.   One can check that $E/\QQ$ has good reduction at $3$ and that $a_3(E)=0$.   Since $x^2-a_3(E)+3  \equiv (x-6)(x+6) \pmod{13}$, we deduce that the eigenvalues of the matrix $\rho_{E,13}(\Frob_3)$ are $6$ and $-6$.   For every matrix in $G_1$ or $G_2$ has an eigenvalue in $(\FF_{13}^\times)^3 = \{\pm 1, \pm 5\}$.    Since $6$ and $-6$ do not belong to $(\FF_{13}^\times)^3$, we deduce that $H_3$ is not conjugate to $G_1$ and $G_2$.   Therefore, $H_3$ is conjugate to $G_3$.

Let $E/\QQ$ be the elliptic curve defined by $y^2=x^3-2227x-59534$.   We have $j_E=J_2(2)$ and $j_E \notin J_1(\QQ\cup\{\infty\})$.   Therefore, $\rho_{E,13}(\Gal_\QQ)$ is conjugate to a subgroup of $H_2$ and not conjugate to a subgroup of $H_1$.  By computing the division polynomial of $E$ at the prime $13$, we find that $E$ has a point $P$ of order $13$ whose $x$-coordinate is $17+8\sqrt{17}$.   So with respect to a basis of $E[13]$ whose first element is $P$, we find that $\rho_{E,13}(\Gal_\QQ)$ is a subgroup of $G_2$.   Therefore, $H_2$ is conjugate to $G_2$, and hence $H_1$ is conjugate to $G_1$.

The groups $H_4$ and $H_5$ are not conjugate in $\GL_2(\FF_{13})$ since one can show that the images of $\PP^1(\QQ)=\QQ \cup\{\infty\}$ under $J_4$ and $J_5$ are distinct.    Therefore, $H_4$ and $H_5$ are conjugate to $G_4$ and $G_5$ which are the applicable subgroups of $B(13)$ of index $3$;  however, we still need to determine which group is conjugate to which.      

Let $E/\QQ$ be the elliptic curve defined by $y^2 = x^3 - 3024x - 69552$.    We have $j_E=J_5(2)$ and  $j_E\notin J_4(\QQ\cup\{\infty\})$. Therefore, $\rho_{E,13}(\Gal_\QQ)$ is conjugate to a subgroup of $H_5$ and not conjugate to a subgroup of $H_4$.   By computing the division polynomial of $E$ at the prime $13$, we find that $E$ has a point $P$ of order $13$ whose $x$-coordinate $w$ is a root of $x^3 - 3024x + 12096$.   The cubic extension $\QQ(w)$ of $\QQ$ is Galois, so with respect to a basis of $E[13]$ whose first element is $P$, we find that $\rho_{E,13}(\Gal_\QQ)$ is a subgroup of $G_5$.   Therefore, $H_5$ is conjugate to $G_5$, and hence $H_4$ is conjugate to $G_4$.  
\end{proof}

We have thus completed the proof of Theorem~\ref{T:main 13}(\ref{T:main 13 b}); we can ignore $t=\infty$ since $J_i(\infty)=\infty$ for $i\neq 3$ and $J_3(\infty)=J_3(0)$.  If $H$ is a proper subgroup of $G_i$ satisfying $\pm H=G_i$, then one can show that $i\in \{4,5\}$ and $H$ is one of the groups $H_{i,j}$.\\

To complete the proof of Theorem~\ref{T:main 13}(\ref{T:main 13 a}), we need only show that the modular curves $X_{G_i\cap G_j}$, with fixed $i\in \{1,2\}$ and $j \in \{3,4,5\}$, have no rational points other than cusps.   It suffices to prove the same thing for the modular curves $X_{H_i\cap H_j}$.

 The function field of $X_{H_i\cap H_j}$ is $\QQ(h_i,h_j)$ and the generators $h_i$ and $h_j$ satisfy the relation $F_i(h_i)=h_6=F_j(h_j)$.   The smooth projective (and geometrically irreducible) curve over $\QQ$ arising from the equation $F_i(x)=F_j(y)$ is thus a model of $X_{H_i\cap H_j}$.  

The following \texttt{Magma} code shows that if $(x,y) \in \QQ^2$ is a solution of $F_i(x)=F_j(y)$ (where we say that both sides equal $\infty$ if the denominators vanish), then $y=0$.   The code considers the projective (and possibly singular) curve $C_{i,j}$ in $\PP^2_\QQ$ defined by the affine equation $F_i(x)=F_j(y)$ (we first clear denominators and homogenize).  We then find a genus 2 curve $C$ that is birational with $C_{i,j}$ and is defined by some Weierstrass equation $y^2=f(x)$ with $f(x)\in \QQ[x]$ a separable polynomial of degree 5 or 6.  We then check that the Jacobian $J$ of $C$ has rank $0$, equivalently, that $J(\QQ)$ is a finite group (\texttt{Magma} accomplishes this by computing the $2$-Selmer group of $J$).   Using that $J(\QQ)$ has rank $0$, the function \texttt{Chabauty0} finds all the rational points on $C$.   Using the birational isomorphism between $C$ and $C_{i,j}$, we can determine the rational points of $C$.
{\small
\begin{verbatimtab}
	K<t>:=FunctionField(Rationals());
	F:=[13*(t^2-t)/(t^3-4*t^2+t+1), (t^3-4*t^2+t+1)/(t^2-t), 
	    (-5*t^3+7*t^2+8*t-5)/(t^3-4*t^2+t+1), 13*t/(t^2-3*t-1), (t^2-3*t-1)/t ];
	P2<x,y,z>:=ProjectiveSpace(Rationals(),2);
	for i in [1,2,3], j in [4,5] do
	    f:=Numerator(Evaluate(F[i],x/z)- Evaluate(F[j],y/z));
	    while Evaluate(f,z,0) eq 0 do  f:= f div z; end while;
	    C0:=Curve(P2,f);    
	    b,C1,f1:=IsHyperelliptic(C0); C2,f2:=SimplifiedModel(C1);
	    Jac:=Jacobian(C2);   RankBound(Jac) eq 0;
	    S:=Chabauty0(Jac);
	    b,g1:=IsInvertible(f1); b,g2:=IsInvertible(f2);
	    T:=g1(g2(S) join SingularPoints(C1)) join SingularPoints(C0); 
	    {P: P in T | P[2] ne 0 and P[3] ne 0} eq {};
	end for;
\end{verbatimtab}
}
We find that if $F_i(x)=F_i(y)$ for some $x,y\in \QQ\cup\{\infty\}$, then $y=0$ or $y=\infty$.   Thus the only rational points of $X_{H_i\cap H_j}$ are cusps since $J_j(0)=J_j(\infty)=\infty$ for $j\in\{4,5\}$.

\subsection{$\ell=17$}
We now prove Theorem~\ref{T:17-37}(\ref{T:17-37 i}).   Let $E/\QQ$ be the elliptic curve defined by the Weierstrass equation $y^2+xy+y=x^3-190891x -36002922$; it has $j$-invariant $-17\cdot 373^3/2^{17}$ and conductor $2\cdot 5^2\cdot 17^2$.    The division polynomial of $E$ at $17$ factors as a product of $f(x)=x^4 + 482x^3 + 1144x^2 - 15809842x - 958623689$ with irreducible polynomials of degree $4$ and $8\cdot 17$.  Fix a point $P \in E(\Qbar)$ whose $x$-coordinate $w$ is a root of $f(x)$; it is a $17$-torsion point.  Let $C$ be the cyclic group of order $17$ generated by $P$; it is stable under the $\Gal_\QQ$ action.   Let $\chi_1\colon \Gal_\QQ\to \FF_{17}^\times$ be the homomorphism such that $\sigma(P)=\chi_1(\sigma)\cdot P$ for $\sigma\in \Gal_\QQ$.   One can show that the degree $4$ extension $\QQ(w)/\QQ$ is Galois, so $\chi_1(\Gal_\QQ)$ has cardinality $4$ or $8$.  There is a second character $\chi_2\colon \Gal_\QQ\to \FF_{17}^\times$ such that, with respect to an appropriate change of basis, we have
\[
\rho_{E,17}(\sigma) = \left(\begin{smallmatrix}\chi_1(\sigma) & * \\0 & \chi_2(\sigma) \end{smallmatrix}\right).
\]
The cardinality of $\rho_{E,17}(\Gal_\QQ)$ is divisible by $17$ since the division polynomial of $E$ at $17$ has an irreducible factor whose degree is divisible by $17$.    We have $\chi_1\chi_2=\omega$ where $\omega\colon \Gal_\QQ \to \FF_{17}^\times$ is the character describing the Galois action on the $17$-th roots of unity (we have $\omega(\Frob_p)=p$ for primes $p\neq 17$).  The characters $\chi_1$ and $\chi_2$ are unramified at primes $p\nmid 2\cdot 5\cdot 17$, so $\chi_1= \omega^a\chi$ and $\chi_2=\omega^{17-a}\chi^{-1}$ for some integer $0\leq a <16$ and some character $\chi\colon \Gal_\QQ \to \FF_{17}^\times$ unramified at $p\nmid 2\cdot 5$.

 Let $H_1$ and $H_2$ be the subgroup of $\GL_2(\FF_\ell)$ consisting of matrices of the form 
 \[
 \left(\begin{smallmatrix} \omega(\sigma)^a & 0 \\0 & \omega(\sigma)^{17-a} \end{smallmatrix}\right) \quad\quad \text{and}\quad\quad\left(\begin{smallmatrix} \chi_1(\sigma) & 0 \\0 & \chi_1(\sigma)^{-1} \end{smallmatrix}\right),
 \] 
 respectively, with $\sigma\in \Gal_\QQ$.   Since $\omega$ and $\chi$ are ramified at different primes, we find that the image of $\rho_{E,\ell}$ is generated by $\left(\begin{smallmatrix} 1 & 1 \\0 & 1 \end{smallmatrix}\right)$ and the groups $H_1$ and $H_2$.
 
 The character $\chi$ is unramified at $p\nmid 2\cdot 5$ and has image in a cyclic group of order $16$.   Therefore, $\chi$ must factor through the group $\Gal(\QQ(\zeta_{64},\zeta_5)/\QQ)$.  Since $641\equiv 1 \pmod{64\cdot 5}$, we have  $\chi(\Frob_{641})=1$.  Therefore, $\chi_1(\Frob_{641})=\omega(\Frob_{641})^a\cdot 1 \equiv 641^a \pmod{17}$ is a root of
\[
x^2-a_{641}(E)x+641 = x^2-(-9)x+641 \equiv (x-641^6)(x-641^{11}) \pmod{17},
\] 
and hence $a\in\{6,11\}$ since $641$ is a primitive root modulo $17$.    If $a=11$, then $\chi_1(\Gal_\QQ)=\FF_{17}^\times$ which is impossible since the cardinality of $\chi_1(\Gal_\QQ)$ is $4$ or $8$.  Therefore, $a=6$.    The group $H_1$ thus consists of matrices of the form $\left(\begin{smallmatrix} c^6 & 0 \\0 & c^{11} \end{smallmatrix}\right)$ with $c\in \FF_{17}^\times$, and in particular is generated by $\left(\begin{smallmatrix} 5^6 & 0 \\0 & 5^{11} \end{smallmatrix}\right)=\left(\begin{smallmatrix} 2 & 0 \\0 & 11 \end{smallmatrix}\right)$.

 To complete the proof that $\rho_{E,17}(\Gal_\QQ)$ is $G_1$, it suffices to show that $H_2$ is generated by $\left(\begin{smallmatrix} 4 & 0 \\0 & -4\end{smallmatrix}\right)$; equivalently, to show that the image of $\chi$ is cyclic of order $4$.  As noted earlier, $\chi$ factors through the group $\Gal(\QQ(\zeta_{64},\zeta_5)/\QQ)\cong (\ZZ/64\cdot 5\ZZ)^\times$.  One can then show that $\Gal(\QQ(\zeta_{64},\zeta_5)/\QQ)$ is generated by $\Frob_{103}$, $\Frob_{137}$ and $\Frob_{307}$.   The primes $p\in \{103, 137, 307\}$ were chosen to be congruent to $1$ modulo $17$, and hence $\chi(\Frob_p)=\chi_1(\Frob_p)$ is a root of $x^2-a_p(E)x+p$ modulo $17$.  It is then straightforward to check that $\chi(\Frob_{103})$, $\chi(\Frob_{137})$ and $\chi(\Frob_{307})$ all have order $4$.

 The elliptic curve $E'/\QQ$ defined by the Weierstrass equation $y^2 + xy + y = x^3 - 3041x + 64278$; it has $j$-invariant $- 17^2 \cdot 101^3/2$.   One can show that  $E/C$ is isomorphic to $E'$.     The group $\rho_{E',17}(\Gal_\QQ)$ is thus conjugate to $G_2$ in $\GL_2(\FF_{17})$.

Finally we note that $G_1$ and $G_2$ have no index $2$ subgroups that do not contain $-I$.

\subsection{$\ell=37$}
We now prove Theorem~\ref{T:17-37}(\ref{T:17-37 ii}).  Let $E/\QQ$ be the elliptic curve defined by the equation $y^2+xy+y=x^3+x^2-8x+6$; it has $j$-invariant $-7\cdot 11^3$ and conductor $5^2\cdot 7^2$.   The division polynomial of $E$ at $17$ factors as a product of $f(x):=x^6 - 15x^5 - 90x^4 - 50x^3 + 225x^2 + 125x - 125$ with   irreducible polynomials of degree $6$, $6$ and $18\cdot 37$.   Fix a point $P \in E(\Qbar)$ whose $x$-coordinate $w$ is a root of $f(x)$; it is a $37$-torsion point.  Let $C$ be the cyclic group of order $37$ generated by $P$; it is stable under the $\Gal_\QQ$ action.  

Let $\chi_1\colon \Gal_\QQ\to \FF_{37}^\times$ be the homomorphism such that $\sigma(P)=\chi_1(\sigma)\cdot P$ for $\sigma\in \Gal_\QQ$.   One can show that the degree $6$ extension $\QQ(w)/\QQ$ is Galois, so $\chi_1(\Gal_\QQ)$ has cardinality $6$ or $12$; in particular $\chi_1(\Gal_\QQ)$ is a subgroup of $(\FF_{37}^\times)^3$.  There is a second character $\chi_2\colon \Gal_\QQ\to \FF_{37}^\times$ such that, with respect to an appropriate change of basis, we have
\[
\rho_{E,37}(\sigma) = \left(\begin{smallmatrix}\chi_1(\sigma) & * \\0 & \chi_2(\sigma) \end{smallmatrix}\right).
\]
The cardinality of $\rho_{E,37}(\Gal_\QQ)$ is divisible by $37$  since the division polynomial of $E$ at $37$ has an irreducible factor whose degree is divisible by $37$.  So to prove that $\rho_{E,37}(\Gal_\QQ)=G_3$, it suffices to show that the homomorphism $\chi_1\times \chi_2 \colon \Gal_\QQ \to (\FF_{37}^\times)^3\times \FF_{37}^\times$ is surjective.

  The characters $\chi_1$ and $\chi_2$ are unramified at primes $p\nmid 5\cdot 7\cdot 37$.  By Proposition~11 of \cite{MR0387283}, we have $\{\chi_1,\chi_2\} = \{ \alpha, \alpha^{-1}\cdot \omega \}$ where $\alpha \colon \Gal_\QQ \to \FF_{37}^\times$ is  a character unramified at primes $p\nmid 5\cdot 7$ and  $\omega\colon \Gal_\QQ \to \FF_{37}^\times$ is the character describing the Galois action on the $37$-th roots of unity.    Since $\alpha$ is unramified at $37$, we find that the character $\alpha^{-1}\cdot \omega$ is surjective and  that $(\alpha \times (\alpha^{-1}\cdot \omega))(\Gal_\QQ)=\alpha(\Gal_\QQ)\times \FF_{37}^\times$.  Since $\chi_1$ is not surjective, we must have $\chi_1=\alpha$ and $\chi_2=\alpha^{-1}\cdot \omega$.  It thus suffices to show that the image of $\alpha$ contains an element of order $12$. The fixed field of the kernel of $\alpha$ is contained in $\QQ(\zeta_5,\zeta_7)$ since it is unramified at $p\nmid 5\cdot 7$ and has image relatively prime to $5\cdot 7$.  Since $107\equiv 2 \pmod{35}$, we have $\alpha(\Frob_2)=\alpha(\Frob_{107})$.   Therefore, $\alpha(\Frob_2)$ is a common root of $x^2-a_2(E)x+2 = x^2+x+2$ and $x^2-a_{107}(E)x+107=x^2+11x+107$ modulo $37$.   This implies that $\alpha(\Frob_2)$ equals $8 \in \FF_{37}^\times$ which has order $12$.

One can show that the quotient of $E$ by $C$ is the elliptic curve $E'/\QQ$ defined by $y^2+xy+y=x^3+x^2-208083x-36621194$; it has $j$-invariant $-7\cdot 137^3\cdot 2083^3$.  The group $\rho_{E',37}(\Gal_\QQ)$ is thus conjugate in $\GL_2(\FF_{37})$  to $G_4$.

Finally we note that $G_3$ and $G_4$ have no index $2$ subgroups that do not contain $-I$.

\section{Quadratic twists} \label{S:twist 1}

Fix an elliptic curve $E/\QQ$ with $j_E\notin \{0,1728\}$ and an integer $N\geq 3$.   

Define the group $G:=\pm \rho_{E,N}(\Gal_\QQ)$ and let $\calH$ be the set of proper subgroups $H$ of $G$ that satisfy $\pm H=G$.   For each group $H \in \calH$, we obtain a character
\[
\chi_{E,H} \colon \Gal_\QQ \to \{\pm 1\}
\]
by composing $\rho_{E,N}$ with the quotient map $G\to G/H \cong \{\pm 1\}$.   The fixed field of the kernel of the character $\chi_{E,H}$ is of the form $\QQ(\sqrt{d_{E,H}})$ for a unique squarefree integer $d_{E,H}$.   Define the set
\[
\calD_E:=\{ d_{E,H} \colon H \in \calH \}.
\]
Using $\pm \rho_{E,N}(\Gal_\QQ)=G$, we find that different groups $H\in \calH$ give rise to distinct characters $\chi_{E,H}$ and thus $|\calD_E|=|\calH|$.  
 
\subsection{Twists with smaller image}

For a squarefree integer $d$, let $E_d/\QQ$ be a quadratic twist of $E/\QQ$ by $d$.   By choosing an appropriate basis of $E_d[\ell]$, we may assume that $\rho_{E_d,N}\colon \Gal_\QQ \to \GL_2(\ZZ/N\ZZ)$ satisfies 
\begin{equation*}\label{E:twist rho}
\rho_{E_d,N} = \chi_d \cdot \rho_{E,N},
\end{equation*}
where $\chi_d \colon \Gal_\QQ \to \{\pm 1\}$ is the character corresponding to the extension $\QQ(\sqrt{d})/\QQ$.   We have $\pm \rho_{E_d,N}(\Gal_\QQ)=\pm \rho_{E,N}(\Gal_\QQ)= G$.     Therefore, $\rho_{E_d,N}(\Gal_\QQ)$ is equal to either $G$ or to one of the subgroups $H\in \calH$.

We now show that $\calD_E$ is precisely the set of squarefree integers $d$ for which the image of $\rho_{E_d,N}$ is not conjugate to $G$.

\begin{lemma} \label{L:twist newer}
Take any squarefree integer $d$.
\begin{romanenum}
\item \label{L:twist newer i}
We have $d\in \calD_E$ if and only if the group $\rho_{E_d,N}(\Gal_\QQ)$ is conjugate in $\GL_2(\ZZ/N\ZZ)$ to a proper subgroup of $G$.
\item \label{L:twist newer ii}
If $d=d_{E,H}$ for some $H\in \calH$, then $\rho_{E_d,N}(\Gal_\QQ)$ is conjugate in $\GL_2(\ZZ/N\ZZ)$ to $H$.
\end{romanenum}
\end{lemma}
\begin{proof}
Take any group $H \in \calH$.   Composing $\rho_{E_d,N}\colon \Gal_\QQ \to G$ with the quotient map $G\to G/H\cong \{\pm 1\}$ gives the character $\chi_d \cdot \chi_{E,H}$.    Therefore, $\rho_{E_d,N}(\Gal_\QQ)$ is a subgroup of $H$ (and hence equal to $H$) if and only if $\chi_{E,H}=\chi_{d}$; equivalently, $d=d_{E,H}$.    Parts (\ref{L:twist newer i}) and (\ref{L:twist newer ii}) are now immediate. 
\end{proof}

Since $|\calD_E|=|\calH|$, we deduce from Lemma~\ref{L:twist newer} that the map 
\[
\calH \to \calD_E, \quad H\mapsto d_{E,H}
\]
is a bijection.

\begin{remark}
Observe that $\rho_{E_d,N}(\Gal_\QQ)$ being conjugate to $H$ in $\GL_2(\ZZ/N\ZZ)$ need not imply that $d=d_{E,H}$. For example, it is possibly for distinct groups in $\calH$ to be conjugate in $\GL_2(\ZZ/N\ZZ)$.
\end{remark}

\subsection{Computing $\calD_E$}

Now assume that $N \geq 3$ is odd; we shall explain how to compute $\calD_E$ (we will later be interested in the case where $N$ is an odd prime).  Let $M_E$ be set of squarefree integers that are divisible only by primes $p$ such that $p|N$ or such that $E$ has bad reduction at $p$.

For each $r\geq 1$, let $\calD_r$ be the set of $d\in M_E$ such that
\begin{equation} \label{E:twist -2}
  a_{p}(E)\not \equiv  -2 \left(\tfrac{d}{p}\right) \pmod{N}
\end{equation}
holds for all primes $p \leq r$ for which $E$ has good reduction and $p\equiv 1 \pmod{N}$.

\begin{lemma} \label{L:calD inclusion}
Suppose that $N$ is odd.   We have $\calD_E\subseteq \calD_r$ with equality holding for all sufficiently large $r$.
\end{lemma}
\begin{proof}
Define $\mathscr{D}:= \cap_r \calD_r$; it is the set of $d\in M_E$ such that (\ref{E:twist -2}) holds for all primes $p\equiv 1 \pmod{N}$ for which $E$ has good reduction.   We have $\calD_{r}\subseteq \calD_{r'}$ if $r\geq r'$,  so it suffices to prove that $\mathscr{D}=\calD_E$. 

Take any $d\in \mathscr{D}$.  We have $a_p(E_d)=\legendre{d}{p} a_{p}(E)\not\equiv -2 \pmod{N}$ for all primes $p\equiv 1 \pmod{N}$ for which $E$ has good reduction.   By the Chebotarev density theorem, there are no elements $g\in \rho_{E_d,N}(\Gal_\QQ)$ satisfying $\det(g)=1$ and $\tr(g)=-2$.   In particular, the group $\rho_{E_d,N}(\Gal_\QQ)$ does not contain $-I$ and hence $d \in \calD_E$ by Lemma~\ref{L:twist newer}(\ref{L:twist newer i}).  Therefore, $\mathscr{D}\subseteq \calD_E$.

We have $\calD_E\subseteq M_E$ since each character $\chi_{E,H}$ factors through $\rho_{E,N}$ (and is hence unramified at all primes $p\nmid N$ for which $E$ has good reduction). 

Now take any $d\in \calD_E - \mathscr{D}$.   There is thus a prime $p\equiv 1 \pmod{N}$ for which $E$ has good reduction and $a_p(E_d)=\legendre{d}{p}   a_{p}(E) \equiv -2 \pmod{N}$.   Define $g:=\rho_{E_d,N}(\Frob_p)$; it has trace $-2$ and determinant $1$.    Since $N$ is odd, some power of $g$ is equal to $-I$.  Therefore, $ \rho_{E_d,N}(\Gal_\QQ)=\pm \rho_{E_d,N}(\Gal_\QQ) = G$ which contradicts that $d\in \calD_E$.    Therefore, $\calD_E - \mathscr{D}$ is empty and hence $\calD_E\subseteq \mathscr{D}$.
\end{proof}

One can compute the finite sets $\calD_r$ for larger and larger values of $r$ until $|\calD_r| = |\calH|$ and  then $\calD_E=\calD_r$.  This works since we always have an inclusion $\calD_E \subseteq\calD_r$ by Lemma~\ref{L:calD inclusion}, and equality holds when $|\calD_r| = |\calH|$ since $|\calD_E|=|\calH|$.\\

When $N$ is a prime, the integers in $\calD_E$ come in pairs.

\begin{lemma} \label{L:ell twist D}
Suppose $N=\ell$ is an odd prime.  Let $\calD'_E$ be the set of $d\in \calD_E$ for which $\ell \nmid d$.   Then 
\[
\calD_E = \bigcup_{d\in \calD_E'} \{d, (-1)^{(\ell-1)/2} \ell\cdot d\}.
\]
\end{lemma}
\begin{proof}
Define $\ell^*:=(-1)^{(\ell-1)/2} \ell$.   Take any $d \in \calD_E$.   We need to show that $d \ell^*$ or $d/\ell^*$  belong to $\calD_E$ (whichever one is a squarefree integer).   After possibly replacing $E$ by $E_d$, we may assume that $d=1$ and hence we need only verify that $\ell^* \in \calD_E$.

So assume that $\rho_{E,\ell}(\Gal_\QQ)$ is a proper subgroup of $G$ and hence is equal to one of the $H\in \calH$.   We need to show that $\rho_{E',\ell}(\Gal_\QQ)$ is also a proper subgroup of $G$, where $E':=E_{\ell^*}$.

   The field $\QQ(\sqrt{\ell^*})\subseteq \QQ(\zeta_\ell)$ is a subfield of both $\QQ(E[\ell])$ and $\QQ(E'[\ell])$.  Since $E$ and $E'$ are isomorphic over $\QQ(\sqrt{\ell^*})$, we deduce that $[\QQ(E'[\ell]):\QQ]=[\QQ(E[\ell]):\QQ]$.   Therefore, 
 \[
 |\rho_{E',\ell}(\Gal_\QQ)| =   [\QQ(E'[\ell]):\QQ]=[\QQ(E[\ell]):\QQ] =  |\rho_{E,\ell}(\Gal_\QQ)| = |H| = |G|/2.
 \]
By cardinality assumption, we deduce that $\rho_{E',\ell}(\Gal_\QQ)$ is conjugate to a proper subgroup of $G$.   
\end{proof}

\begin{remark}
One could also use the methods of this section to help determine $\calH$.   For example, if $\calD_r=\emptyset$ for some $r$, then $\calH= \emptyset$.   Suppose we are in the setting, like what happens often in the introduction, where we know that $|\calH| \geq 2$ because we have two explicit elements of $\calH$.  Then to verify that $|\calH|=2$, one need only find an $r$ such that $|\calD_r| =2$.
\end{remark}

\subsection{Some examples}
\subsubsection{} 
Take $\ell=7$.   Let $E/\QQ$ be the elliptic curve defined by $y^2=x^3-5^37^3x -5^47^2 106$; it has $j$-invariant $3^3\cdot 5\cdot 7^5/2^7$ and conductor $2\cdot 5^2 \cdot 7^2$.    From the part of Theorem~\ref{T:main7} proved in \S\ref{SS:main proof 7}, we know that $\pm\rho_{E,7}(\Gal_\QQ)$ is conjugate to the group $G_1$ of \S\ref{SS:applicable 7}.    Let $\calH$ be the set of proper subgroups $H$ of $G_1$ such that $\pm H=G_1$.   The set $\calH$ consists of two groups; they are both conjugate in $\GL_2(\FF_7)$ to the group $H_{1,1}$ of \S\ref{SS:applicable 7}.   The curve $E$ is denoted by $\calE_1$ in \S\ref{SS:applicable 7}.  

We have $\calD_E \subseteq M_E = \{\pm 1,\pm 2, \pm 5, \pm 7, \pm 10, \pm 14, \pm 35, \pm 70\}$.   The primes $211$, $239$ and $337$ are congruent to $1$ modulo $\ell$.   One can check that 
\[
a_{211}(E)=16 \equiv 2 \pmod{7}, \quad a_{239}(E)=-5\equiv 2 \pmod{7}, \quad a_{337}(E)=-5 \equiv 2 \pmod{7}.
\]   
So if $d\in \calD_{337}$, then $\left(\tfrac{d}{211}\right)=1$, $\left(\tfrac{d}{239}\right)=1$ and $\left(\tfrac{d}{337}\right)=1$.   Checking the $d\in M_E$, we find that $\calD_{337} \subseteq \{1,-7\}$.   Since $|\calH|=2$, we deduce that $\calD_{E}=\{1,-7\}$.

Now let $E'/\QQ$ be any elliptic curve with $j$-invariant $3^3\cdot 5\cdot 7^5/2^7$.   Using Lemma~\ref{L:twist newer}, we deduce that $\rho_{E',7}(\Gal_\QQ)$ is conjugate to $G_1$ if and only if $E'$ is not isomorphic to $E$ or its quadratic twist by $-7$.   When $\rho_{E',7}(\Gal_\QQ)$ is not conjugate to $G_1$ it must be conjugate to $H_{1,1}$ in $\GL_2(\FF_7)$.

\subsubsection{}
Take $\ell=11$.   Let $G_1$, $H_{1,1}$ and $H_{1,2}$ be the groups from \S\ref{SS:applicable 11}.  The set $\calH$ of proper subgroups $H$ of $G_1$ for which $\pm H=G_1$ is equal to $\{H_{1,1}, H_{1,2}\}$.  

Let $E/\QQ$ be the elliptic curve defined by $y^2+xy+y= x^3+x^2-305x+7888$; it has $j$-invariant $-11^2$ and is isomorphic to the curve $\calE_1$ of \S\ref{SS:applicable 11}.   In \S\ref{SS:11 borel}, we showed that $\rho_{E,11}(\Gal_\QQ)$ and $\pm \rho_{E,11}(\Gal_\QQ)$ are conjugate in $\GL_2(\FF_{11})$ to $H_{1,1}$ and $G_1$, respectively.   

Using Lemma~\ref{L:ell twist D} and $|\calD_E|=|\calH|$, we deduce that $\calD_E = \{1,-11\}$.    Lemma~\ref{L:twist newer} implies that if $E'/\QQ$ has $j$-invariant $-11^2$, then $\rho_{E',11}(\Gal_\QQ)$ is not conjugate to $G_1$ if and  only if $E'$ is isomorphic to $E$ or its quadratic twist by $-11$.    If $E'$ is isomorphic to $E$ or its twist by $-11$, then $\rho_{E',11}(\Gal_\QQ)$ is conjugate in $\GL_2(\FF_{11})$ to $H_{1,1}$ or $H_{1,2}$, respectively.

\subsubsection{}

Take $\ell=11$.   Let $G_2$, $H_{2,1}$ and $H_{2,2}$ be the subgroups of $\GL_2(\FF_{11})$ from \S\ref{SS:applicable 11}.  The set $\calH$ of proper subgroups $H$ of $G_2$ for which $\pm H=G_2$ is equal to $\{H_{2,1}, H_{2,2}\}$.  

Let $E/\QQ$ be the elliptic curve defined by $y^2+xy= x^3+x^2-3632x+82757$; it has $j$-invariant $-11\cdot 131^3$ and is isomorphic to the curve $\calE_2$ of \S\ref{SS:applicable 11}.   In \S\ref{SS:11 borel}, we showed that $\rho_{E,11}(\Gal_\QQ)$ and $\pm \rho_{E,11}(\Gal_\QQ)$ are conjugate in $\GL_2(\FF_{11})$ to $H_{2,1}$ and $G_2$, respectively.   

Using Lemma~\ref{L:ell twist D} and $|\calD_E|=|\calH|$, we deduce that $\calD_E = \{1,-11\}$.   Lemma~\ref{L:twist newer} implies that if $E'/\QQ$ has $j$-invariant $-11\cdot 131^3$, then $\rho_{E',11}(\Gal_\QQ)$ is not conjugate to $G_2$ if and  only if $E'$ is isomorphic to $E$ or its quadratic twist by $-11$.    If $E'$ is isomorphic to $E$ or its twist by $-11$, then $\rho_{E',11}(\Gal_\QQ)$  is conjugate in $\GL_2(\FF_{11})$ to $H_{2,1}$ or $H_{2,2}$, respectively.

\section{Quadratic twists of families} \label{S:twists 2}

In this section, we complete the proof of the theorems from \S\ref{S:classification}.  

\subsection{General setting} \label{SS:twist setup}

Fix an integer $N\geq 3$ and an applicable subgroup $G$ of $\GL_2(\ZZ/N\ZZ)$.  Let $\calH$ be the set of proper subgroups $H$ of $G$ that satisfy $\pm H=G$.

  Assume that the morphism $\pi_G\colon X_G \to \PP^1_\QQ$ arises from a rational function $J(t)\in \QQ(t)$, i.e., the function field of $X_G$ is of the form $\QQ(h)$ where $j=J(h)$.  \\

Let $g(t)$ be a rational function $\QQ(t)$ such that
\[
a(t):=-3 g(t)^2 J(t)/(J(t)-1728)  \quad \text{ and }\quad b(t):=-2 g(t)^3 J(t)/(J(t)-1728)
\]
belong to $\ZZ[t]$, and for which there is no irreducible element $\pi$ of the ring $\ZZ[t]$ such that $\pi^2$ divides $a$ and $\pi^3$ divides $b$.  After possibly changing $g$ by a sign, we may assume that $g$ is the quotient of two polynomials with positive leading coefficient; the function $g(t)$ is now uniquely determined.   Define $\Delta:=-16(4a^3+27b^2)$; it is a polynomial in $\ZZ[t]$ and equals $2^{12} 3^6 J(t)^2/(J(t)-1728)^3 g(t)^6$.   Let $\scrM$ be the set of squarefree $f(t) \in \ZZ[t]$ which divide $N \Delta(t)$.\\   

Take any $u \in \QQ$ for which $J(u)\notin \{0,1728,\infty\}$.  We have $\Delta(u)\neq 0$ and hence $f(u)\neq 0$ for all $f\in \scrM$.  Let $E_u/\QQ$ be the elliptic curve defined by the Weierstrass equation $y^2= x^3+a(u) x + b(u)$; note that $\Delta(u)\neq 0$ since $J(u)\notin\{0,1728,\infty\}$.  One can readily check that the curve $E_u$ has $j$-invariant $J(u)$.    \emph{Warning}: this is not to be confused with the quadratic twist notation we used in \S\ref{S:twist 1}.

\begin{prop} \label{P:fH prop}
There is an injective map 
\[
\calH\to \scrM,\quad H\mapsto f_H
\] 
such that for any $u\in \QQ$ with $J(u)\notin \{0,1728,\infty\}$ and $\pm \rho_{E_u,N}(\Gal_\QQ)$ conjugate to $G$ in $\GL_2(\ZZ/N\ZZ)$, the following hold:
\begin{alphenum}
\item \label{P:twists H a}
If $E'/\QQ$ is an elliptic curve with $j$-invariant $J(u)$, then  $\rho_{E',N}(\Gal_\QQ)$ is conjugate to $G$ in $\GL_2(\ZZ/N\ZZ)$ if and only if $E'$ is not isomorphic to the quadratic twist of $E_u$ by $f_H(u)$ for all $H\in \calH$.
\item \label{P:twists H b}
If $E'/\QQ$ is isomorphic to the quadratic twist of $E_u$ by $f(u)$ for some $H\in \calH$, then $\rho_{E',N}(\Gal_\QQ)$ is conjugate to $H$ in $\GL_2(\ZZ/N\ZZ)$.
\end{alphenum}
The sets $\{f_H(u): H \in \calH\}$ and $\calD_{E_u}$ represent the same cosets in $\QQ^\times/(\QQ^\times)^2$, with $\calD_{E_u}$ defined as in \S\ref{S:twist 1}.

\end{prop}
\begin{proof}
Define the scheme $U:=\Spec \ZZ[t, N^{-1}, \Delta(t)^{-1}]$.  By taking the square root of a polynomial $f\in \scrM$, we obtain an \'etale extension of $U$ of degree $1$ or $2$; we denote the corresponding quadratic character by  $\chi_f\colon \pi_1(U) \to \{\pm 1\}$. Conversely, every (continuous) character $\pi_1(U)\to\{\pm 1\}$ is of the form $\chi_f$ for a unique $f\in \scrM$.  (Note that $2$ always  divides $\Delta(t)$). 

The Weierstrass equation 
\[
y^2=x^3 + a(t)x+b(t)
\] 
defines a relative elliptic curve $E\to U$.   Let $E[N]$ be the $N$-torsion subscheme of $E$.   The morphism $E[N]\to U$ allows us to view $E[N]$ as a lisse sheaf of $\ZZ/N\ZZ$-modules on $U$ that is free of rank $2$.   The sheaf $E[N]$ then gives rise to a representation 
\[
\rho_N \colon \pi_1(U)\to \GL_2(\ZZ/N\ZZ)
\]  
that is uniquely defined up to conjugacy (we will suppress the base point in our fundamental group since we are only interested in $\rho_N$ up to conjugacy).  

We now consider specializations of $E$.   Take any $u\in U(\QQ)$, i.e., an element $u\in \QQ$ with $\Delta(u)\neq 0$.  One can show that elements $u \in U(\QQ)$ can also be described as those $u\in \QQ$ for which $J(u)\notin\{0,1728,\infty\}$.   We can specialize $E$ at $u$ to obtain the elliptic curve that we have denoted $E_{u}/\QQ$; it is defined by $y^2=x^3+a(u) x + b(u)$ and has $j$-invariant $J(u)$.    

Let $\rho_{u,N}\colon \Gal_\QQ \to \GL_2(\ZZ/N\ZZ)$ be the specialization of $\rho_N$ at $u$; it is obtained by composing the homomorphism $u_*\colon \Gal_\QQ \to \pi_1(U)$ coming from $u\in U(\QQ)$ with $\rho_N$.     The homomorphism $\rho_{u,N}$ agrees, up to conjugacy, with the representation $\rho_{E_u,N}$ that describes the Galois action on the $N$-torsion points of $E_{u}$.   So taking $\rho_{E_u,N}=\rho_{u,N}$, specialization gives an inclusion $\rho_{E_{u},N}(\Gal_\QQ) \subseteq \rho_N(\pi_1(U))$.  

{We claim that $\pm \rho_N(\pi_1(U))$ and $G$ are conjugate in $\GL_2(\ZZ/N\ZZ)$.   By Lemma~\ref{L:basic HIT}, the group $\pm \rho_{E_u,N}(\Gal_\QQ)$ is conjugate to $G$ in $\GL_2(\ZZ/N\ZZ)$ for ``most'' $u\in \QQ$.  By Hilbert's irreducibility theorem, the group $\pm  \rho_{E_{u},N}(\Gal_\QQ)$ equals  $\pm \rho_N(\pi_1(U))$ for ``most'' $u\in \QQ$.   This proves the claim.}\\

We may thus assume that $G=\pm \rho_N(\pi_1(U))$ and hence we have a representation $\rho_N\colon \pi_1(U)\to G$.    Specializations thus give inclusions $\rho_{E_{u},N}(\Gal_\QQ) \subseteq G$.    Take any $H \in \calH$ and let  $\chi_H\colon \pi_1(U) \to\{\pm 1\}$ be the character obtained by composing $\rho_N$ with the quotient map $G\to G/H\cong \{\pm 1\}$.    We thus have $\chi_H=\chi_{f_H}$ for a unique polynomial $f_H\in \scrM$.   

Specializing $\chi_H$ at $u$, we obtain the character $\chi_{E_u,H}\colon \Gal_\QQ \to \{\pm 1\}$ from \S\ref{S:twist 1}.  With notation as in \S\ref{S:twist 1}, we find that the integer $d_{E_u,H}$ lies in the same class in $\QQ^\times/(\QQ^\times)^2$ as $f_H(u)$.   Therefore, the classes of $\calD_{E_u}$ in $\QQ^\times/(\QQ^\times)^2$ are represented by the set $\{f_H(u): H \in \calH\}$.  Parts (\ref{P:twists H a}) and (\ref{P:twists H b}) are now immediate consequences of Lemma~\ref{L:twist newer}.
\end{proof}

We claim that the set of polynomials 
\[
\scrF:=\{f_H : H \in \calH\}
\] 
is uniquely determined and has cardinality $|\calH|$.     By Hilbert irreducibility, one can chose $u\in U(\QQ)$ such that $\pm \rho_{E_u,N}(\Gal_\QQ)=G$ and such that the map $\scrM \to \QQ^\times/(\QQ^\times)^2$, $f\mapsto f(u)\cdot (\QQ^\times)^2$ is injective.   The uniqueness of $\scrF$ then follows from part (\ref{P:twists H a}) of Proposition~\ref{P:fH prop}.

\subsection{Computing $\scrF$}

We now focus on the case where $N$ is a prime $\ell \in \{3,5,7,13\}$.    Fix notation as in the subsection of \S\ref{S:classification} for the given $\ell$.   \\

Let $G$ be one of the subgroups $G_i$ of $\GL_2(\FF_\ell)$  in \S\ref{S:classification} for which there is a corresponding rational function $J(t):=J_i(t) \in \QQ(t)-\QQ$.   The group $G$ is applicable and in particular contains $-I$.

We take notation as in \S\ref{SS:twist setup}.  In particular, $\calH$ is the set of proper subgroups $H$ of $G$ such that $\pm H = G$.    We shall assume that $\calH\neq \emptyset$ (otherwise $\scrF=\emptyset$); this holds when
 \[
 (\ell,i) \in \big\{ (3,1),(3,3), (5,1), (5,5), (5,6), (7,1), (7,3), (7,4), (7,5), (7,7), (13,4), (13,5) \big\}.
\]
In each of these cases, one can check that $|\calH|=2$.   \\

We now explain how to compute the set $\scrF=\{f_H : H \in \calH\}$; it has cardinality $|\calH|=2$.    Take any $u \in \QQ$ with $J(u)\notin \{0,1728,\infty\}$ such that $J(u) \notin J_j(\QQ)$ for all $j < i$.   From the parts of the main theorems proved in \S\ref{S:main classification}, this implies that  $\pm \rho_{E_u,\ell}(\Gal_\QQ)$ is conjugate to $G$.    Let $\calD_{E_u}$ be the (computable!) set from \S\ref{S:twist 1}.   From Proposition~\ref{P:fH prop}, we find that 
\begin{equation} \label{E:scrF inclusion}
\scrF \subseteq \{ f \in \scrM : f(u) \in d(\QQ^\times)^2 \text{ for some } d \in \calD_{E_u} \}.
\end{equation}

By considering (\ref{E:scrF inclusion}) with many such $u\in \QQ$, one is eventually left with only $|\calH|$ candidates $f\in \scrF$ to be of the form $f_H$; this then produces the set $\{f_H: H\in \calH\}$ of order $|\calH|$  (for our examples, one only needs to check $u \in \{1,2,3,4\}$).   One could also work with a single $u\in \QQ$ chosen so that the map $\scrF\to \QQ^\times/(\QQ^\times)^2$, $f\mapsto f(u) \cdot (\QQ^\times)^2$ is injective.   This method thus produces $\scrF$.\\

Doing the above computations, we find that 
\[
\{f_H : H \in \calH\} = \{ f_1, \ell^* f_1\}
\] 
for a unique polynomial $f_1\in \scrF$, where $\ell^* := (-1)^{(\ell-1)/2} \cdot \ell$; this can also be deduced from $|\calH|=2$ and Lemma~\ref{L:ell twist D}.  We thus have $f_1 = f_{M_1}$ and $\ell^* f_1 = f_{M_2}$, where $\calH=\{M_1,M_2\}$.    

Let $h$ be the largest element of $\ZZ[t]$, in terms of divisibility, with positive leading coefficient such that $h^4$ divides $af_1^2$ and $h^6$ divides $b f_1^3$; define $A:= (af_1^2)/h^4$ and $B:=(bf_1^3)/h^6$ in $\ZZ[t]$. The Weierstrass equation 
\[
y^2= x^3+A(t) x + B(t)
\]
is precisely the equation given for $\calE_{i,t}$ in the subsection of \S\ref{S:classification} corresponding to the prime $\ell$.  (For code verifying these claims, see the link given in \S\ref{SS:overview}.)
\\

 For $u\in \QQ$ with $J(u)\notin \{0,1728,\infty\}$, let $\calE_{i,u}$ be the elliptic curve over $\QQ$ defined by setting $t$ equal to $u$.  Let $E'/\QQ$ be any elliptic curve with $j_{E'}\notin \{0,1728\}$ for which $\pm \rho_{E',\ell}(\Gal_\QQ)$ is conjugate to $G$ in $\GL_2(\FF_\ell)$.   From the parts of the main theorems proved in \S\ref{S:main classification}, we have $j_{E'}=J(u)$ for some $u\in \QQ$.   The curve $E_u/\QQ$ also has $j$-invariant $J(u)$.    The twist of $E_u$ by $f_1(u)$ is isomorphic to the the curve $\calE_{i,u}/\QQ$.   By Proposition~\ref{P:fH prop}, we deduce that $\rho_{E',\ell}(\Gal_\QQ)$ is conjugate to $G$ if and only if $E'$ is not isomorphic to $\calE_{i,u}$ and not isomorphic to the quadratic twist of $\calE_{i,u}$ by $\ell^*$.   By Proposition~\ref{P:fH prop}, $\rho_{E',\ell}(\Gal_\QQ)$ is conjugate to $M_1$ or $M_2$ when $E'$ is isomorphic to $\calE_{i,u}$ or the quadratic twist of $\calE_{i,u}$ by $\ell^*$, respectively.

It thus remains to determine $M_1$ and $M_2$.\\ 

If $(\ell,i) \in \{(3,1), (7,1)\}$, then $M_1$ and $M_2$ are both conjugate to $H_{i,1}$ since the two groups in $\calH$ are conjugate in $\GL_2(\FF_\ell)$.   We shall now assume that $(\ell,i) \notin \{(3,1), (7,1)\}$.  We then have $\calH=\{H_{i,1},H_{i,2}\}$.    It thus remains to prove that $M_1=H_{i,1}$ (and hence $M_2=H_{i,2}$).  \\

Suppose that $(\ell,i) \in \big\{ (5,1), (5,5), (5,6), (7,3), (7,4), (13,4), (13,5) \}$.    Take $u$, $p$ and $a$ as in Table 2 below for the pair $(\ell,i)$.   
{
\renewcommand{\arraystretch}{1.1}
\begin{table}[htdp] 
\begin{center}\begin{tabular}{c||c|c|c|c|c|c|c}%\hline  
$(\ell,i)$ & $( 5, 1 )$ & $( 5, 5 )$ & $( 5, 6 )$ & $(7, 3)$ & $(7, 4)$  & $(13, 4)$ & $(13, 5)$   \\ \hline%\hline
$u$ & $1$ &  $2$ & $1$ &  $2$ &$2$ & $1$ &  $1$ \\
$p$ &  $2$ &  $3$ & $2$ & $3$ & $3$ &$2$  &   $2$ \\
$a$ &  $-2$ & $-1$ & $-2$ & $-3$ & $-3$ & $2$ & $2$\\
%\hline 
\end{tabular} \caption{}
\end{center}
\end{table}
}

The element $u\in \QQ$ is chosen so that $J_i(u) \notin \{0,1728,\infty\}$ and such that $J_i(u) \notin J_j(\QQ \cup \{\infty\})$ for all $j<i$.     Define the elliptic curve $E:=\calE_{i,u}/\QQ$.  By our choice of $u$, the group $\rho_{E,\ell}(\Gal_\QQ)$ is conjugate in $\GL_2(\FF_\ell)$ to $M_1$.

The curve has good reduction at the prime $p$ and we have $a=a_p(E)$.   Let $t_p$ be the image of $(a,p)$ in $\FF_\ell^2$; it equals $(\tr(A),\det(A))$ with $A:=\rho_{E,\ell}(\Frob_p) \in M_1$.   A direct computation shows that $t_p \notin \{ (\tr(A),\det(A)): A \in H_{i,2}\}$.   Therefore, $M_1$ is not conjugate to $H_{i,2}$.   So $M_1$ must be conjugate to $H_{i,1}$ and hence $M_2$ is conjugate to $H_{i,2}$.\\

Finally, consider the remaining pairs $(\ell,i) \in \{(3,3),(7,5),(7,7)\}$.

 Consider $(\ell,i)=(3,3)$.  The pair $(3(u+1)^2, 4u(u+1)^2)$ is a point of order $3$ of $\calE_{3,u}$ for all $u$.   This implies that $M_1$ is conjugate in $\GL_2(\FF_3)$ to a subgroup of $\left(\begin{smallmatrix}1 & * \\0 & * \end{smallmatrix}\right)$.   So $M_1\neq H_{i,2}$ and hence $M_1=H_{i,1}$.

We may now suppose that $\ell=7$ and $i\in\{5,7\}$.

Take $i=5$.  Let $E'/\QQ$ be the elliptic curve defined by $y^2=x^3-2835(-7)^2 x-71442(-7)^3$; it is the quadratic twist of $\calE_{5,0}$ by $-7$.   Using Theorem~\ref{T:main7}(ii), which we proved in \S\ref{S:main classification}, we find that $\pm \rho_{E',7}(\Gal_\QQ)$ is conjugate to $G_5$.   The group $\rho_{E',7}(\Gal_\QQ)$ is thus conjugate to $M_2$.    So to prove that $M_2=H_{i,2}$, and hence $M_1=H_{i,1}$, we need only verify that $E'$ has a $7$-torsion point defined over some cubic field.   Let $w\in \Qbar$ be a root of the irreducible polynomial $x^3-441x^2-83349x+22754277$.    The pair $(w, 21w -1323)$ is a point of order $7$ on $E'$.

Finally, take $i=7$.  Let $E'/\QQ$ be the elliptic curve defined by $y^2=x^3-17870609043(-7)^2 x-919511455160466(-7)^3$; it is the quadratic twist of $\calE_{7,1}$ by $-7$.   Using Theorem~\ref{T:main7}(ii), which we proved in \S\ref{S:main classification}, we find that $\pm \rho_{E',7}(\Gal_\QQ)$ is conjugate to $G_7$.   The group $\rho_{E',7}(\Gal_\QQ)$ is thus conjugate to $M_2$.    So to prove that $M_2=H_{i,2}$, and hence $M_1=H_{i,1}$, we need only verify that $E'$ has a $7$-torsion point defined over some cubic field.   Let $w\in \Qbar$ be a root of the irreducible polynomial $x^3 - 1750329x^2 + 1015924207851x - 195667237639563291$.    The pair $(w, 1323w - 714884373)$ is a point of order $7$ on $E'$.

\section{Proof of Propositions from \S\ref{SS:CM}}  \label{SS:CM proofs}

Let $E$ be an elliptic curve defined over $\QQ$ that has complex multiplication.  Let $R$ be the ring of endomorphisms  of $E_{\Qbar}$.  Let $k \subseteq \Qbar$ be the minimal extension of $\QQ$ over which all the endomorphisms of $E_{\Qbar}$ are defined; it is an imaginary quadratic field.   Moreover, we can identify $k$ with $R\otimes_\ZZ \QQ$ (the action of $R$ on the Lie algebra of $E_k$ gives a ring homomorphism $R\to k$ that extends to an isomorphism $R\otimes_\ZZ \QQ \to k$).   The field $k$ has discriminant $-D$.

Take any \emph{odd} prime $\ell$.   For each integer $n\geq 1$, let $E[\ell^n]$ be the $\ell^n$-torsion subgroup of $E(\Qbar)$.  The $\ell$-adic \defi{Tate module} $T_\ell(E)$ of $E$ is the inverse limit of the groups $E[\ell^n]$ with multiplication by $\ell$ giving transition maps $E[\ell^{n+1}]\to E[\ell]$; it is a free $\ZZ_\ell$-module of rank $2$.   The natural Galois action on $T_\ell(E)$ can be expressed in terms of a representation
\[
\rho_{E,\ell^\infty} \colon \Gal_k \to \Aut_{\ZZ_\ell}(T_\ell(E)).
\]
The ring $R$ acts on each of the $E[\ell^n]$ and this induces a faithful action of $R$ on $T_\ell(E)$.   

The Tate module  $T_\ell(E)$ is actually a free module over $R_\ell:=R\otimes_\ZZ \ZZ_\ell$ of rank $1$ (see the remarks at the end of \S4 of \cite{MR0236190}).  We can thus make an identification $\Aut_{R_\ell}(T_\ell (E)) = R_\ell^\times$.   The actions of $\Gal_k=\Gal(\Qbar/k)$ and $R_\ell$ on $T_\ell(E)$ commute, so the restriction of $\rho_{E,\ell^\infty}$ to $\Gal_k$ gives a representation
\[
\Gal_k \to \Aut_{R_\ell}(T_\ell (E))=R_\ell^\times.
\]

\begin{lemma} \label{L:CM ell-adic image}   
\begin{romanenum}
\item \label{L:CM ell-adic image a}
If $E$ has good reduction at $\ell$, then $\rho_{E,\ell^\infty}(\Gal_k) = R_\ell^\times$.
\item \label{L:CM ell-adic image b}
If $j_E\neq 0$, then $\rho_{E,\ell^\infty}(\Gal_k)$ is an open subgroup of $R_\ell^\times$ whose index is a power of $2$.
\end{romanenum}
\end{lemma}
\begin{proof}
Since $R_\ell^\times$ is commutative, we can factor $\rho_{E,\ell^\infty}|_{\Gal_k}$ through the maximal abelian quotient of $\Gal_k$.   Composing with the reciprocity map of class field theory, we obtain a continuous representation $\varrho_{E,\ell^\infty}\colon \AA_k^\times \to R_\ell^\times$, where $\AA_k^\times$ is the group of ideles of $k$.  Define  $k_\ell := k \otimes_\ZZ \QQ_\ell = {\prod}_{v | \ell} k_v$, where the product is over the places $v$ of $k$ lying over $\ell$ and $k_v$ is the completion of $k$ at $v$.   For an idele $a\in \AA_k^\times$, let $a_\ell$ be the component of $a$ in $k_\ell^\times$.   From \cite{MR0236190}*{Theorems 10 \& 11}, there is a unique homomorphism $\varepsilon\colon \AA_k^\times \to k^\times$ such that $\varrho_{E,\ell^\infty}(a) = \varepsilon(a) a_\ell^{-1}$ for $a\in \AA_k^\times$.   The homomorphism $\varepsilon$ satisfies $\varepsilon(x)=x$ for all $x\in k^\times$ and its kernel is open in $\AA_k^\times$.  We identify $R_\ell^\times = \prod_{v|\ell} \OO_v^\times$, where $\OO_v$ is the valuation ring of $k_v$, with a subgroup of $\AA^\times_k$ (by letting the coordinates at the places $v\nmid \ell$ of $k$ be $1$).    Let $B$ be the kernel of $\varepsilon|_{R_\ell^\times}$.      

First suppose that $E$ has good reduction at $\ell$, and hence at all places $v|\ell$ of $k$.   By the first corollary of Theorem~11 in \cite{MR0236190}, we deduce that that $\varepsilon$ is unramified at all $v | \ell$.  Therefore, $B=R_\ell^\times$ and hence $\varrho_{E,\ell^\infty}(R_\ell^\times)=R_\ell^\times$.    Therefore, $\rho_{E,\ell^\infty}(\Gal_k)$ contains, and hence is equal to, $R_\ell^\times$.

Now suppose that $j_E \neq 0$.   Since $\ell$ is odd and $j_E\neq 0$, the subgroup of $R[\ell^{-1}]^\times$ consisting of roots of unity has order $2$ or $4$.   By Theorem~11(ii) and Theorem~6(b) in \cite{MR0236190}, we find that $B$ is an open subgroup of $R_\ell^\times$ with index a power of $2$.   So $\varrho_{E,\ell^\infty}(B)=B$ and hence $\rho_{E,\ell^\infty}(\Gal_k) \supseteq B$.   Therefore, $\rho_{E,\ell^\infty}(\Gal_k)$ is an open subgroup of $R_\ell^\times$ whose index is a power of $2$.
\end{proof}

The following gives constraints on the elements of $\rho_{E,\ell^\infty}(\Gal_\QQ-\Gal_k)$.   Since $R$ is a quadratic order, there is an element $\beta \in R-\ZZ$ such that $\beta^2\in \ZZ$; note that $\beta$ is not defined over $\QQ$.  We can view $\beta$ as an endomorphism of $T_\ell(E)$.

\begin{lemma} \label{L:non-comm of R}
For any $\sigma \in \Gal_\QQ - \Gal_k$, we have $\rho_{E,\ell^\infty}(\sigma) \beta=-\beta \rho_{E,\ell^\infty}(\sigma)$ and $\tr(\rho_{E,\ell^\infty}(\sigma))=0$.
\end{lemma}
\begin{proof}
Take any $\sigma \in \Gal_\QQ - \Gal_k$.   The group $\Gal_\QQ$ acts on $R$ and we have $\sigma(\beta)=-\beta$ since $\beta^2 \in \ZZ$ and $\beta$ is not defined over $\QQ$ (but is defined over $k$).  So for each $P\in E[\ell^n]$, we have $\sigma( \beta(P))=\sigma(\beta)(\sigma(P))=-\beta(\sigma(P))$.  Taking an inverse limit, we deduce that $\rho_{E,\ell^\infty}(\sigma) \beta=-\beta \rho_{E,\ell^\infty}(\sigma)$.      In $\Aut_{\QQ_\ell}(T_\ell(E)\otimes_{\ZZ_\ell} \QQ_\ell) \cong \GL_2(\QQ_\ell)$, we have $\rho_{E,\ell^\infty}(\sigma) =-\beta \rho_{E,\ell^\infty}(\sigma)\beta^{-1}$.   Taking traces we deduce that $\tr(\rho_{E,\ell^\infty}(\sigma))=-\tr(\rho_{E,\ell^\infty}(\sigma))$ and hence $\tr(\rho_{E,\ell^\infty}(\sigma))=0$.
\end{proof}

\begin{lemma}  \label{L:EDf generics}
Suppose that $\ell \nmid D$ and that $E$ has good reduction at $\ell$.
\begin{romanenum}
\item
If $\ell$ splits in $k$, then $\rho_{E,\ell}(\Gal_\QQ)$ is conjugate in $\GL_2(\FF_\ell)$ to $N_s(\ell)$.
\item
If $\ell$ is inert in $k$, then $\rho_{E,\ell}(\Gal_\QQ)$ is conjugate in $\GL_2(\FF_\ell)$ to $N_{ns}(\ell)$.
\end{romanenum}
\end{lemma}
\begin{proof}
Lemma~\ref{L:CM ell-adic image} implies that the group $C:=\rho_{E,\ell}(\Gal_k)$ is isomorphic to $(R/\ell R)^\times$.   The ring $R/\ell R$ is isomorphic to $\FF_\ell \times \FF_\ell$ or $\FF_{\ell^2}$ when $\ell$ splits or is inert in $k$, respectively.    Therefore, $C$ is a Cartan subgroup of $\GL_2(\FF_\ell)$; it is split if and only if $\ell$ splits in $k$.    Let $N$ be the normalizer of $C$ in $\GL_2(\FF_\ell)$.     The group $C=\rho_{E,\ell}(\Gal_k)$ is normal in $\rho_{E,\ell}(\Gal_\QQ)$ since $k/\QQ$ is a Galois extension, so $\rho_{E,\ell}(\Gal_\QQ) \subseteq N$.  

It remains to show that $\rho_{E,\ell}(\Gal_\QQ)= N$.  Suppose that $\rho_{E,\ell}(\Gal_\QQ) \neq N$, and hence $\rho_{E,\ell}(\Gal_\QQ)=C=\rho_{E,\ell}(\Gal_k)$.  This implies that the actions of $\Gal_\QQ$ and $R$ on $E[\ell]$ commute.  However, this contradicts Lemma~\ref{L:non-comm of R} which implies that the actions of $\sigma \in \Gal_\QQ-\Gal_k$ and $\beta$ on $E[\ell]$ anti-commute.   Therefore, $\rho_{E,\ell}(\Gal_\QQ) = N$.
\end{proof}

We now describe the commutator subgroup of the normalizer $N$ of a Cartan subgroup $C$ of $\GL_2(\FF_\ell)$.   Let $\varepsilon \colon N \to N/C \cong \{\pm 1\}$ be the quotient map, and define the homomorphism
\[
\varphi\colon N \to \{\pm 1\} \times \FF_\ell^\times,\quad A \mapsto (\varepsilon(A),\det(A));
\]
it is surjective.

\begin{lemma} \label{L:N group theory}
\begin{romanenum}
\item \label{L:N group theory i}
The commutator subgroup of $N$ is $\ker \varphi$, i.e., the subgroup of $C$ consisting of matrices with determinant $1$.  
\item \label{L:N group theory ii}
If $H$ is a subgroup of $N$ satisfying $\pm H = N$, then $H=N$.
\end{romanenum}
\end{lemma}
\begin{proof}
The kernel of $\varphi$ contains the commutator subgroup of $N$ since the image of $\varphi$ is abelian.   It suffices to show that every element in $\ker \varphi$ is a commutator.  If $N$ is conjugate to $N_{s}(\ell)$, this is immediate since 
$\left(\begin{smallmatrix}0 & 1 \\1 & 0 \end{smallmatrix}\right)\left(\begin{smallmatrix}1 & 0 \\0 & a \end{smallmatrix}\right)\left(\begin{smallmatrix}0 & 1 \\1 & 0 \end{smallmatrix}\right)^{-1}\left(\begin{smallmatrix}1 & 0 \\0 & a \end{smallmatrix}\right)^{-1}=\left(\begin{smallmatrix}a & 0 \\0 & a^{-1} \end{smallmatrix}\right)$.

We now consider the non-split case.  We may take $C=C(\ell)$, $N=N(\ell)$ and with the explicit $\epsilon\in \FF_\ell^\times$ as given in the notation section of \S\ref{S:classification}.  Fix $\beta\in \FF_{\ell^2}$ for which $\beta^2 = \epsilon$.   The map $C(\ell)\to \FF_{\ell^2}^\times$, $\left(\begin{smallmatrix}a & b\epsilon \\b & a \end{smallmatrix}\right) \to a+b \beta$ is a group isomorphism.   Fix any $B \in N(\ell)-C(\ell)$.  One can check that the map
\begin{align} \label{E:actually norm}
C(\ell) \to C(\ell),\quad A \mapsto BAB^{-1} A^{-1}
\end{align}
corresponds to the homomorphism $\FF_{\ell^2}^\times \to \FF_{\ell^2}^\times$,  $\alpha \mapsto \alpha^{\ell-1}$.   In particular, the image of the map (\ref{E:actually norm}) is the unique (cyclic) subgroup of $C(\ell)$ of order $\ell+1$; these are the matrices in $C(\ell)$ with determinant $1$.  This completes the proof of (\ref{L:N group theory i}).

Finally, let $H$ be a subgroup of $N$ satisfying $\pm H=N$.  The group $H$ is normal in $N$ and $N/H$ is abelian, so $H$ contains the commutator subgroup of $N$.   From (\ref{L:N group theory i}), the commutator subgroup of $N$, and hence $H$, contains $-I$.  Therefore, $H=\pm H=N$.
\end{proof}

\subsection{Proof of Proposition~\ref{P:CM main}(\ref{P:CM main a}) and (\ref{P:CM main b})}
Let $E/\QQ$ be an CM elliptic curve with $j_E \neq 0$.   The curve $E$ is thus a twist of one of the curves $E_{D,f}/\QQ$ from Table~1.   Take any odd prime $\ell \nmid D$.  The curve $E_{D,f}$ has good reduction at $\ell$.  By Lemma~\ref{L:EDf generics}, the group $\rho_{E_{D,f},\ell}(\Gal_\QQ)$ is the normalizer $N$ of a Cartan subgroup $C$ of $\GL_2(\FF_\ell)$. Also the Cartan subgroup $C$ is split or non-split if $\ell$ is split or inert, respectively, in $k$.  

First suppose that $j_E\neq 1728$.  Since $j_E\notin\{0,1728\}$, the curve $E$ is a quadratic twist of $E_{D,f}$.  As noted in the introduction, this implies that $\pm \rho_{E,\ell}(\Gal_\QQ)$ and $\pm \rho_{E_{D,f},\ell}(\Gal_\QQ)=N$ are conjugate in $\GL_2(\FF_\ell)$.    After first conjugating $\rho_{E,\ell}(\Gal_\QQ)$, we may assume that $N=\pm\rho_{E,\ell}(\Gal_\QQ)$.     By Lemma~\ref{L:N group theory}, we have $\rho_{E,\ell}(\Gal_\QQ)=N$.

Now suppose that $j_E=1728$.  Let $\mu_4$ be the group of $4$-th roots of unity in $R$.   The elliptic curve $E/\QQ$ can be defined by an equation of the form $y^2=x^3+dx$ for some non-zero integer $d$, i.e., $E$ is a \defi{quartic twist} of $E_{4,1}$.  There is thus a character $\alpha \colon \Gal_k \to \mu_4 \subseteq R^\times$ such that the representations $\rho_{E,\ell^\infty}$ and $\alpha\cdot \rho_{E_{4,1},\ell^\infty}\colon \Gal_k \to R_\ell^\times$ are equal.  We have $\rho_{E_{4,1},\ell^\infty}(\Gal_k)= R_\ell^\times$ by Lemma~\ref{L:CM ell-adic image}(\ref{L:CM ell-adic image a}), so the image of $\rho_{E,\ell^\infty}(\Gal_k)$ in $R_\ell^\times/\{\pm 1\}$ has index $1$ or $2$.    Therefore, the image of $\rho_{E,\ell}(\Gal_k)$ in $C/\{\pm I\}$ has index $1$ or $2$.    We have $\rho_{E,\ell}(\Gal_\QQ)\not\subseteq C$ since otherwise the actions of $\Gal_\QQ$ and $R$ on $E[\ell]$ would commute (which is impossible by Lemma~\ref{L:non-comm of R}).   Therefore, the image of $\rho_{E,\ell}(\Gal_\QQ)$ in $N/\{\pm I\}$ is an index $1$ or $2$ subgroup.    

The group $G:=\pm\rho_{E,\ell}(\Gal_\QQ)$ thus has index $1$ or $2$ in $N$.   Since $\rho_{E,\ell}(\Gal_k)\subseteq C$ and $\rho_{E,\ell}(\Gal_\QQ)\not\subseteq C$, the quadratic character $\varepsilon\circ \rho_{E,\ell} \colon \Gal_\QQ \to \{\pm 1\}$ corresponds to the extension $k=\QQ(i)$ of $\QQ$.     The homomorphism $\det\circ \rho_{E,\ell} \colon \Gal_\QQ \to \FF_\ell^\times$ is surjective and factors through $\Gal(\QQ(\zeta_\ell)/\QQ)$.  We have $\QQ(i)\cap \QQ(\zeta_\ell)=\QQ$ since $\ell$ is odd, so $\varphi(\rho_{E,\ell}(\Gal_\QQ))=\{\pm 1\} \times \FF_\ell^\times$ and hence $\varphi(G)=\{\pm 1\} \times \FF_\ell^\times$.       Since $[N:G]\leq 2$, the group $G$ is normal in $N$ with abelian quotient $N/G$.  In particular, $G$ contains the commutator subgroup of $N$.  By Lemma~\ref{L:N group theory}, we deduce that $G$ contains the kernel of $\varphi$.   Since $G$ contains the kernel of $\varphi$ and $\varphi(G)=\{\pm 1\} \times \FF_\ell^\times$, we have $G=N$.    By Lemma~\ref{L:N group theory}, we conclude that $\rho_{E,\ell}(\Gal_\QQ)=N$.

\subsection{Proof of Proposition~\ref{P:CM main}(\ref{P:CM main c})}

We first consider the elliptic curve $E=E_{D,f}$ over $\QQ$ from Table~1 with $D=\ell$, where $\ell$ is an odd prime and $j_E\neq 0$.  We have $k=\QQ(\sqrt{-\ell})$.

\begin{lemma}  \label{L:CM D eq ell}
The group $\pm \rho_{E,\ell}(\Gal_\QQ)$ is conjugate to $G$.
\end{lemma}
\begin{proof}
Let $\bbar\beta$ be the image of $f\sqrt{-D}$ in $R/\ell R$.  Since $\ell$ is odd, the $\FF_\ell$-module $R/\ell R$ has basis $\{\bbar\beta,1\}$ and $\bbar\beta^2=0$.   Using this basis, we find that $R/\ell R$ is isomorphic to the subring $A:=\FF_\ell\left(\begin{smallmatrix}0 & 1 \\0 & 0 \end{smallmatrix}\right) \oplus \FF_\ell\left(\begin{smallmatrix}1 & 0 \\0 & 1 \end{smallmatrix}\right)$ of $M_2(\FF_\ell)$.   Using that $\rho_{E,\ell^\infty}(\Gal_k) \subseteq R_\ell^\times$, we deduce that $\rho_{E,\ell}(\Gal_k)$ is conjugate in $\GL_2(\FF_\ell)$ to a subgroup of $A^\times$.  We may thus assume that 
\[
\rho_{E,\ell}(\Gal_k) \subseteq A^\times = \{ \left(\begin{smallmatrix}a & b \\0 & a \end{smallmatrix}\right) : a\in \FF_\ell^\times, b\in \FF_\ell\}.
\]
By Lemma~\ref{L:CM ell-adic image}(\ref{L:CM ell-adic image b}), we deduce that $[A^\times:\rho_{E,\ell}(\Gal_k)]$ is a power of $2$ and hence $\rho_{E,\ell}(\Gal_k)$ contains the order $\ell$ group $\langle \left(\begin{smallmatrix}1 & 1 \\0 & 1 \end{smallmatrix}\right) \rangle$.   The order of $\rho_{E,\ell}(\Gal_k)$ is divisible by $(\ell-1)/2$ since $\det(\rho_{E,\ell}(\Gal_k))=(\FF_\ell^\times)^2$, so $\rho_{E,\ell}(\Gal_k)$ contains $\{ \left(\begin{smallmatrix}a & b \\0 & a \end{smallmatrix}\right) : a\in (\FF_\ell^\times)^2, b\in \FF_\ell\}$.  Therefore, $\pm \rho_{E,\ell}(\Gal_k)=A^\times$ since $-1$ is not a square in $\FF_\ell$ (we have $\ell\equiv 3\pmod{4}$).

Fix any $\sigma\in \Gal_\QQ-\Gal_k$. The matrix $g = \rho_{E,\ell}(\sigma)$ is upper triangular since the Borel subgroup $B(\ell)$ is the normalizer of $A^\times=\pm\rho_{E,\ell}(\Gal_k)$ in $\GL_2(\FF_\ell)$.   We have $\tr(g)=0$ by Lemma~\ref{L:non-comm of R}, so $g=\left(\begin{smallmatrix}a & b \\0 & -a \end{smallmatrix}\right)$ for some $a\in \FF_\ell^\times$ and $b\in \FF_\ell$.   The group $\pm \rho_{E,\ell}(\Gal_\QQ)$ is generated by $g$ and $A^\times$ and is thus $G$.
\end{proof}

The subgroups $H$ of $G$ that satisfy $\pm H = G$ are $H_1$, $H_2$ and $G$.    

\begin{lemma} \label{L:H1new}
The groups $\rho_{E,\ell}(\Gal_\QQ)$ and $H_1$ are conjugate in $\GL_2(\FF_\ell)$.
\end{lemma}
\begin{proof}
By Lemma~\ref{L:CM D eq ell}, we may assume that $\pm\rho_{E,\ell}(\Gal_\QQ)=G$.   There are thus unique characters $\psi_1,\psi_2\colon \Gal_\QQ \to \FF_\ell^\times$ such that $\rho_{E,\ell}= \left(\begin{smallmatrix} \psi_1 & * \\0 & \psi_2 \end{smallmatrix}\right)$.  Let $f \in \QQ[x]$ be the $\ell$-th division polynomial of $E/\QQ$; it is a polynomial of degree $(\ell^2-1)/2$ whose roots in $\Qbar$ are the $x$-coordinates of the non-zero points in $E[\ell]$.   Since $\pm \rho_{E,\ell}(\Gal_\QQ)= G$, we find that $f=f_1 f_2$ where the polynomials $f_1,f_2\in \QQ[x]$ are irreducible, and $f_1$ has degree $(\ell-1)/2$.  We may take $f_1$ so that it is monic.  Take any root $a\in \Qbar$ of $f_1$ and choose a point $P=(a,b)$ in $E[\ell]$.  We have $\sigma(P)=\psi_1(\sigma)P$ for all $\sigma \in \Gal_\QQ$.   Therefore, $\QQ(a,b)$ is the fixed field in $\Qbar$ of $\ker \psi_1$.

Suppose that $\ell=3$.   The point $(3,-2)$ of $E_{3,2}$ has order $3$.  The point $(12, -4)$ of $E_{3,3}$ has order $3$.  Therefore, $\QQ(a,b)=\QQ$.

Suppose that $\ell=7$.   For the curve $E_{7,1}$ we have computed that $f_1=x^3 - 441x^2 + 59339x - 2523451$.   If $a\in \Qbar$ is a root of $f_1$, then one can check that $(a,-7a+ 49)$ belongs to $E$.  For the curve $E_{7,2}$ we have computed that $f_1=x^3 - 49x^2 - 1029x + 31213$.   If $a\in \Qbar$ is a root of $f_1$, then one can check that $(a,21a - 2107)$ belongs to $E$.  In both cases, we have $[\QQ(a,b):\QQ]=3$.

Suppose that $\ell >7$.  Dieulefait, Gonz\'alez-Jim\'enez and Jim\'enez-Urroz have computed $\QQ(a,b)$ and found it to be equal to the maximal totally real subfield $\QQ(\zeta_\ell)^+$ of $\QQ(\zeta_\ell)$, cf.~Lemma~4 of \cite{MR2775372}.   They also give a link to files containing an explicit polynomial $f_1$.  In particular, $[\QQ(a,b):\QQ]=(\ell-1)/2$.     (However, note that the conclusions on the image of $\rho_{E,\ell}$ in Proposition~9 of  \cite{MR2775372} are not correct.)    

In all cases, the image of $\psi_1$ has order $[\QQ(a,b):\QQ]=(\ell-1)/2$, so the group $\rho_{E,\ell}(\Gal_\QQ)$ cannot be $G$ or $H_2$.   Therefore, $\rho_{E,\ell}(\Gal_\QQ)=H_1$.
\end{proof}

Take any elliptic curve $E'/\QQ$ with the same $j$-invariant as $E=E_{D,f}$; it is a quadratic twist.    Now take $\calD_E$ as in \S\ref{S:twist 1}.   Since $\#\calD_E=\#\{H_1, H_2\}=2$, we deduce from Lemma~\ref{L:ell twist D} (and $\ell\equiv 3\pmod 4$) that $\calD_{E}=\{1,-\ell\}$.    

 Since $\calD_{E}=\{1,-\ell\}$, we deduce from Lemma~\ref{L:H1new} that if $E'/\QQ$ is not isomorphic to $E$ or its quadratic twist by $-\ell$, then $\rho_{E',\ell}(\Gal_\QQ)$ is conjugate to $\pm\rho_{E,\ell}(\Gal_\QQ)=\pm H_1 =G$.    If $E'$ is isomorphic to $E$ or its quadratic twist by $-\ell$, then $\rho_{E,\ell}(\Gal_\QQ)$ is conjugate to $H_1$ or $H_2$, respectively.

\subsection{Proof of Proposition~\ref{P: prime 2}}

If $E/\QQ$ is given by $y^2=f(x)$ with $f(x)\in \QQ[x]$ a separable cubic, then $\rho_{E,2}(\Gal_\QQ)$ is isomorphic to the groups $\Gal(f)$, i.e., the Galois group of the splitting field of $f$ over $\QQ$.   Observe that $\GL_2(\FF_2)\cong \mathfrak{S}_3$.  It thus suffices to compute $\Gal(f)$ since the cardinality of a subgroup of $\mathfrak{S}_3$ determines it up to conjugacy.

For the $j_E=1728$ case, we have $f(x)=x^3-dx = x(x^2-d)$.   We have $\Gal(f)=\Gal(\QQ(\sqrt{d})/\QQ)$ which has order $1$ or $2$ when $d$ is a square or non-square, respectively.

For the $j_E=0$ case, we have $f(x)=x^3+d$.   We have $\Gal(f) = \Gal(\QQ(\sqrt[3]{d},\zeta_3)/\QQ)=\Gal(\QQ(\sqrt[3]{d},\sqrt{-3})$ which has order $2$ or $6$ when $d$ is a cube or non-cube, respectively.

For $j_E\notin \{0,1728\}$, the group $\rho_{E,2}(\Gal_\QQ)$ does not change if we replace $E$ by a quadratic twist (since $-I\equiv I \pmod{\ell}$), so one need only consider the specific curve $E=E_{D,f}$.  Using the $f(x)$ of Table~1, one can check that $\Gal(f)$ has order $2$ for the $j$-invariants listed in (\ref{P: prime 2 i}) and otherwise has order $6$.

Proposition~\ref{P: prime 2} is now a direct consequence of the above computations.

\subsection{Proof of Proposition~\ref{P:j=0 situation}}
Take any prime $\ell\geq 5$; we will deal with $\ell=3$ in \S\ref{L:j eq 0 ell eq 3}.   We first consider an elliptic curve $E_d/\QQ$ defined by the equation 
\[
y^2=x^3 +16d^2
\]
for a fixed \emph{cube-free} integer $d\geq 1$.    We have $R=\ZZ[\omega]$ and $k=\QQ(\omega)$, where $\omega:=(-1+\sqrt{-3})/2$ is a cube root of unity in $k$.   The ring $R$ is a PID.   

If $\ell$ is congruent to $1$ or $2$ modulo $3$, define  $C(\ell)$ be the Cartan subgroup $C_s(\ell)$ or $C_{ns}(\ell)$, respectively.  Let $N(\ell)$ be the normalizer of $C(\ell)$ in $\GL_2(\FF_\ell)$.

\begin{lemma} \label{L:new 1 or 3}
After replacing $\rho_{E_d,\ell}$ by a conjugate representation, we will have $\rho_{E_d,\ell}(\Gal_\QQ) \subseteq N(\ell)$ and $\rho_{E_d,\ell}(\Gal_k) \subseteq C(\ell)$ with 
\[
[N(\ell): \rho_{E_d,\ell}(\Gal_\QQ)]=[C(\ell): \rho_{E_d,\ell}(\Gal_k)] \in \{1,3\}.
\]
\end{lemma}
\begin{proof}
We have $E_1=E_{3,1}$.   By Lemma~\ref{L:EDf generics},  we have $\rho_{E_1,\ell}(\Gal_\QQ) = N(\ell)$.   The curves $E_d$ and $E_1$ are isomorphic over $\QQ(\sqrt[3]{d})$, so $\rho_{E_d,\ell}(\Gal_{\QQ(\sqrt[3]{d})})$ is conjugate to a subgroup of $N(\ell)$ of index $1$ or $3$.  Therefore, $\rho_{E_d,\ell}(\Gal_{\QQ})$ is conjugate to a subgroup of $N(\ell)$ of index $1$ or $3$.  Since $\rho_{E,d,\ell}(\Gal_\QQ) \not\subseteq C(\ell)$ and $\rho_{E,d,\ell}(\Gal_k) \subseteq C(\ell)$, we deduce that $[N(\ell): \rho_{E_d,\ell}(\Gal_\QQ)]=[C(\ell): \rho_{E_d,\ell}(\Gal_k)]$.
\end{proof}

To determine the index in Lemma~\ref{L:new 1 or 3}, we first compute some cubic residue symbols.  Recall that we have already defined a representation $\rho_{E_d,\ell^\infty}\colon \Gal_k \to R_\ell^\times$.  

\begin{lemma}  \label{L:cubic reciprocity ap}
Let $\lambda$ be a prime of $R$ dividing $\ell$ that satisfies $\lambda\equiv 2 \pmod{3R}$.   Take any non-zero prime ideal $\p \nmid 6d\ell$ of $R$.  We have $\p=R\pi$ for some $\pi\equiv 2 \pmod{3R}$.  Then
  \[
\legendre{\rho_{E_d,\ell^\infty}(\Frob_\p)}{\lambda}_3= \legendre{d^{\frac{2(\ell^e-1)}{3}}\, \lambda}{\pi}_3,
\]
where we are using cubic residue characters and the field $R/\lambda R$ has order $\ell^e$.
\end{lemma}
\begin{proof}
By Example~10.6 of \cite{MR1312368}*{II \S10}, we have $\rho_{E_d,\ell^\infty}(\Frob_\p) =  -  \bbar{\legendre{4\cdot 16 d^2}{\pi}}_6 \cdot \pi$, where we are using the $6$-th power residue symbol.  Therefore,
\[
\rho_{E_d,\ell^\infty}(\Frob_\p) = -  \bbar{\legendre{d}{\pi}}_6^2 \pi = -  \bbar{\legendre{d}{\pi}}_3 \pi=  - \legendre{d}{\pi}_3^2\cdot \pi
\]
and hence
\[
\legendre{\rho_{E_d,\ell^\infty}(\Frob_\p)}{\lambda}_3=  \legendre{-\legendre{d^2}{\pi}_3 \cdot \pi}{\lambda}_3 =\legendre{d^2}{\pi}^{\frac{\ell^e-1}{3}}_3 \legendre{\pi}{\lambda}_3 =\legendre{d^2}{\pi}^{\frac{\ell^e-1}{3}}_3 \legendre{\lambda}{\pi}_3 = \legendre{d^{\frac{2(\ell^e-1)}{3}} \lambda}{\pi}_3,
\]
where we have used cubic reciprocity.
\end{proof}

\begin{lemma} \label{L:j=0 case 1}
Suppose that $\ell \equiv 2 \pmod{3}$.   Then the group $\rho_{E_d,\ell}(\Gal_k)$ has index $3$ in $C(\ell)$ if and only if $\ell\equiv 2\pmod{9}$ and $d=\ell$, or $\ell\equiv 5\pmod{9}$ and $d=\ell^2$.   Note that $C(\ell)$ has a unique index $3$ subgroup.
\end{lemma}
\begin{proof}
Using Lemma~\ref{L:new 1 or 3} and $\ell\geq 5$, we find that $\rho_{E_d,\ell}(\Gal_k)$ is an index $3$ subgroup of $C(\ell)$ if and only if $\rho_{E_d,\ell^\infty}(\Gal_k)$ lies in a closed subgroup of $R_\ell^\times$ of index $3$.  We have $C(\ell)=C_{ns}(\ell)$ since $\ell \equiv 2 \pmod{3}$, so $R_\ell^\times$ has a unique index $3$ closed subgroup, i.e., the group of $a\in R_\ell^\times$ with $\legendre{a}{\ell}_3=1$.

By the Chebotarev density theorem and Lemma~\ref{L:cubic reciprocity ap} with $\lambda=\ell$, we deduce that $\rho_{E_d,\ell}(\Gal_k)$ is an index $3$ subgroup of $C(\ell)$ if and only if $d^{{2(\ell^2-1)}/{3}} \ell$ is a cube in $R/\p$ for all primes $\p \nmid 6d\ell$ of $R$; equivalently, $d^{2(\ell^2-1)/3} \ell$ is a cube in $R$.  Since $d^{{2(\ell^2-1)}/{3}} \ell$ is a rational integer, it is a cube in $R$ if and only if it is a cube in $\ZZ$.    

We have $2(\ell^2-1)/3 \equiv 2(\ell+1)/3 \pmod{3}$, so we need only determine when the integer $d^{2(\ell+1)/3} \ell$ is a cube.  In the following, we use that $d\geq 1$ is cube-free and that $\ZZ$ has unique factorization.   If $\ell = 2 +9m$, then $d^{2+ 6m}\ell$ is a cube if and only if $d=\ell$.  If $\ell = 5 +9m$, then $d^{4+ 6m}\ell$ is a cube if and only if $d=\ell^2$.  If $\ell = 8 +9m$, then $d^{6+ 6m}\ell$ is never a cube.
\end{proof}

\begin{lemma} \label{L:j=0 case 2}
Suppose that $\ell \equiv 1 \pmod{3}$.   Then the group $\rho_{E_d,\ell}(\Gal_k)$ has index $3$ in $C(\ell)$ if and only if $\ell\equiv 4 \pmod{9}$ and $d=\ell^2$, or $\ell\equiv 7\pmod{9}$ and $d=\ell$.

The group $\rho_{E_d,\ell}(\Gal_k)$ is conjugate to $C(\ell)=C_s(\ell)$ or the subgroup consisting of matrices of the form  $\left(\begin{smallmatrix}a & 0 \\0 & b \end{smallmatrix}\right)$ with $a/b \in \FF_\ell^\times$ a cube.
\end{lemma}
\begin{proof}
Using Lemma~\ref{L:new 1 or 3} and $\ell\geq 5$, we find that $\rho_{E_d,\ell}(\Gal_k)$ is an index $3$ subgroup of $C(\ell)$ if and only if $\rho_{E_d,\ell^\infty}(\Gal_k)$ lies in a closed subgroup of $R_\ell^\times$ of index $3$.    Let us describe the index $3$ subgroups of $R_\ell^\times$.      Since $\ell \equiv 1 \pmod{3}$, we have $\ell=\lambda_1 \lambda_2$ for irreducibles $\lambda_i\in R$ that we may choose to be congruent to $2$ modulo $3R$.   We have $R_\ell^\times = R_{\lambda_1}^\times \times R_{\lambda_2}^\times$.    The cubic residue symbol $\legendre{\cdot}{\lambda_i}$ defines a homomorphism $\varphi_i\colon R_\ell^\times \to \mu_3:=\langle\omega \rangle$.  Since $\ell\geq 5$, we find that every non-trivial homomorphism $R_\ell^\times \to \mu_3$ is of the form $\varphi_e:=\varphi_1^{e_1} \varphi_2^{e_2}$ with $e=(e_1,e_2) \in \{0,1,2\}^2-\{(0,0)\}$.     Therefore, $\rho_{E_d,\ell}(\Gal_k)$ is an index $3$ subgroup of $C(\ell)$ if and only if $\rho_{E_d,\ell^\infty}(\Gal_k)\subseteq \ker \varphi_e$ for some $e\neq (0,0)$.

   By Lemma~\ref{L:cubic reciprocity ap}, we have
 $\legendre{\rho_{E_d,\ell^\infty}(\Frob_\p)}{\lambda_i}_3= \legendre{d^{\frac{2(\ell-1)}{3}}\, \lambda_i}{\pi}_3$
and hence 
\begin{equation} \label{E:split case main}
\varphi_e(\rho_{E_d,\ell^\infty}(\Frob_\p)) =  \legendre{d^{\frac{2(\ell-1)}{3}}\, \lambda_1}{\pi}_3^{e_1} \legendre{d^{\frac{2(\ell-1)}{3}}\, \lambda_2}{\pi}_3^{e_2} = \legendre{d^{\frac{2(\ell-1)(e_1+e_2)}{3}}\, \lambda_1^{e_1} \lambda^{e_2}}{\pi}_3
\end{equation}
for all $\p\nmid 6d\ell$.  Using the Chebotarev density theorem, we deduce that $\rho_{E_d,\ell^\infty}(\Gal_k) \subseteq \ker \varphi_e$ if and only if $\beta:=d^{\frac{2(\ell-1)(e_1+e_2)}{3}}\, \lambda_1^{e_1} \lambda^{e_2}$ is a cube in $R$.

First suppose that $e_1\neq e_2$.    Let $v_{\lambda_i} \colon R^\times \twoheadrightarrow \ZZ$ be the valuation for the prime $\lambda_i$ and let $v_\ell \colon \QQ^\times \twoheadrightarrow \ZZ$ be the valuation for $\ell$.  We have
\[
v_{\lambda_i}(\beta) = e_i + \tfrac{2(\ell-1)(e_1+e_2)}{3} v_{\lambda_i}(d)= e_i + \tfrac{2(\ell-1)(e_1+e_2)}{3} v_{\ell}(d).
\]
We have $e_1\not\equiv e_2 \pmod{3}$ since $e_1\neq e_2$, so $v_{\lambda_i}(\beta) \not\equiv 0 \pmod{3}$ for some $i\in\{1,2\}$.    Therefore, $\beta\in R$ is not a cube. 

Now suppose that $e_1 = e_2$.  We may assume that $e_1=e_2= 1$ since $\varphi_{(2,2)}$ is the square of $\varphi_{(1,1)}$ and hence have the same kernel.   So $\beta=d^{4(\ell-1)/3}\ell$.   Since $\beta$ is a rational integer, it is a cube in $\ZZ$ if and only if it is a cube in $R$.  If $\ell = 1 + 9m$, then $\beta=  (d^{4m})^3 \ell$ is not  a cube.  If $\ell = 4 + 9m$, then $\beta=  (d^{4m+1})^3 \cdot d\ell$ which is a cube if and only if $d=\ell^2$ (recall that $d$ is positive and cube-free).  If $\ell = 7 + 9m$, then $\beta=  (d^{4m+2})^3 \cdot d^2\ell$ which is a cube if and only if $d=\ell$.\\

Finally, suppose we are in the case where $\rho_{E_d,\ell}(\Gal_k)$ is an index $3$ subgroup of $C_s(\ell)$.   There are $4$ index $3$ subgroups of $C_s(\ell)$.  Two of the groups consist of the matrices $A:=\left(\begin{smallmatrix}a & 0 \\0 & b \end{smallmatrix}\right)$ for which $a$ is a cube (or $b$ is a cube); these groups cannot equal $\rho_{E_d,\ell}(\Gal_k)$ since it would correspond to the case where $e_1=0$ or $e_2=0$ (and hence $e_1\neq e_2$).    Another index $3$ subgroup of $C_s(\ell)$ is the subgroup of matrices whose determinant is a cube; this is impossible since $\det(\rho_{E_d,\ell}(\Gal_k))=\FF_\ell^\times$.   Therefore, the only possibility for the image of $\rho_{E_d,\ell}$ is the group of $A$ with $a/b$ a cube.
\end{proof}

We now complete the proof of the proposition for the curve $E_d/\QQ$.    From Lemmas~\ref{L:new 1 or 3}, \ref{L:j=0 case 1} and \ref{L:j=0 case 2}, we deduce that $\rho_{E_d}(\Gal_\QQ)$ has index $1$ or $3$ in $N(\ell)$, with index $3$ occurring if and only if one of the following hold:
\begin{itemize}
\item $\ell\equiv 2\pmod{9}$ and $d=\ell$,
\item $\ell\equiv 5\pmod{9}$ and $d=\ell^2$,
\item $\ell\equiv 4 \pmod{9}$ and $d=\ell^2$,
\item $\ell\equiv 7\pmod{9}$ and $d=\ell$.
\end{itemize}
Set $M:=\rho_{E_d,\ell}(\Gal_k)$; we may assume that it is the index $3$ subgroup of $C(\ell)$ from Lemma~\ref{L:j=0 case 1} or \ref{L:j=0 case 2}.   The group $M$ is normal in $N(\ell)$.   We have $[N(\ell):M]=6$ and $\det(M)=\FF_\ell^\times$, so $N(\ell)/M$ is non-abelian by Lemma~\ref{L:N group theory}(\ref{L:N group theory i}). So $N(\ell)/M$ is isomorphic to $\mathfrak{S}_3$ and hence, up to conjugation, $N(\ell)$ has a unique index $3$ subgroup $G'$ satisfying $G'\subseteq M$.     Therefore, $G'$ is conjugate in $\GL_2(\FF_\ell)$ to both $\rho_{E_d,\ell}(\Gal_\QQ)$ and the group $G$ from part (\ref{P:j=0 situation iii}) or (\ref{P:j=0 situation iv}) of Lemma~\ref{P:j=0 situation}.     This finishes the proof of Proposition~\ref{P:j=0 situation} for the curve $E_d/\QQ$ and $\ell>3$.
\\

Finally suppose that $E/\QQ$ is any elliptic curve with $j$-invariant $0$; it is defined by a Weierstrass equation $y^2=x^3+dm^3$ for some integer $m\neq 0$ and cube-free integer $d$.  It suffices to show that $\rho_{E,\ell}(\Gal_\QQ)$ is conjugate to $\rho_{E_d,\ell}(\Gal_\QQ)$ in $\GL_2(\FF_\ell)$.  The curves $E$ and $E_d$ are quadratic twists, so $\pm \rho_{E,\ell}(\Gal_\QQ)$ is conjugate to $\pm \rho_{E_d,\ell}(\Gal_\QQ)$.    The general case of Proposition~\ref{P:j=0 situation} is thus a consequence of the following lemma.

 \begin{lemma}
There are no proper subgroups $\pm \rho_{E_d,\ell}(\Gal_\QQ)$ has no proper subgroups $H$ such that $\pm H= \pm \rho_{E_d,\ell}(\Gal_\QQ)$. 
 \end{lemma}
 \begin{proof}
If $\pm \rho_{E_d,\ell}(\Gal_\QQ)$ is conjugate to $N(\ell)$, then the lemma follows immediately from Lemma~\ref{L:N group theory}(\ref{L:N group theory ii}).    From the case of Proposition~\ref{P:j=0 situation} we have already proved (i.e., for the curve $E_d$ and prime $\ell>3$), we need only show that the group $G$ from parts (\ref{P:j=0 situation iii}) and (\ref{P:j=0 situation iv}) of Lemma~\ref{P:j=0 situation} have no proper subgroups $H$ satisfying $\pm H=G$.   Equivalently, we need to show that $-I$ is a commutator of such a subgroup $G$.    With $G$ as in Lemma~\ref{P:j=0 situation}(\ref{P:j=0 situation iii}), this follows from
$\left(\begin{smallmatrix}0 & 1 \\1 & 0 \end{smallmatrix}\right)\left(\begin{smallmatrix}1 & 0 \\0 & -1 \end{smallmatrix}\right)\left(\begin{smallmatrix}0 & 1 \\1 & 0 \end{smallmatrix}\right)^{-1}\left(\begin{smallmatrix}1 & 0 \\0 & -1 \end{smallmatrix}\right)^{-1}=\left(\begin{smallmatrix}-1 & 0 \\0 & -1 \end{smallmatrix}\right)$.  So we may take $G$ as in Lemma~\ref{P:j=0 situation}(\ref{P:j=0 situation iv}).

Fix any $B \in G-C(\ell)$.   As noted in the proof of Lemma~\ref{L:N group theory}, the map $\varphi\colon C(\ell) \to C(\ell)$, $A\mapsto BABA^{-1}$ is a homomorphism whose image is cyclic of order $\ell+1$.    Therefore, $\varphi(G\cap C(\ell))$ is the cyclic subgroup of $C(\ell)$ of order $(\ell+1)/3$.   In particular, $\varphi(G\cap C(\ell))$ contains $-I$ which is the unique element of order $2$ in $C(\ell)$.   Therefore, $-I$ is a commutator of $G$.
 \end{proof}
 
\subsubsection{$\ell=3$ case} \label{L:j eq 0 ell eq 3}

We now consider the prime $\ell=3$ with $E/\QQ$ defined by the elliptic curve $y^2=x^3+d$.   The division polynomial of $E/\QQ$ at $3$ is $3x(x^3+4d)$.    The points of order $3$ in $E(\Qbar)$ are thus $(0, \pm  \sqrt{d} )$ and $( -\sqrt[3]{4d} \omega^e , \pm  \sqrt{-3}  \sqrt{d})$ with $e\in\{0,1,2\}$.   The points $P_1=(0,  \sqrt{d} )$ and $P_2=( -\sqrt[3]{4d}  , \sqrt{-3}  \sqrt{d})$ form a basis of $E[3]$.  With respect to this basis, we have 
\[
\rho_{E,3} = \left(\begin{smallmatrix}\psi_1 & * \\0 & \psi_2 \end{smallmatrix}\right),
\]
with characters $\psi_1,\psi_2 \colon \Gal_\QQ \to \FF_3^\times$.   The quadratic character $\psi_1$ describes the Galois action on $P_1$ and it thus corresponds to the extension $\QQ(\sqrt{d})$ of $\QQ$.   The quadratic character $\psi_1\psi_2=\det\circ \rho_{E,3}\colon \Gal_\QQ \to \FF_{3}^\times$ corresponds to the extension $\QQ(\zeta_3)=\QQ(\sqrt{-3})$ of $\QQ$.   Therefore, 
\[
(\psi_1\times \psi_2)(\Gal_\QQ) = \begin{cases}
\{1\}\times \FF_3^\times & \text{ if $d$ is a square}, \\
\FF_3^\times \times\{1\} & \text{ if $-3d$ is a square}, \\
\FF_3^\times \times \FF_3^\times & \text{ otherwise}.
\end{cases}
\]
To compute the image of $\rho_{E,3}(\Gal_\QQ)$ it remains to determine when its cardinality is divisible by $3$ or not.  From $P_1$ and $P_2$, it is clear that $\rho_{E,3}(\Gal_\QQ)$ is divisible by $3$ if and only if $4d$ is not a cube.

\section{Proof of Proposition~\ref{P:big inertia last}} \label{S:big inertia last}

By Theorem~\ref{T:Mazur-Serre-BPR}, we may assume that $\rho_{E,\ell}(\Gal_\QQ)$ is a subgroup of $N_{ns}(\ell)$.    Let $I_\ell$ be an inertia subgroup of $\Gal_\QQ$ for the prime $\ell$.   We will show that $\rho_{E,\ell}$ has large image by showing that the group $\rho_{E,\ell}(I_\ell)$ is large.    The cardinality of $\rho_{E,\ell}(I_\ell)$ is not divisible by $\ell$ since it is a subgroup of $N_{ns}(\ell)$.  The group $\rho_{E,\ell}(I_\ell)$ is thus cyclic since  the \emph{tame inertia group} at $\ell$ is pro-cyclic, cf.~\cite{MR0387283}*{\S1.3}.
\\

 Let $v_\ell$ be the $\ell$-adic valuation on $\QQ_\ell$ normalized so that $v_\ell(\ell)=1$.  Let $\QQ_\ell^\un$ be the maximal unramified extension of $\QQ_\ell$ in a fixed algebraic closed field $\Qbar_\ell$.   An embedding $\Qbar\hookrightarrow \Qbar_\ell$ allows us to identify $I_\ell$ with the subgroup $\Gal(\Qbar_\ell/\QQ_\ell^{\un})$ of $\Gal_{\QQ_\ell}:=\Gal(\Qbar_\ell/\QQ_\ell)$.  Let $\Delta_E$ be the minimal discriminant of $E/\QQ$.   \\

\noindent  $\bullet$ First suppose that $v_\ell(j_E) \geq 0$ and that $v_\ell(\Delta_E)$ is not congruent to $2$ and $10$ modulo $12$.

Let $L$ be the smallest extension of $\QQ_\ell^\un$ for which $E$ base extended to $L$ has good reduction.  Define $e=[L:\QQ_\ell^\un]$.   There is thus a finite extension $K/\QQ_\ell$ such that $E$ base extended to $K$ has good reduction and that $v_\ell(K^\times)= e^{-1} \ZZ$, where $v_\ell$ is the valuation on $K$ that extends $v_\ell$.     From \cite{MR0387283}*{\S5.6}, we find that $e\in \{1,2,3,4\}$; this uses our assumption on $v_\ell(\Delta_E)$.  

Let $\calI$ be the inertia subgroup of $\Gal_K:=\Gal(\Qbar_\ell/K)$; it is a subgroup of $I_\ell$.  The action of $\calI$ on $E[\ell]$ is semi-simple since the cardinality of $\rho_{E,\ell}(\calI)$ is relatively prime to $\ell$ (the group $N_{ns}(\ell)$ has this property).   Let $\theta_1\colon \calI \twoheadrightarrow \FF_\ell^\times$ and $\theta_2 \colon \calI \twoheadrightarrow \FF_{\ell^2}^\times$ be \emph{fundamental characters} of level $1$ and $2$, respectively, cf.~\cite{MR0387283}*{\S1.7}.    

\begin{lemma} \label{L:inertia irred}
The representation $\rho_{E,\ell}|_{\calI} \colon \calI \to \GL_2(\FF_\ell)$ is irreducible. 
\end{lemma}
\begin{proof} 
Suppose that $\rho_{E,\ell}|_{\calI}$ is reducible.  The representation $\rho_{E,\ell}|_{\calI}$ is given by a pair of characters $\theta_1^{e_1}$ and $\theta_1^{e_2}$ with $0\leq e_1 \leq e_2 <  \ell-1$.    From Proposition~11 of \cite{MR0387283}, we can take $e_1=0$ and $e_2=e$.   The image of $\rho_{E,\ell}(\calI)$ in $\PGL_2(\FF_\ell)$ is thus isomorphic to $\theta_1^{e}(\calI)$ and hence is cyclic of order $(\ell-1)/\gcd(\ell-1,e)$.  

The matrix $A^2$ is scalar for all $A\in N_{ns}(\ell)-C_{ns}(\ell)$.   Therefore, the order of every element in the image of $N_{ns}(\ell)\to\PGL_2(\FF_\ell)$ divides $\ell+1$.  Since $\gcd(\ell+1,\ell-1)=2$, we deduce that $(\ell-1)/\gcd(\ell-1,e)$ equals $1$ or $2$.      This is a contradiction since $\ell\geq 17$.
\end{proof}

Scalar multiplication and a choice of $\FF_\ell$-basis for $\FF_{\ell^2}$ allows us to identify $\FF_{\ell^2}^\times$ with a subgroup of $\Aut_{\FF_\ell}(\FF_{\ell^2})\cong \GL_2(\FF_\ell)$.   Since $\rho_{E,\ell}|_{\calI}$ is irreducible by Lemma~\ref{L:inertia irred}, it is isomorphic to $\theta_2^{e_1+e_2\ell}\colon \calI \to \FF_{\ell^2}^\times \hookrightarrow \GL_2(\FF_\ell) $ for some  $0\leq e_1,e_2 \leq \ell-1$.  As an $\FF_\ell[\calI]$-module, $E[\ell]$ is isomorphic to the dual of the \'etale cohomology group $H^1_{\text{\'et}}(E_{\Kbar}, \FF_\ell)$.   By Th\'eor\`eme~1.2 of \cite{MR2372809}, we may take $0\leq e_1 , e_2 \leq e$ (when $E$ has good reduction at $\ell$, and hence $e=1$, this follows from \cite{MR0387283}*{Prop.~12}).   We have $e_1 \neq e_2$ since otherwise $\theta_2^{e_1+e_2\ell}=(\theta_2^{\ell+1})^{e_1}$ is not irreducible.

Let $g$ be the greatest common divisor of $e_1+e_2\ell$ and $\ell+1$.   We have $(e_1 +e_2 \ell)-e_2(\ell+1) = e_1-e_2 \in \{\pm 1,\pm 2,\pm 3, \pm 4\}$ since $0\leq e \leq 4$, so $g \in \{1,2,3,4\}$.    Therefore, $\rho_{E,\ell}(\calI)$ contains a cyclic group of order $(\ell+1)/g$.

\begin{lemma}\label{L:index 1 or 3 new}
The group $\rho_{E,\ell}(I_\ell)$ is a subgroup of $C_{ns}(\ell)$ with index $1$ or $3$.
\end{lemma}
\begin{proof}
Set $H:=\rho_{E,\ell}(I_\ell)$; it is cyclic.  We claim that $H$ is a subgroup of $C_{ns}(\ell)$.  Suppose not, then the order of $H$ divides $2(\ell-1)$ since $A^2$ is a scalar matrix for any $A\in N_{ns}(\ell)-C_{ns}(\ell)$.  Therefore, $(\ell+1)/g$ divides $2(\ell-1)$ since $\rho_{E,\ell}(\calI)\subseteq H$ contains an element of order $(\ell+1)/g$.  Since $\gcd(\ell+1,\ell-1)=2$, we deduce that $(\ell+1)/g$ divides $4$.   This is impossible since $\ell\geq 17$ and $g\leq 4$.

It remains to bound the index of $H$ in $C_{ns}(\ell)$.  We have $\det(H)=\FF_\ell^\times$ since $\det\circ \rho_{E,\ell}$ describes the Galois action on the $\ell$-th roots of unity.  Therefore, the group $H$ is cyclic and its order is divisible by $\ell-1$ and $(\ell+1)/g$.   Since $\gcd(\ell+1,\ell-1)=2$, we deduce that the order of $H$ is divisible by $(\ell-1)(\ell+1)/(2g)$.  Therefore, the index $b:=[C_{ns}(\ell) : H]$ divides $2g$.

Suppose $b$ is even.  Since $C_{ns}(\ell)$ is cyclic, the group $H$ must be contained in $\{A \in C_{ns}(\ell) : \det(A) \in (\FF_\ell^\times)^2\}$; this is the unique index $2$ subgroup of $C_{ns}(\ell)$.   However, this is impossible since $\det(H)=\FF_\ell^\times$.  So $b$ is odd and divides $2g \in \{2,4,6,8\}$.   Therefore, $b$ is $1$ or $3$.
\end{proof}

Now define $H:= \rho_{E,\ell}(\Gal_\QQ) \cap C_{ns}(\ell)$.     We have $\rho_{E,\ell}(\Gal_\QQ) \not \subseteq C_{ns}(\ell)$ since $C_{ns}(\ell)$ is not applicable; it does not contain an element with trace $0$ and determinant $-1$.     So if $H=C_{ns}(\ell)$, then $\rho_{E,\ell}(\Gal_\QQ) = N_{ns}(\ell)$.    

We are thus left to consider the case where $H$ is the (unique) index $3$ subgroup of $C_{ns}(\ell)$.     The group $H$ is a normal subgroup of $N_{ns}(\ell)$ of index $6$.  

\begin{lemma}
We have $\ell \equiv 2 \pmod{3}$ and the quotient group $N_{ns}(\ell)/H$ is isomorphic to $\mathfrak{S}_3$.
\end{lemma}
\begin{proof}
If $\ell \equiv 1 \pmod{3}$, then $\det(H)=(\FF_\ell^\times)^3 \subsetneq \FF_\ell^\times$.  This is impossible since $\det(\rho_{E,\ell}(\Gal_\QQ))=\FF_\ell^\times$ and $[\rho_{E,\ell}(\Gal_\QQ):H]=2$.  Therefore, $\ell \equiv 2 \pmod{3}$.  One can now verify that $N_{ns}(\ell)$ quotiented out by the scalar matrices is isomorphic to a dihedral group.   It is then easy to check that $N_{ns}(\ell)/H$ is the dihedral group of order $2\cdot 3$; it is thus isomorphic to $\mathfrak{S}_3$.  
\end{proof}

The index $3$ subgroups of $\mathfrak{S}_3$ are all conjugate so, up to conjugacy, $G$ (as in the statement of Proposition~\ref{P:big inertia last}) is the unique index $3$ subgroup of $N_{ns}(\ell)$ that contains $H$.     Therefore, $\rho_{E,\ell}(\Gal_\QQ)$ and $G$ are conjugate subgroups.\\

\noindent  $\bullet$ Suppose that $v_\ell(j_E) \geq 0$. 

By twisting $E/\QQ$ by $1$ or $\ell$, we obtain an elliptic curve $E'/\QQ$ with $v_\ell(\Delta_{E'})$ not congruent to $2$ and $10$ modulo $12$.   The group $\pm \rho_{E,\ell}(\Gal_\QQ)$ is conjugate to $\pm \rho_{E',\ell}(\Gal_\QQ)$.   The previous case applies and shows that $\pm \rho_{E,\ell}(\Gal_\QQ)$ is conjugate to $\pm G =G$ or $\pm N_{ns}(\ell) = N_{ns}(\ell)$ in $\GL_2(\FF_\ell)$.   

It remains to show that $\pm \rho_{E,\ell}(\Gal_\QQ)=\rho_{E,\ell}(\Gal_\QQ)$; if not then there is an index $2$ subgroup $H$ of $G$ or $N_{ns}(\ell)$ such that $-I \notin H$.    The group $H\cap C_{ns}(\ell)$ is then an index $2$ or $6$ subgroup of $C_{ns}(\ell)$ that does not contain $-I$.  However, the cardinality of $H\cap C_{ns}(\ell)$ is even, so it contains an element of order $2$ which most be $-I$.    
\\

\noindent  $\bullet$ Finally suppose that $v_\ell(j_E) < 0$.
  
There exists an element $q\in \QQ_\ell$ with $v_\ell(q)=-v_\ell(j_E)>0$ such that
\[
j_E = (1+240{\sum}_{n\geq 1} n^3 q^n/(1-q^n) )^3/(q{\prod}_{n\geq 1} (1-q^n)^{24}).
\]
Let $\calE/\QQ_\ell$ be the \defi{Tate curve} associated to $q$, cf.~\cite{MR1312368}*{V\S3}; it is an elliptic curve with $j$-invariant $j_E$ and the group $\calE(\Qbar_\ell)$ is isomorphic to $\Qbar_\ell^\times/ \ang{q}$ as a $\Gal_{\QQ_\ell}$-module. In particular, the $\ell$-torsion subgroup $\calE[\ell]$ is isomorphic as an $\FF_\ell[\Gal_{\QQ_\ell}]$-module to the subgroup of $\Qbar_\ell^\times/ \ang{q}$ generated by an $\ell$-th root of unity $\zeta$ and a chosen $\ell$-th root $q^{1/\ell}$ of $q$.   Let $\alpha\colon \Gal_{\QQ_\ell}\to \FF_\ell^\times$ and $\beta\colon \Gal_{\QQ_\ell} \to \FF_\ell$ be the maps defined so that  $\sigma(\zeta)=\zeta^{\alpha(\sigma)}$ and $\sigma(q^{1/\ell})=\zeta^{\beta(\sigma)} q^{1/\ell}$.  So with respect to the basis $\{\zeta, q^{1/\ell}\}$ for $\calE[\ell]$, we have  $\rho_{\calE,\ell}(\sigma)= \left(\begin{smallmatrix}\alpha(\sigma) & \beta(\sigma) \\0 & 1 \end{smallmatrix}\right)$  for $\sigma\in \Gal_{\QQ_\ell}$.      The curves $E$ and $\calE$ are quadratic twists of each other over $\QQ_\ell$ (the curve $E$ is non-CM since its $j$-invariant is not an integer).   So there is a character $\chi\colon \Gal_{\QQ_\ell}\to \{\pm1\}$ such that, after an appropriate choice of basis for $E[\ell]$, we have
\[
\rho_{E,\ell}(\sigma) = \chi(\sigma)\left(\begin{smallmatrix} \alpha(\sigma) & \beta(\sigma) \\0 & 1\end{smallmatrix}\right)
\]
for all $\sigma\in \Gal_{\QQ_\ell}$.  Since $\alpha$ is surjective, we find that the image of $\rho_{E,\ell}(\Gal_\QQ)$ in $\PGL_2(\FF_\ell)$ contains a cyclic group of order $\ell-1$.   However, the image of $N_{ns}(\ell)$ in $\PGL_2(\FF_\ell)$ has order $2(\ell+1)$.    Since $\rho_{E,\ell}(\Gal_\QQ) \subseteq N_{ns}(\ell)$, we find that $\ell-1$ divides $2(\ell+1)$; this is impossible since $\gcd(\ell-1,\ell+1)=2$ and $\ell\geq 17$.

%\bibliographystyle{plain}
%\bibliography{/Users/zywina/Documents/papers/bib/master}

% \bib, bibdiv, biblist are defined by the amsrefs package.

\end{document}